\documentclass{amsart}

\usepackage{amssymb} \usepackage{amsfonts} \usepackage{amsmath}
\usepackage{amsthm} \usepackage{epsfig} 
\usepackage{color}
\usepackage{pinlabel}
\usepackage{amscd}
\usepackage[all]{xy}
\usepackage{pinlabel}
\usepackage{tcolorbox}
\usepackage{seqsplit}
\usepackage{accents}

\usepackage{tikz}
\usepackage{xcolor}

\newtheorem{lemma}{Lemma}%[section]
\newtheorem{teo}[lemma]{Theorem}
\newtheorem{prop}[lemma]{Proposition}
\newtheorem{cor}[lemma]{Corollary}

\theoremstyle{definition}
\newtheorem{defn}[lemma]{Definition}

\newtheorem{quest}[lemma]{Question}

\theoremstyle{remark}

\newcommand{\matr} [4] {\big({\tiny\begin{array}{@{}c@{\ }c@{}} #1 & #2 \\ #3 & #4 \\ \end{array}} \big)}

\newcommand{\Out}{{\rm Out}}

\newcommand{\CAT}{{\rm CAT}}

\newcommand{\HW} {\ensuremath {{\rm HW}}}

\newcommand{\matR} {\ensuremath {\mathbb{R}}}
\newcommand{\matQ} {\ensuremath {\mathbb{Q}}}
\newcommand{\matZ} {\ensuremath {\mathbb{Z}}}
\newcommand{\matC} {\ensuremath {\mathbb{C}}}

\newcommand{\matH} {\ensuremath {\mathbb{H}}}

\newcommand{\matRP} {\ensuremath {\mathbb{RP}}}

\newcommand{\calF} {\ensuremath {\mathcal{F}}}

\newcommand{\T}{\mathsf{T}}
\newcommand{\Orb}{\mathsf{O}}

\newcommand{\id} {\ensuremath {{\rm id}}}

\newcommand{\Isom} {\ensuremath {{\rm Isom}}}

\author{Bruno Martelli}

\title{A 4-dimensional pseudo-Anosov homeomorphism}

\begin{document}

\begin{abstract}
We know from previous work with Italiano and Migliorini that there exists some hyperbolic 5-manifold that fibers over the circle. Here we build one example where the monodromy is a \emph{pseudo-Anosov homeomorphism} of the 4-dimensional fiber, in a way that is surprisingly similar to the familiar and beautiful two-dimensional picture of Nielsen and Thurston for surfaces.
This fact has various consequences:

\begin{enumerate}
\item There is a compact smooth 4-manifold $M$ such that no non-trivial class in $H_2(M)$ is represented by immersed tori, and infinitely many classes are represented by smoothly embedded genus two surfaces.
\item There is a compact locally $\CAT(0)$ space $Y$ such that $\pi_1(Y)$ is not hyperbolic and does not contain $\matZ \times \matZ$.
\end{enumerate}
The latter answers a question of Gromov, known as the Closing Flat Problem.

\end{abstract}

\maketitle

\section*{Introduction}
We say that a compact smooth manifold, possibly with boundary, is \emph{hyperbolic} if its interior admits a complete finite-volume hyperbolic metric. It is a consequence of the Margulis Lemma that the boundary of a hyperbolic $n$-manifold is either empty or a union of finitely many $(n-1)$-manifolds that admit some flat structure, and are therefore finitely covered by the $(n-1)$-torus by Bieberbach's Theorem. 

We know from previous work with Italiano and Migliorini \cite{IMM} that there exists some hyperbolic 5-manifold $N^5$ that fibers over the circle, that is that forms the total space of a smooth fibration $N^5 \to S^1$. The fiber is a compact 4-manifold $F$, and by dragging the fiber along the base circle we obtain a self-homeomorphism $\varphi\colon F \to F$, called the \emph{monodromy}, that is well defined only up to isotopy. It is natural to ask whether the monodromy has some preferred representative.

In this paper we construct a fibering hyperbolic 5-manifold $N^5 \to S^1$ whose monodromy can be represented by an explicit \emph{pseudo-Anosov homeomorphism} $\varphi \colon F \to F$, in a way that is surprisingly similar to the familiar and beautiful two-dimensional picture due to Nielsen and Thurston \cite{Th-s}.
This fact has various notable consequences: we mention here Theorems \ref{homology:teo} and \ref{closing:teo}, that are of quite different nature, despite being originated from the same source. Homology groups in this paper are with integer coefficients, if not otherwise mentioned.

\begin{teo} \label{homology:teo}
There is a compact smooth orientable 4-manifold $M$ such that no non-trivial class in $H_2(M)$ can be represented by immersed tori, and infinitely many classes are represented by smoothly embedded genus two surfaces.
\end{teo}

To the best of our knowledge, no prior example of such a manifold $M$ is known, in any dimension (in dimension 3 it cannot exist because of Thurston's norm theory \cite{Th-n}). The example $M$ that we build here is the fiber $F$ and has dimension 4. A closed example can be constructed by doubling $F$ along its boundary.

\begin{teo} \label{closing:teo}
There is a compact locally $\CAT(0)$ space $Y$ such that $\pi_1(Y)$ is not Gromov hyperbolic and does not contain $\matZ \times \matZ$.
\end{teo}

We recall \cite[Theorem III.H.I.5]{BH} that the fundamental group $\pi_1(Y)$ of a compact locally $\CAT(0)$ space $Y$ is Gromov hyperbolic if and only if the universal cover $\tilde Y$ does not contain any (isometrically embedded) flat plane. So in particular by setting $\Gamma = \pi_1(Y)$ we get 

\begin{cor}
There exist a $\CAT(0)$ space $\tilde Y$ that contains flat planes and a cocompact group $\Gamma$ of isometries acting properly on $\tilde Y$ that does not contain $\matZ \times \matZ$.
\end{cor} 

This answers a well-known question of Gromov \cite[6.${\rm B}_3$]{G2}, sometimes called the \emph{Flat Closing problem}, see Sageev and Wise \cite{SW}. The space $Y$ that we construct here is a finite covering of the fiber $F$, with all its boundary components shrinked to points, equipped with a locally $\CAT(0)$ metric that arises from the pseudo-Anosov homeomorphism $\varphi \colon F \to F$. 

Theorems \ref{homology:teo} and \ref{closing:teo} are both derived from the 4-dimensional pseudo-Anosov map $\varphi\colon F \to F$ built here. Before describing it, we briefly introduce the context.

\subsection*{Fibering hyperbolic 3-manifolds}
We recall a celebrated theorem of Thurston:

\begin{teo} \label{Th:teo}
A compact 3-manifold $M^3$ that fibers over the circle is hyperbolic if and only if the fiber surface $\Sigma$ has $\chi(\Sigma) <0$ and the monodromy can be represented by a pseudo-Anosov homeomorphism $\varphi \colon \Sigma \to \Sigma$.
\end{teo}

A \emph{pseudo-Anosov homeomorphism} $\varphi \colon \Sigma \to \Sigma$ is a package that consists of the following objects. Let $\bar \Sigma$ be the closed surface obtained by shrinking each boundary component of $\Sigma$ to a \emph{point at infinity}. The surface $\bar \Sigma$ is equipped with: 
\begin{enumerate}
\item A \emph{flat cone structure} with finitely many singular points with cone angles $k\pi$, where $k\geq 1, k \neq 2$, and only the points at infinity may have $k=1$;
\item Two orthogonal \emph{geodesic foliations} $\calF^s$ and $\calF^u$, called \emph{stable} and \emph{unstable};
\item A homeomorphism $\varphi \colon \bar \Sigma \to \bar \Sigma$
that preserves each foliation, stretches the leaves of $\calF^u$ via a factor $\lambda > 1$ and contracts those of $\calF^s$ by $1/\lambda$.
\end{enumerate}

A \emph{geodesic foliation} of $\bar \Sigma$ is a geodesic foliation of $\bar \Sigma$ minus its singular points, which looks as in Figure \ref{foliation:fig} near each singular point. The number $\lambda > 1$ is the \emph{stretch factor} of $\varphi$, and it turns out to be an algebraic integer. The homeomorphism $\varphi$ is defined on $\bar \Sigma$, and to get one on $\Sigma$ it suffices to blow up the points at infinity as explained in Farb -- Looijenga \cite[Section 2.3]{FL}.

\begin{figure}
\centering
\includegraphics[width=7 cm]{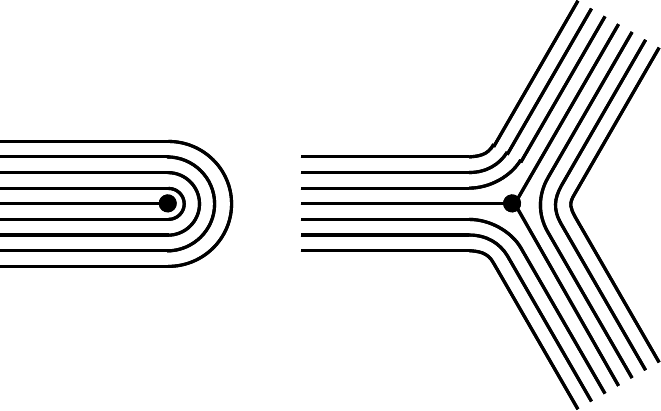}
\caption{Singular points with cone angles $\pi$ and $3\pi$ in a foliated flat cone surface.}\label{foliation:fig}
\end{figure}

\subsection*{Higher dimension}
Since it has been recently shown that fibering hyperbolic manifolds exist also in higher dimension $n>3$, it is natural to ask the following:

\begin{quest} \label{far:quest}
Does a version of Theorem \ref{Th:teo} hold in general dimension $n$?
\end{quest} 

At the time of writing, there is only one commensurability class of hyperbolic manifolds of dimension $n>3$ that is known to contain manifolds that fiber over the circle: these were built with Italiano and Migliorini in \cite{IMM} using Bestvina -- Brady theory \cite{BB} via the fundamental contribution of Jankiewicz -- Norin -- Wise \cite{JNW}, and the initial goal of this project was to understand their monodromies. Bestvina -- Brady theory is very efficient to construct fibrations, but it is unfortunately not very well suited to produce a satisfactory picture of the monodromy. After working hard to get such a picture, we finally discovered one that it is very similar to a pseudo-Anosov homeomorphism of a surface $\Sigma$.

\subsection*{A four-dimensional pseudo-Anosov homeomorphism}
We propose in this paper a natural higher dimensional generalization of the notion of pseudo-Anosov homeomorphism for surfaces. Let $M$ be an even dimensional compact manifold, possibly with boundary. Let $\bar M$ be obtained from $M$ by shrinking each boundary component to a \emph{point at infinity}. A \emph{pseudo-Anosov homeomorphism} $\varphi \colon M \to M$ is a package that consists of:

\begin{enumerate}
\item A \emph{flat cone structure} on $\bar M$ with only even-dimensional singular strata, that is locally CAT(0) everywhere except possibly at the points at infinity; 
\item Two orthogonal \emph{geodesic foliations} $\calF^s$ and $\calF^u$, called \emph{stable} and \emph{unstable};
\item A homeomorphism $\varphi \colon \bar M \to \bar M$
that preserves each foliation, stretches the leaves of $\calF^u$ via a factor $\lambda > 1$ and contracts those of $\calF^s$ by $1/\lambda$.
\end{enumerate}

We briefly explain the terminology, referring to Section \ref{pA:subsection} for more details. The notion of \emph{flat cone manifold}, due to Thurston \cite{Th-p}, generalises both that of a flat cone surface and that of a flat locally orientable orbifold. By McMullen \cite{McM} a flat cone manifold is naturally stratified into totally geodesic flat submanifolds of varying dimensions, that we require here to be even. The positive codimensional strata form the \emph{singular set}. Every codimension-two stratum has a \emph{cone angle}.

A \emph{geodesic foliation} $\calF$ of $\bar M$ is a foliation of each $2k$-stratum into totally geodesic flat $k$-submanifolds, such that leaves of different strata match nicely (see Section \ref{pA:subsection}). The geodesic foliations $\calF^u$ and $\calF^s$ intersect transversely and orthogonally in each $2k$-stratum. 
The homeomorphism $\varphi$ preserves the stratification of $\bar M$ and the foliations $\calF^u$ and $\calF^s$. It acts on the union of the $2k$-strata (that is a flat $2k$-manifold) locally as an affine map that stretches homothetically the leaves of $\calF^u$ by a factor $\lambda$ and contracts those of $\calF^s$ by $1/\lambda$. It may permute the $2k$-strata.

The space $\bar M$ is not a topological manifold in general since the link of a point at infinity is a boundary component of $M$. The space $\bar M$ is equipped with a metric structure (more specifically, a flat cone manifold structure), which does not lift to $M$. However, the homeomorphism $\varphi$ on $\bar M$ does lift to $M$ by blowing up the points at infinity as explained by Farb -- Looijenga \cite[Section 2.3]{FL}, see Section \ref{pA:subsection}.

The proposed higher dimensional generalization of a pseudo-Anosov map indeed coincides with the usual one on surfaces. The following is our main result. 

\begin{teo} \label{monodromy:teo}
There is a hyperbolic 5-manifold $N^5$ that fibers over the circle, whose monodromy can be represented by a pseudo-Anosov homeomorphism $\varphi \colon F \to F$ of the fiber $F$. 
\end{teo}

We strongly believe that $N^5$ is the Ratcliffe -- Tschantz hyperbolic 5-manifold \cite{RT5}, and that the fibration is the one built in \cite{IMM}. However, the description provided here is so different from that of \cite{IMM, RT5}, that proving both facts would require a non-trivial amount of extra work that we prefer to avoid for now.

The fiber $F$ is a compact 4-manifold with boundary $\partial F$ consisting of 5 copies of the Hantsche -- Wendt 3-manifold $\HW$, the unique closed flat 3-manifold that is a rational homology sphere \cite{HW}. The boundary $\partial N^5$ consists of two flat 4-manifolds, each fibering over the circle with fiber $\HW$. The monodromy $\varphi$ fixes one boundary component of $F$ and permutes cyclically the other four.

The compact space $\bar F$ is obtained by shrinking the 5 boundary components of $F$ to 5 points at infinity $P_1,\ldots, P_5$, hence it is a manifold everywhere except at the points $P_i$, whose link is homeomorphic to $\HW$. Theorem \ref{monodromy:teo} says that $\bar F$ has a flat cone manifold structure and two orthogonal geodesic foliations $\calF^s$ and $\calF^u$ preserved by $\varphi$, which stretches the leaves of $\calF^u$ by some $\lambda > 1$ and contracts those of $\calF^s$ by $1/\lambda$. We furnish more information on $\bar F$ and $\varphi$.

\begin{teo} \label{more:teo}
The singular set of $\bar F$ is a flat square torus $T = \matR^2/\matZ^2$ that contains the points at infinity
$$P_1 = (0,0), \quad P_2 = \big(\tfrac 45, \tfrac 35 \big),
\quad P_3 = \big(\tfrac 35, \tfrac 15 \big), \quad P_4 = \big( \tfrac 25, \tfrac 45 \big),
\quad P_5 = \big(\tfrac 15, \tfrac 25\big).
$$
It is stratified into the 2-stratum $T\setminus \{P_i\}$ and the 0-strata $P_i$. The cone angle of the 2-stratum is $3\pi$. The pseudo-Anosov homeomorphism $\varphi$ has stretching factor
$$\lambda = \frac{\sqrt 5 +1}2$$
and it acts on $T$ like the Anosov map $\matr 1110$, with orbits $\{P_1\}$ and $\{P_2 , P_3 , P_4 , P_5 \}$.
\end{teo}

The two orbits $\{P_1\}$ and $\{P_2 , P_3 , P_4 , P_5 \}$ correspond to the two boundary components of $N^5$. The foliations $\calF^s$ and $\calF^u$ consist of orthogonal lines in $T$ and orthogonal flat surfaces in $\bar F \setminus T$, that match nicely: near any point $p\in T \setminus \{P_i\}$ each $\calF^s$, $\calF^u$ is like the foliation of Figure \ref{foliation:fig}-(right) multiplied with a trivial foliation of lines in $\matR^2$. The foliations near a point at infinity $P_i$ are a bit more complicated.
 
\subsection*{The construction of $F$ and $\varphi$}
Quite surprisingly, both $F$ and $\varphi$ can be described easily, via some elementary algebra and topology. Pick the fifth root of unity
$$\zeta = e^\frac{2\pi i}5$$
and consider two non-conjugate Galois embeddings of $\matQ(\zeta)$ in $\matC$. The image of  $\matZ[\zeta]$ in $\matC^2$ along these two embeddings is a lattice $\Lambda$ and the quotient 
$$\T = \matC^2 / \Lambda$$
is a nice flat 4-torus. Every group automorphism of $\matZ[\zeta]$ induces a linear automorphism of $\matC^2$ that preserves $\Lambda$ and hence descends to $\T$. The automorphisms
$$r\colon z \longmapsto \zeta z,\quad
s \colon z \longmapsto -\bar z
$$
of $\matZ[\zeta]$ generate a dihedral group $D_{10}$ of isometries of $\T$, and the quotient 
$$\Orb = \T /D_{10}$$
is a flat 4-orbifold, with total space $S^4$ and singular locus a torus $T\subset S^4$, that is locally flat in $S^4$ except at 5 points $P_1,\ldots, P_5$ whose link in $T$ is the figure-eight knot, see Figures \ref{torus5:fig} and \ref{figure-8:fig}. The link of $P_i$ in $\Orb$ is the orbifold $S^3/D_{10}$, that has total space $S^3$ and cone angle $\pi$ at the figure eight knot. The 2-stratum $T \setminus \{P_i\}$ in $\Orb$ is flat and totally geodesic, and has cone angle $\pi$.

\begin{figure}
\vspace{.4 cm}
\centering
\labellist
\small\hair 2pt
\pinlabel $T$ at 10 140
\pinlabel $P_1$ at 25 0
\pinlabel $P_2$ at 128 0
\pinlabel $P_3$ at 165 100
\pinlabel $P_4$ at 80 165
\pinlabel $P_5$ at -10 100
\endlabellist
\includegraphics[width=4 cm]{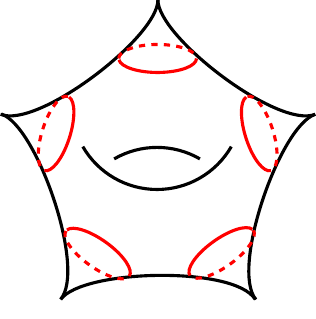}
\caption{The flat orbifold $\Orb = \T/D_{10}$ is $S^4$ with singular set a torus $T\subset S^4$ that is locally flatly embedded except at five points $P_1,\ldots,P_5$ whose link (drawn in red) is the figure-eight knot $K\subset S^3$ shown in Figure \ref{figure-8:fig}.} 
\label{torus5:fig}
\end{figure}

\begin{figure}
\centering
\includegraphics[width=4 cm]{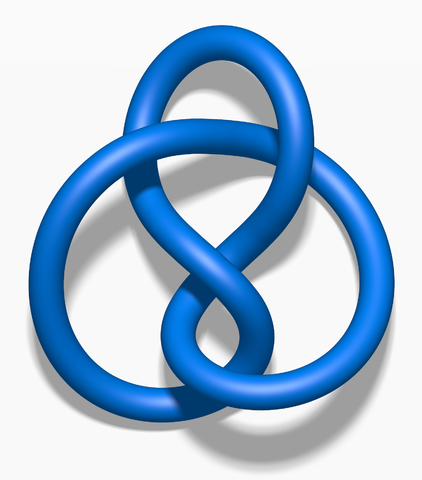}
\caption{The figure eight knot $K$ in $S^3$. The double branched covering over $K$ is the elliptic manifold $L(5,2)$, while the triple branched covering is the flat 3-manifold $\HW$. Both these facts are important here. By quotienting $S^3$ via the standard dihedral group $D_{10}$ of isometries we get an elliptic orbifold, that is $S^3$ with cone angle $\pi$ along $K$, doubly covered by $L(5,2)$. The triple branched covering $\HW$ inherits a spherical cone structure from $S^3/D_{10}$, whose singular locus has cone angle $3\pi$. The 3-manifold $\HW$ has both a flat and a cone spherical structure.}\label{figure-8:fig}
\end{figure}

The space $\bar F$ is the triple branched covering over $S^4$ ramified along $T$. It inherits from the flat orbifold $\Orb$ a cone flat manifold structure. The singular torus $T$ and the points $P_1,\ldots,P_5$ lift from $\Orb$ to $\bar F$, where we denote them with the same symbols. The link of $P_i$ in $\bar F$ is the triple branched covering over the figure-eight knot, that is indeed \cite{Z} the Hantsche -- Wendt 3-manifold $\HW$, as required. The cone angle of $T \setminus \{P_i\}$ is $3\pi$. Since this is not smaller than $2\pi$, the space $\bar F$ is locally $\CAT(0)$ everywhere except at the points at infinity $P_i$.

We now define the monodromy $\varphi \colon \bar F \to \bar F$. The golden ratio
$$\lambda = -\zeta^2 - \zeta^3 = \frac{\sqrt 5+1}2$$
is an invertible element in $\matZ[\zeta]$, hence
$$\varphi(z) = \lambda z$$
is an infinite order group automorphism of $\matZ[\zeta]$ that induces an infinite order automorphisms on $\T$, which first descends to $\Orb$, and then lifts to $\bar F$. This is the pseudo-Anosov homeomorphism $\varphi \colon \bar F \to \bar F$. The automorphism $\varphi$ of $\matC^2$ preserves the coordinate foliations $\matC \times \{p\}$ and $\{p\} \times \matC$, which first descend to $\Orb$ and then lift to $\bar F$. These are the geodesic foliations $\calF^u$ and $\calF^s$ preserved by $\varphi$.

\subsection*{The topology of $F$}
The given description allows us to study the topology of $F$ with reasonable effort. 
We start by stating a theorem that is valid for the fiber $F$ of any hyperbolic manifold $N$ that fibers over the circle, in any dimension.
A properly immersed surface in a manifold $M$ is \emph{essential} if it is $\pi_1$-injective and cannot be homotoped (relative to $\partial $) into $\partial M$. 

\begin{teo} \label{F:teo}
The compact manifold $F$ is aspherical and does not contain any essential properly immersed connected orientable surface $\Sigma$ with $\chi(\Sigma) \geq 0$. Every $\matZ \times \matZ$ subgroup of $\pi_1(F)$ is peripheral. Each boundary component of $F$ is $\pi_1$-injective. The group $\Out(\pi_1(F))$ is infinite.
\end{teo}
\begin{proof}
All the properties listed except the last one are satisfied by the hyperbolic manifold $N$ and then easily inherited by $F$. The last one is proved by noting that no power of $\varphi_*\in \Out(\pi_1(F))$ can preserve any non-peripheral conjugacy class in $\pi_1(F)$, since this would yield a non-peripheral $\matZ \times \matZ < \pi_1(N)$.
\end{proof}

We can roughly summarize this theorem by saying that $F$ looks like a hyperbolic manifold at a first glance, except that $\Out(\pi_1(F))$ is infinite, a striking fact that in fact prevents $F$ from being  hyperbolic when $\dim F > 2$ by Mostow -- Prasad rigidity. 
Given these premises, it is interesting to see how the monodromy $\varphi$ acts on the basic topological invariants of $F$ like the fundamental group and the homology. 

\begin{teo} \label{pi:F:teo}
The fundamental group of $F$ has the following presentation:
$$\pi_1(F) = \langle a_i, b_i \ |\ a_{i+2} = a_{i}a_{i+1}, \ b_{i+2} = b_{i}b_{i+1}, \ a_{i}^{-1}b_{i+1}a_{i+2} = b_i^{-1}a_{i+1}b_{i+2} \rangle$$
where $i=1,\ldots, 6$ is considered modulo 6. 
The automorphism $\varphi_*$ acts as:
\begin{alignat*}5
a_1 & \longmapsto a_3b_5a_4^{-1}b_3^{-1}a_1^{-1}, \qquad a_2 &\ \longmapsto a_3b_3^{-1}a_1^{-1}, \qquad b_1 & \longmapsto a_3^{-1}, \qquad b_2 &\ \longmapsto a_3a_1^{-1},
\\
a_3 & \longmapsto  a_1b_3a_2^{-1}b_1^{-1}a_5^{-1}, \qquad a_4 &\ \longmapsto a_1b_1^{-1}a_5^{-1}, \qquad  
b_3 & \longmapsto  a_1^{-1}, 
\qquad b_4 & \longmapsto  a_1a_5^{-1}, \\
a_5 & \longmapsto a_5b_1a_6^{-1}b_5^{-1}a_3^{-1}, \qquad a_6 & \longmapsto a_5b_5^{-1}a_3^{-1}, \qquad 
b_5 & \longmapsto a_5^{-1}, \qquad b_6 & \longmapsto  a_5a_3^{-1}. 
\end{alignat*}
\end{teo}

The automorphism $\varphi_*$ is determined only up to conjugation, and as we already noted $\varphi_*\in \Out(\pi_1(F))$ has infinite order. Recall that the fundamental group of the Hantsche -- Wendt manifold $\HW$ is the \emph{Fibonacci group} 
$$\pi_1(\HW) = \langle a_i\ |\ a_{i+2}=a_ia_{i+1}\rangle$$ 
with $i=1,\ldots, 6$. The elements $a_i$ and $b_i$ in $\pi_1(F)$ generate two peripheral subgroups of $\pi_1(F)$. The group $\pi_1(F)$ is in fact obtained from $\pi_1(\HW)*\pi_1(\HW)$ with generators $a_i, b_i$ by adding the 6 relations 
$$a_{i}^{-1}b_{i+1}a_{i+2} = b_i^{-1}a_{i+1}b_{i+2}.$$ 
We may substitute these 6 relations with the following ones:
$$[a_2,b_2] = [a_4,b_4] = [a_6,b_6], \qquad [a_1,b_1] = [a_3,b_3] = [a_5,b_5].$$

\begin{teo} \label{H:F:teo}
The manifold $F$ is orientable, spin, and mirrorable. We have
$$H_1(F) = (\matZ/4\matZ)^4, \quad H_2(F) = \matZ^4, \quad H_3(F) = \matZ^4.$$
In particular $\chi(F) = 1$. The intersection form on $H_2(F)$ is
$$Q = \begin{pmatrix}
2 & -1 & -1 & 1 \\
-1 & 2 & 0 & -1 \\
-1 & 0 & -2 & 1 \\
1 & -1 & 1 & -2
\end{pmatrix}
$$
with respect to some appropriate basis, represented by four embedded oriented surfaces of genus two.
In particular $Q$ is even, it has signature $\sigma = 0$, and $\det Q =16$. No non-trivial class in $H_2(F)$ can be represented by an immersed sphere or torus.

The isomorphism $\varphi_*\colon H_2(F) \to H_2(F)$ acts via left multiplication by the matrix
$$A=\begin{pmatrix}
-1 & -1 & -1 & 1 \\
-2 & 1 & 0 & 1 \\
-1 & 1 & 0 & 1 \\
0 & 1 & 1 & 0
\end{pmatrix}.
$$

The isomorphisms $\varphi_*$ of $H_1(F)$ and $H_3(F)$ have finite order.
\end{teo}

It is interesting to note that $\varphi_*\colon H_2(F) \to H_2(F)$ has infinite order.
The eigenvalues of $A$ are $\pm 1 \pm \sqrt 2$, we have $\det A = 1$, and
$\varphi_*$ of course preserves the bilinear form $Q$, that is $A^TQA = Q$. 
In particular we deduce the following. 

\begin{cor} 
Infinitely many distinct classes in $H_2(F)$ are represented by an embedded genus two surface. The Gromov seminorm on $H_2(F, \matR)$ vanishes.
\end{cor}
\begin{proof}
Both facts follow from the existence of a basis of eigenvectors for $\varphi_*$ with eigenvalues different from $\pm 1$. All the orbits of the action of $\varphi_*$ on $H_2(F)\setminus \{0\}$ are infinite, and the Gromov seminorm of each eigenvector is zero, hence it  vanishes.
\end{proof}

To the best of our knowledge, each point of the following corollary is new, in any dimension. No such examples may exist in dimension 2 and 3.

\begin{cor} \label{H2:cor}
There is a closed aspherical 4-manifold $M$ with non-trivial $H_2(M)$, such that no non-trivial class in $H_2(M)$ is represented by immersed tori, and
\begin{enumerate}
\item Infinitely many classes in $H_2(M)$ are represented by genus two surfaces;
\item The Gromov seminorm on $H_2(M, \matR)$ vanishes.
\end{enumerate}
\end{cor}
\begin{proof}
Let $M$ be the double of $F$ along its boundary. Since $\partial F$ consists of rational homology spheres we have $H_2(M,\matR) = H_2(F,\matR) \oplus H_2(F,\matR) = \matR^4 \oplus \matR^4$. The fact that $F$ has no essential immersed annuli and tori ensures that no class in $H_2(M)$ is represented by immersed tori. 
\end{proof}

This implies Theorem \ref{homology:teo}.
%We recall that when $M$ is closed and hyperbolic, the Gromov seminorm is a honest norm on $H_k(M, \matR)$ for every $k\geq 2$. Therefore only finitely many classes in $H_2(M)$ may be represented by surfaces of bounded genus. On the other hand, if $M$ is simply connected every class in $H_2(M)$ is represented by an immersed sphere, and hence by an immersed torus.

\subsection*{Sketch of the proof of Theorem \ref{monodromy:teo}}
Having introduced $F$ and $\varphi$, to prove Theorem \ref{monodromy:teo} it only remains to check that the mapping torus of $\varphi$, that is the 5-manifold $N^5=F \times [0,1]/\!\sim$ with $(x,1) \sim (\varphi(x),0)$, is hyperbolic. This is done in Section \ref{RT:section}, and we expose here the main ideas.

\begin{figure}
\vspace{.4 cm}
\centering
\labellist
\small\hair 2pt
\pinlabel $f$ at 82 32
\pinlabel $\Delta$ at -10 10
\pinlabel $\Delta'$ at 180 10
\pinlabel $T$ at 290 10
\endlabellist
\includegraphics[width=8 cm]{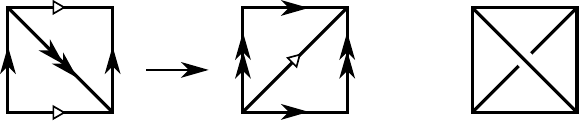}
\caption{The monodromy $f$ of the fibered hyperbolic Gieseking 3-manifold sends one ideal triangulation $\Delta$ of the punctured torus to another ideal triangulation $\Delta'$ that differs from $\Delta$ only by a flip. By juxtaposing the two triangulations we get an ideal triangulation of the Gieseking manifold with one ideal tetrahedron $T$.}\label{Gieseking:fig}
\end{figure}
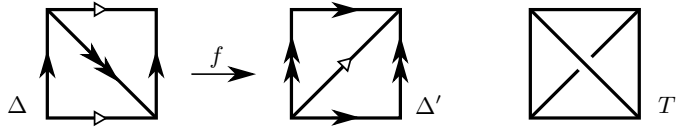

We recall an instructive 3-dimensional example. The \emph{Gieseking manifold} is the non-orientable hyperbolic 3-manifold obtained as the mapping torus of the punctured torus with monodromy $f = \matr 1110$. The simplest way to prove that this manifold is indeed hyperbolic consists in noting that $f$ sends an ideal triangulation $\Delta$ of the punctured torus to another ideal triangulation $\Delta'$ of the punctured torus, which differs from $\Delta$ only by a flip, see Figure \ref{Gieseking:fig}. If we pick an ideal tetrahedron $T$ as in the figure with faces $\Delta \cup \Delta'$ and pair the 4 triangles along $f$ we get an ideal triangulation of the Gieseking manifold with only one tetrahedron $T$. All the edges of the tetrahedron are in fact identified to a single one with valence 6. Therefore if we assign to the ideal tetrahedron the structure of a \emph{regular ideal hyperbolic tetrahedron}, whose dihedral angles are $\pi/3$, we equip the Gieseking manifold with a hyperbolic structure.\footnote{The Gieseking manifold is the smallest cusped hyperbolic 3-manifold by Adams \cite{A}; its orientable double cover is the figure-eight knot complement, that fibers with monodromy $f^2 =\matr 2111$.}

We do here a similar construction to equip the mapping torus of $\varphi \colon F \to F$ with a hyperbolic structure. In Section \ref{tessellation:section} we show that the flat torus $\T$ can also be described elegantly using the \emph{$A_4$ lattice}, and this representation provides a natural $D_{10}$-invariant tessellation of $\T$ into 10 \emph{rectified simplexes} and 10 \emph{simplexes}. This descends to a tessellation of the flat orbifold $\Orb = \T/D_{10}$ into one rectifield simplex $R$ and one simplex $S$. (We will use this tessellation to show that the underlying space of $\Orb$ is $S^4$.) Then we note that, similarly to the square in Figure \ref{Gieseking:fig}, the polytope $R$ admits two distinct but isomorphic triangulations, each with 11 simplexes. This yields two distinct but isomorphic triangulations $\Delta, \Delta'$ for $\Orb$ with 12 simplexes each, such that $\varphi(\Delta) = \Delta'$.

Similarly as in Figure \ref{Gieseking:fig}, we juxtapose the two triangulations $\Delta$ and $\Delta'$ to find a 5-dimensional polytope with 24 facets, which is in fact a \emph{cross-polytope} with some paired facets. It is remarkable that we find here a cross-polytope, that is one of the three 5-dimensional regular polytopes. By lifting this construction to the triple branched covering $\bar F \to \Orb$ we find that $N^5$ has an ideal tessellation into three cross-polytopes, and by direct inspection we find that every condimension-two face in this tessellation has valence 3, that is it is contained in 3 cross-polytopes (counted with multiplicity). Quite luckily, the dihedral angles of the 5-dimensional \emph{regular ideal hyperbolic cross-polytope} $C \subset \matH^5$ are precisely $2\pi/3$, as already noted by Ratcliffe and Tschantz \cite{RT5}. Therefore if we assign this hyperbolic structure to each cross-polytope we obtain a hyperbolic structure on the 5-dimensional mapping torus, much as in the three-dimensional case with the Gieseking manifold. 

\subsection*{Sketch of the proof of Theorem \ref{closing:teo}}
Theorem \ref{monodromy:teo} says that the flat cone manifold $\bar F$ is locally $\CAT(0)$ everywhere, except possibly at the vertices $P_1, \ldots, P_5$. However we can easily check that $\bar F$ is not locally $\CAT(0)$ at $P_i$, because the link of $P_i$ contains some closed geodesics shorter than $2\pi$ and hence it is not $\CAT(1)$. 

We will prove that, after substituting $F$ and $N^5$ with some sufficiently large finite covering $F_*$ and $N_*^5$, the following hold:
\begin{enumerate}
\item The (spherical cone) links of all the points at infinity in $\bar F_*$ do not contain closed geodesics of length $< 2\pi$;
\item There are (flat) disjoint cusp sections of $N_*^5$ that do not contain closed geodesics of length $<2 \pi$.
\end{enumerate}

The space $F_*$ is a fiber of $N_*^5$ with some monodromy $\varphi_*$. The two similar-looking conditions (1) and (2) imply respectively that (i) $\bar F_*$ is locally $\CAT(0)$, and (ii) the mapping torus $\bar N^5_*$ of $\varphi_* \colon \bar F_* \to \bar F_*$ admits a $\CAT(-1)$ metric, via Fujiwara -- Manning \cite{FM}. Here $\bar N^5_*$ is $N_*^5$ with all boundary components shrinked to circles. 

We then prove Theorem \ref{closing:teo} with $Y= \bar F_*$ similarly as in \cite{IMM}. The group $\pi_1(\bar N^5_*)$ is hyperbolic, but its subgroup $\pi_1(\bar F_*)$ is not because $\varphi \in \Out(\pi_1(\bar F_*))$ has infinite order; therefore $\pi_1(\bar N^5_*)$, and hence \emph{a fortiori} $\pi_1(\bar F_*)$, does not contain $\matZ \times \matZ$. 

To prove both (1) and (2) we rely on residual finiteness of $\pi_1(N^5)$, and on the fact that the universal covers of the (spherical cone) link and of the (flat) sections do not contain closed geodesics of length $<2\pi$. While the second fact is obvious, the first is certainly not! To prove it we will need many \emph{ad hoc} estimates and a computer-assisted final proof, see the long Section \ref{closing:section}. 

\subsection*{Comments and open questions}
We describe some related results in the literature, make a few comments, and suggest some possible future lines of research.

\subsubsection*{Four-dimensional self-homeomorphisms}
In the context of complex surfaces and their holomorphic automorphisms, the closer analogue to a pseudo-Anosov homeomorphism is a hyperbolic automorphism on a complex K3 surface. These were studied in particular by Cantat \cite{C} and McMullen \cite{McMK3}. Cantat showed the existence of a unique pair of projectively invariant currents that are analogous to the stable and unstable laminations in Nielsen -- Thurston's theory. McMullen constructed some automorphisms with Siegel discs, that are in particular not ergodic while having positive entropy.

In the real four-dimensional context, Farb -- Looijenga recently defined a notion of pseudo-Anosov diffeomorphism for a K3 surface \cite{FL}. Such a diffeomorphism is Anosov in the complement of a small invariant tubular neighbourhood of a codimension two submanifold, that consists of the 16 spheres with self-intersection $-2$ arising from the Kummer construction, see \cite[Section 2.4]{FL} for more details. The paper \cite{FL} fits into a program of the authors to classify the elements of the mapping class group of a K3 surface into three types, similarly to Nielsen and Thurston: the finite order and reducible types were considered in \cite{FL1} and \cite{FL2}.

Our notion of pseudo-Anosov homeomorphism is somehow stronger than the one of Farb -- Loojienga, since it also involves a cone flat structure, and therefore it could make sense to call our maps \emph{strongly} or \emph{geometrically} pseudo-Anosov in comparison to theirs. 

It is worth noting that a K3 surface, as any simply connected 4-manifold, cannot arise as the fiber of a fibering hyperbolic 5-manifold. We do not know if a complex surface may be the fiber of a fibering hyperbolic 5-manifold: such a surface is probably necessarily of general type, and in that case the monodromy $\varphi$ cannot be homotopic to a biolomorphism for any complex structure, since these have finite order by Andreotti \cite{And}.

\subsubsection*{More examples of pseudo-Anosov maps and fibering hyperbolic manifolds}
The far reaching Question \ref{far:quest} asks whether a version of Thurston's Theorem (that relates hyperbolic 3-manifolds fibering over the circle to pseudo-Anosov homeomorphism of surfaces), may hold in higher dimension. For the time being, we only have one commensurabillity class of fibering hyperbolic 5-manifolds. We can construct some interesting new fibrations by taking a finite cover $\tilde N^5 \to N^5$ with $b_1(\tilde N^5)>1$, lifting the fibration $N^5 \to S^1$ to $\tilde N^5 \to S^1$, and then perturbing it. Is the monodromy still represented by a pseudo-Anosov map after a perturbation? 

Conversely, we could try to modify the relatively simple construction of $\bar F$ presented here to build new pseudo-Anosov maps, for instance by taking Galois embeddings of $\matZ[\zeta]$ for a $(2n+1)$-th root of unity $\zeta$ in $\matC^n$. Can we classify the pseudo-Anosov maps that are constructed, as here, as branched coverings over flat orbifolds? Are their mapping tori hyperbolic? 
Some conditions on the links of the points at infinity is probably required here (they should be flat manifolds, like our $\HW$ here). 
More generally, do they have negative curvature?

\subsubsection*{Second homology group}
Corollary \ref{H2:cor} says that there is a 4-manifold $M$ such that no non-trivial homology class in $H_2(M)$ is represented by immersed tori, while (1) infinitely many ones are represented by embedded genus two surfaces, and (2) the Gromov seminorm on $H_2(M)$ vanishes. As already mentioned, both conclusions are impossible in dimension 3, because the Thurston and Gromov seminorms coincide up to a factor by Gabai \cite[Corollary 6.18]{G}. Note also that such a manifold $M$ cannot be simply connected (every homology class would be represented by immersed tori), and neither it may admit a negative sectional curvature metric, because in this case the Gromov seminorm is a norm by Inoue and Yano \cite{IY}. 

Calegari has described to us a nice construction of a compact aspherical 4-manifold $Z$ where $H_2(Z)$ is not generated by immersed tori, it contains infinitely many classes represented by immersed genus two surfaces, and the Gromov seminorm of $H_2(Z)$ vanishes. 
This example is illuminating because it clearly illustrates a phenomenon that may arise in dimension 4 but is forbidden in dimension 3. 

Here is the construction. Let $X$ be the 2-complex constructed by picking two tori, three oriented closed geodesics $a,b,c$ in each torus representing the classes $(1,0), (0,1)$ and $(-1,-1)$, and connecting each pair of corresponding geodesics with an annulus. Note that $X$ is locally $\CAT(0)$, hence it thickens to a compact aspherical 4-manifold $Z$ with boundary. We have $H_2(X) = \matZ^3$ generated by the two tori and one surface of genus two $\Sigma \subset X$ that contains the three annuli. The classes represented by tori generate a subspace of dimension two. However, the Gromov seminorm on $H_2(X)$ vanishes: the trivial 1-cycle $n(1,0) + n(0,1) + n(-1,-1)$ is the boundary of a map from a pair-of-pants to a torus, and by gluing two such maps (one in each torus) along the annuli we represent $n[\Sigma]$ via a map from a genus two surface to $X$ for each $n$. Hence $\|[\Sigma]\|=0$. We have used implicitly that the stable commutator length of the 1-cycle is zero: this argument has interesting variations, see for instance \cite[Example 3.18]{Cal} and Calegari \cite{scl} for an introduction to the subject.

Another related result is a paper of Berge -- Ghys \cite{BG} where they construct some amenable group $\Gamma$ such that $H_2(\Gamma)$ is not generated by maps from tori to a $K(\Gamma,1)$ (in fact, by maps from surfaces of bounded genus $\leq g$ with $\Gamma$ depending on $g$). The Gromov seminorm on $H_2(\Gamma)$ vanishes since $\Gamma$ is amenable. 

We do not know if the genus two surfaces $\Sigma$ that generate $H_2(F)$ mentioned in Theorem \ref{H:F:teo} are $\pi_1$-injective! We only know that they are incompressible, \emph{i.e.} every simple closed curve in $\Sigma$ injects in $\pi_1(F)$, otherwise by compressing it we would get a homologically non-trivial torus, that is excluded. Incompressibility of an embedded surface does not guarantee $\pi_1$-injectivity in dimension 4: see Cooper -- Manning \cite{CM} and more specifically Calegari \cite[Example 2]{Ca:p}.

%In higher dimension $n>3$, the minimum  and it is thus natural to wonder how Gromov's norm behaves on $H_2(M)$ in dimension $n>3$ where Thurston's norm is not defined anymore.

\subsubsection*{Entropy}
If $f \colon \Sigma \to \Sigma$ is a pseudo-Anosov homeomorphism of a surface $\Sigma$, its entropy $h(f)$, the entropy $h(f_*)$ of its action $f_*$ on $\pi_1(\Sigma)$, the stretching factor $\lambda$, and the spectral radius $\rho$ of its action on the homology ring $H_*(\Sigma)$ are related as:
$$h(f) = h(f_*) = \log \lambda \geq \log \rho.$$

Since in general $h(f) \geq h(f_*)$, we deduce that $f$ has the smallest entropy in its homotopy class, see Fathi and Shub \cite{FS}. On the other hand, if $f$ is a diffeomorphism of a compact manifold $M$ of any dimension, a theorem of Yomdin \cite{Yom} ensures that $h(f) \geq \log \rho$, and a theorem of Gromov \cite{Gr2} states that $h(f) = \log \rho$ when $M$ is K\"ahler and $f$ is holomorphic, and thus also in this case $f$ minimizes the entropy in its homotopy class among smooth diffeomorphisms.

More recently, Farb -- Looijenga \cite{FL} found infinitely many entropy-minimizing diffeomorphisms of the K3 manifold that are not homotopic to any biolohomorphic map for any complex structure; for these diffeomorphisms $f$ we have $h(f) = \log \rho$ as in the biholomorphic case. Moreover they also produced some diffeomorphisms $f$ with $h(f) > \log \rho$, which they also conjecture to be entropy-minimizing. All these diffeomorphisms are of pseudo-Anosov type in the sense of Farb -- Loojienga, see \cite[Section 2.4]{FL}.

It is natural to wonder whether the pseudo-Anosov map $\varphi \colon F \to F$ defined here also minimizes the entropy in its homotopy class. We conjecture that this is true, and moreover that the following equalities should hold:
$$h(\varphi) = \log \lambda^2, \quad h(\varphi_*) = \log \lambda.$$
Note that by Theorem \ref{H:F:teo} the spectral radius $\rho = 1 + \sqrt 2 = 2.41 \ldots $ of the action of $\varphi$ on $H_*(F)$ is slightly smaller than $\lambda^2 = (1+\sqrt 5)^2/4 = 2.61 \ldots$ so our conjecture would be coherent with Yomdin's inequality $h \geq \log\rho$, which is proved only for diffeomorphisms but might reasonably be valid also in this context.

\subsubsection*{A $\CAT(0)$ metric for $F$?}
The interior of $F$ has a locally $\CAT(0)$ incomplete metric, obtained by removing the points at infinity $P_1,\ldots, P_5$ from $\bar F$. We do not know if the interior of $F$ has a complete $\CAT(0)$ metric, and more interestingly we do not know if $F$ itself has a $\CAT(0)$ metric. Is it possible to remove a small open neighborhood of the points at infinity in $\bar F$ to get a $\CAT(0)$ metric on $F$? 

We know that $F$ is aspherical since it is a fiber of the aspherical manifold $N^5$. Is the universal cover of $F$ diffeomorphic to $\matR^4$? Is it at least homeomorphic to it? How do the lifted stable and unstable foliations look like in the universal cover? 

\subsubsection*{Subgroups of hyperbolic groups}
Together with Italiano and Migliorini, we constructed in \cite{IMM} a finite type subgroup $H< G$ of a hyperbolic group $G$ that is not hyperbolic. Here we build one example where the subgroup $H$ is also $\CAT(0)$, that is it acts isometrically, properly, and cocompactly on a $\CAT(0)$ space. 

The subgroup $H<G$ is the first example of a non-hyperbolic subgroup of a hyperbolic group whose Dehn function growth rate is known precisely: being $\CAT(0)$ and not hyperbolic, its Dehn function is quadratic \cite[page 444, Remark 1.7]{BH}. All the known examples of non-hyperbolic subgroups of a hyperbolic group are kernels of some maps onto $\matZ^k$ and therefore they satisfy a polynomial upper bound by works of Gestern -- Short \cite{GS}, Llosa Isenrich \cite{CLI} and Kropholler -- Llosa Isenrich -- Soroko \cite{KLIS}. 

Is there a non-hyperbolic subgroup of a hyperbolic group whose Dehn function grows more than quadratically? Is there a finite type non-hyperbolic subgroup of some hyperbolic group that is neither $\CAT(0)$ nor hyperbolic?

\subsubsection*{More efficient $\CAT(1)$ detection} A long part of this paper, Section \ref{closing:section}, is devoted to the proof that a single explicit piecewise spherical 3-dimensional space is $\CAT(1)$. This is an essential step in proving that the compact 4-dimensional piecewise Euclidean space $Y$ in Theorem \ref{closing:teo} is $\CAT(0)$, and it consists of a long list of hand-made estimates, plus a computer analysis, to show that there are no closed geodesics of length $<2 \pi$. It would be (or at least would have been) very useful to find either a simpler conceptual proof, or a more fully computer assisted one.
There is an algorithm by Edler and McCammond \cite{EM} to check if a compact piecewise-Euclidean space is $\CAT(0)$, that has been implemented by the authors in dimension 3 with a code available on their web page. This is however one dimension less than what was needed here.

\subsection*{Structure of the paper}
In Section \ref{pA:section} we describe more rigorously many of the constructions sketched in this introduction: we build the flat torus $\T$, the quotient flat orbifold $\Orb$, the branched covering $\bar F$, and the automorphism $\varphi$. Then we propose a general definition of pseudo-Anosov homeomorphism and show that the automorphism $\varphi$ of $\bar F$ is pseudo-Anosov. 

In Section \ref{tessellation:section} we furnish another description of $\T$ that uses the \emph{$A_4$ lattice}. From this we derive various tessellations and triangulations for $\T$, $\Orb$, and $\bar F$. In Section \ref{RT:section} we build a hyperbolic structure on the mapping torus $N^5$ of $\varphi \colon F \to F$. This completes the proof of Theorem \ref{monodromy:teo}.

In Section \ref{topological:section} we study the topology of $F$ and $\varphi$, proving in particular Theorems \ref{F:teo}, \ref{pi:F:teo}, and \ref{H:F:teo}, which imply Theorem \ref{homology:teo} as explained above. Finally, we prove Theorem \ref{closing:teo}  in the longer Section \ref{closing:section}. 

\subsection*{Acknowledgements}
We warmly thank Martin Bridson for advising us to study whether the fiber of the fibration constructed in \cite{IMM} may have a locally $\CAT(0)$ metric, Davide Lombardo for illustrating to us the elegant algebraic description of $\T$ and $\varphi$ via Galois embeddings, Alessandro Sisto for illuminating discussions on the subtelties of locally CAT(1) spaces, and Danny Calegari for introducing the example $Z$ mentioned above. We also thank Benson Farb and Claudio Llosa Isenrich for insightful comments on an earlier version of this paper.

\section{The pseudo-Anosov homeomorphism} \label{pA:section}

We define the flat torus $\T$, the flat orbifold $\Orb = \T/D_{10}$, and the triple branched covering $\bar F$ of $\Orb$. Each of these is a flat cone manifold equipped with an infinite-order automorphism $\varphi$. We propose a higher-dimensional generalization of the notion of pseudo-Anosov homeomorphism and show that $\varphi \colon \bar F \to \bar F$ is pseudo-Anosov.

\subsection{A complex torus $\T$ with $D_{10}$ symmetry} \label{T:subsection}
Pick the fifth root of unity
$$\zeta = e^\frac{2\pi i}5$$
and consider the lattice $\Lambda$ in $\matC^2$ generated by the vectors
\begin{equation} \label{basis:eqn}
(1,1), \quad (\zeta, \zeta^2), \quad (\zeta^2, \zeta^4), \quad (\zeta^3, \zeta).
\end{equation}
The Gram matrix of these generators is
\begin{equation*} 
G = \frac 12 
\begin{pmatrix}
4 & -1 & -1 & -1 \\
-1 & 4 & -1 & -1 \\
-1 & -1 & 4 & -1 \\
-1 & -1 & -1 & 4
\end{pmatrix}.
\end{equation*}
It is easy to prove that $(\zeta^4, \zeta^3) \in \Lambda$ and that the lattice $\Lambda$ is invariant under the (anti-)holomorphic isometries
\begin{equation} \label{D10:eqn}
r\colon (z,w) \longmapsto (\zeta z, \zeta^2 w), \quad
s \colon (z,w) \longmapsto (-\bar z, -\bar w),
\end{equation}
which generate a dihedral group $D_{10}$ of order 10. 

The lattice $\Lambda$ is in fact a \emph{subring} of $\matC^2$, constructed by taking two different Galois embeddings of $\matZ[\zeta]$ in $\matC$. That is, the ring $\Lambda$ is isomorphic to $\matZ[\zeta]$ via the morphism $\matZ[\zeta] \to \Lambda, \zeta \mapsto (\zeta,\zeta^2)$. The isometries $r$ and $s$ correspond to the group automorphisms $z\mapsto \zeta z$ and $z \mapsto -\bar z$ of the ring $\matZ[\zeta]$.

We now define the complex torus 
$$\T = \matC^2 / \Lambda$$
equipped with the K\"ahler metric induced from $\matC^2$.
The (anti-)holomorphic isometries $r,s$ descend to $\T$ and generate a dihedral group $D_{10}$ of (anti-)holomorphic isometries of $\T$.

\subsection{Two Lagrangian tori} \label{Lagrangian:subsection}
The orthogonal Lagrangian tori 
 $$T =(i\matR)^2 / (\Lambda \cap (i\matR)^2), \qquad T' = \matR^2 / (\Lambda \cap \matR^2)$$
in $\T$ are preserved by $s$. The lattices $\Lambda \cap (i\matR)^2$ and $\Lambda \cap \matR^2$ are generated by
\begin{equation} \label{bases:eqn}
(\zeta - \zeta^4, \zeta^2 - \zeta^3), \quad (\zeta^2 - \zeta^3, \zeta^4 - \zeta);
\qquad
(1,1), \quad (\zeta + \zeta^4, \zeta^2 + \zeta^3)
\end{equation}
respectively.
The corresponding Gram matrices are 
\begin{equation*} 
G_T = \begin{pmatrix} 5 & 0 \\ 0 & 5 \end{pmatrix}, \qquad G_{T'} = \begin{pmatrix} 2 & -1 \\ -1 & 3 \end{pmatrix}.
\end{equation*}
We note that $T$ is a square torus. Since $\det G_T \det G_{T'} = 4^2 \det G$, the elements in \eqref{bases:eqn} generate a subgroup of index 4 in $\Lambda$ and the tori $T$ and $T'$ intersect in 4 points
$$Q_1 = (0,0), \quad Q_2= \frac 12 (1,1), \quad Q_3=\frac 12 (\zeta + \zeta^4, \zeta^2 + \zeta^3), 
\quad Q_4=\frac 12 (\zeta^2+\zeta^3, \zeta + \zeta^4).$$
These four points in $T$, written with respect to the orthogonal basis \eqref{bases:eqn}, are
$$Q_1 = (0,0)
\quad
Q_2 = \left( \tfrac 12, \tfrac 12 \right), 
\quad
Q_3 = \left(\tfrac 12, 0 \right),
\quad
Q_4 = \left(0, \tfrac 12 \right).
$$

\subsection{Fixed points}
We study the fixed points of the isometries $s,r$ of $\T$.
We can check that those of $s$ form the square Lagrangian torus $T$, 
while $r$ has 5 fixed points 
$$P_t = \frac {t-1} 5 (4+3\zeta+2\zeta^2+\zeta^3,4+3\zeta^2+2\zeta^4+\zeta)$$
as $t=1,\ldots, 5$. We can  
write them as
$$P_t= \frac {t-1}5(\zeta^4-\zeta+2(\zeta^3-\zeta^2), \zeta^3-\zeta^2 + 2(\zeta-\zeta^4))$$
and deduce that they are contained in $T$. With respect to the orthogonal basis \eqref{bases:eqn}, they have coordinates
$$P_1 = (0,0), \quad P_2 = \big(\tfrac 45, \tfrac 35 \big),
\quad P_3 = \big(\tfrac 35, \tfrac 15 \big), \quad P_4 = \big( \tfrac 25, \tfrac 45 \big),
\quad P_5 = \big(\tfrac 15, \tfrac 25\big),
$$

\subsection{An infinite order bihomolorphism on $\T$}
An interesting group automorphism for $\matZ[\zeta]$ is the multiplication by the golden ratio 
$$\lambda = -\zeta^2 - \zeta^3 = \frac{\sqrt 5+1}2.$$
This is an automorphism since $\lambda$ is invertible in $\matZ[\zeta]$ with inverse  
$$\lambda^{-1} = \zeta+ \zeta^4 = \frac{\sqrt 5 - 1}2.$$
This automorphism defines a biholomorphism of $\T$
$$\varphi (z,w) = \left(\lambda z, -\lambda^{-1} w\right) = -\left((\zeta^2 + \zeta^3)z, (\zeta + \zeta^4) w \right)$$
that has infinite order and commutes with the $D_{10}$ action.
The biolomorphism $\varphi$ preserves two orthogonal \emph{unstable} and \emph{stable} complex foliations $\calF^u$ and $\calF^s$ on $\T$, obtained by projecting the standard horizontal and vertical foliations into complex lines $\matC \times \{w\}$ and $\{z\} \times \matC$ of $\matC^2$. 
The map $\varphi$ acts on the leaves holomorphically, by stretching the leaves $\calF^u$ by the factor $\lambda$, and contracting those of $\calF^s$ by $1/\lambda$. It preserves the standard volume form on $\T$. 

The action of $\varphi$ on $H^1(\T, \matC)$ with the basis $dz, d\bar z, dw, d\bar w$ is a diagonal matrix with entries $\lambda,\lambda,-\lambda^{-1}, -\lambda^{-1}$. The action on $H^2(\T,\matC)=H^{2,0}(\T) \oplus H^{1,1}(\T)\oplus H^{0,2}(\T)$ with basis $dz\wedge d w$; $dz \wedge d\bar z, dw \wedge d\bar w, dz \wedge d \bar w, dw \wedge d\bar z$; $d\bar z \wedge d \bar w$ is a diagonal matrix with entries $-1$; $\lambda^{2}, \lambda^{-2}, -1, -1$; $-1$. 

The action of $\varphi$ on $H_1(\T, \matZ) = \Lambda$ with the basis \eqref{basis:eqn} is the matrix
\begin{equation} \label{matrix:eqn}
\begin{pmatrix}
0 & 1 & 0 & -1 \\
0 & 1 & 1 & -1 \\
-1 & 1 & 1 & 0 \\
-1 & 0 & 1 & 0
\end{pmatrix}
\end{equation}
which is of course diagonalizable with eigenvalues $\lambda, -\lambda^{-1}$, and with eigenvectors
\begin{equation*}
\begin{pmatrix}
-1 \\ 0 \\ \lambda \\ \lambda
\end{pmatrix}, \quad
\begin{pmatrix}
1 \\ 2 \\ \lambda \\ 2-\lambda
\end{pmatrix}, \quad
\begin{pmatrix}
\lambda \\ 0 \\ 1 \\ 1 
\end{pmatrix}, \quad
\begin{pmatrix}
1 \\ 2 \\ 1-\lambda \\ 1 + \lambda
\end{pmatrix}.
\end{equation*}

These correspond up to rescaling to $(1,0), (i,0), (0,1), (0,i)$ in $\matC^2$.
The biolomorphism $\varphi$ preserves the two orthogonal Lagrangian flat tori $T,T'$, and its action on $H_1(T,\matZ)$ and $H_1(T',\matZ)$ with respect to each of the bases \eqref{bases:eqn} is 
$$\begin{pmatrix} 1 & 1 \\ 1 & 0 \end{pmatrix}
$$
with eigenvalues $\lambda, -\lambda^{-1}$. The invariant foliations $\calF^u$ and $\calF^s$ in $\T$ intersect each $T$ and $T'$ into two orthogonal invariant real geodesic foliations, that are preserved with the same stretching factors. 
The biholomorphism $\varphi$ acts on the 5 points $P_1,\ldots,P_5$ by fixing $P_1$ and permuting $P_2,\ldots, P_5$ cyclically as $(2453)$. It also acts on the 4 points $Q_1, \ldots, Q_4$ as $(324)$, but the latter points will be less important for us.

\subsection{The flat 4-orbifold $\Orb$}
We now consider the flat 4-orbifold
$$\Orb = \T /D_{10}.$$

The singular set of $\Orb$ is the flat square torus $T\subset \Orb$, that is the isometric image of the Lagrangian flat torus $T \subset \T$ fixed by $s$, denoted with the same letter for simplicity. As usual with orbifolds, the singular set $T$ is naturally stratified: 
\begin{itemize}
\item The 0-strata are the 5 points $P_1,\ldots, P_5$, images of the corresponding points in $\T$, whose link is the spherical orbifold $S^3/D_{10}$; 
\item The 2-stratum is the complement $T\setminus \{P_i\}$, where the link of any point is the spherical orbifold $S^3/\langle f \rangle$ for a $\pi$-rotation $f$ around a closed geodesic.
\end{itemize}

It turns out that both links $S^3/D_{10}$ and $S^3/\langle f \rangle$ are homeomorphic to $S^3$, and hence the underlying space of $\Orb$ is a closed 4-manifold (we will show below that it is homeomorphic to $S^4$). Indeed $S^3/D_{10}$ is homeomorphic to $S^3$ with singular set the figure-eight knot \cite[Example 2.30]{CHK} shown in Figure \ref{figure-8:fig}, while $S^3/\langle f \rangle$ is $S^3$ with singular set the unknot. 

\subsection{Symmetries of $\Orb$} \label{additional:subsection}
We now list various symmetries of $\Orb$. These are affine real isomorphisms of $\matC^2$ that are easily checked to normalize $\Lambda$ and then $D_{10}$, and hence descend first to $\T$ and then to $\Orb$. As automorphisms of $\Orb$, they all preserve the singular set $T$ and permute the points $P_1,\ldots, P_5$. Some of them are isometries.

The most important symmetry of $\Orb$ is the infinite-order automorphism 
$$\varphi(z,w) = -((\zeta^2+\zeta^3) z, (\zeta+\zeta^4) w) = (\lambda z, -\lambda^{-1}w)$$ 
already considered above. The orbifold $\Orb$ has also a few finite-order symmetries:
\begin{align*}
\rho(z,w) & = (z,w) + P_2, \\
\sigma(z,w) & = -(z,w) = (\bar z, \bar w), \\
\tau(z,w) & = (w, \bar z), \\
\psi(z,w) & = -((\zeta+\zeta^4)\bar w, (\zeta^2+\zeta^3)z) = (-\lambda^{-1} \bar w, \lambda z).
\end{align*}

The symmetries $\rho, \sigma, \tau, \psi$ of $\Orb$ have order 5, 2, 4, 2 respectively. 
We have
$$\sigma = \tau^2, \quad \varphi = \tau\psi.
$$
We have discovered that the infinite-order automorphism $\varphi$ is a composition of two finite-order ones.
The symmetries $\rho,\sigma,\tau$ are isometries, while $\varphi, \psi$ are not. 
They act on the points $P_1,\ldots, P_5$ via the following permutations:
$$\varphi_* = \tau_* = (2453), \quad \rho_* = (12345), \quad \sigma_* = (25)(34), \quad \psi_* = \id.$$
Each symmetry preserves the singular set $T$ of $\Orb$. We identify $T$ with the square torus $\matR^2/\matZ^2$ using the basis \eqref{bases:eqn}, and with this identification the actions are
\begin{gather*}
\varphi(x,y) = (x+y,x), \quad \rho(x,y) = \left(x+\tfrac 45 ,y+\tfrac 35\right), \quad
\sigma(x,y) = (-x,-y) \\
\tau(x,y) = (-y, x), \quad
\psi(x,y) = (x,-x-y).
\end{gather*}

\subsection{Flat, hyperbolic, spherical cone manifolds} \label{flat:subsection}
The flat orbifold $\Orb$ is a particular kind of \emph{flat cone manifold}. We briefly recall the definition of these objects, following Thurston \cite{Th-p} and McMullen \cite{McM}.

Flat/hyperbolic/spherical cone manifolds were defined by Thurston \cite{Th-p} inductively on the dimension as follows: a \emph{spherical cone 1-manifold} is an ordinary Riemannian 1-manifold, and a \emph{flat}/\emph{hyperbolic}/\emph{spherical} \emph{cone $n$-manifold} is a metric space that is locally a flat/hyperbolic/spherical cone over a compact connected spherical cone $(n-1)$-manifold. 

The link of every point in a flat/hyperbolic/spherical cone $n$-manifold $M$ is a spherical cone $(n-1)$-manifold, and the point is \emph{singular} if the link is not isometric to $S^{n-1}$. The singular points form the \emph{singular set} in $ M$, that 
has codimension at least 2 and a natural stratification \cite{McM}, where every $k$-stratum is a totally geodesic connected flat/hyperbolic/spherical $k$-manifold. Each $(n-2)$-stratum has a well defined \emph{cone angle}.

Every locally orientable flat/hyperbolic/spherical orbifold is also naturally a flat/hyperbolic/spherical cone manifold, whose cone angles divide $2 \pi$. 

\subsection{The branched coverings}
We will prove in Corollary \ref{S4:cor} that $\Orb$ has underlying space $S^4$. Therefore by Alexander duality $H_1(S^4 \setminus T) = \matZ$ and for every $n\geq 2$ there is a well-defined regular branched covering
$$W_n \longrightarrow S^4$$
of degree $n$ ramified over $T$. The singular torus $T$ and points $P_1,\ldots, P_5$ lift from $\Orb$ to $W_n$, where we denote them with the same letters.

The space $W_n$ inherits from $\Orb$ the structure of a \emph{flat cone 4-manifold} with singular set the flat torus $T$. The cone angle of the stratum $T \setminus \{P_i\}$ in $W_n$ is $n\pi$. The geometry at the points $P_i$ is interesting. The link of $P_i$ in $\Orb$ is $S^3/D_{10}$, that is $S^3$ with an elliptic cone structure with singular set the figure-eight knot $K$ with cone angle $\pi$, see \cite[Example 2.30]{CHK}. The link of $P_i$ in $W_n$ is therefore the $n$-th branched covering $M_n$ of $S^3$ ramified over the figure-eight knot $K$, equipped with the lifted spherical cone structure: the singular set is a lifted copy of $K$, still denoted by $K$, with cone angle $n\pi$. It is well-known \cite[Example 2.33]{CHK} that $M_2 = L(5,2)$ has an elliptic structure, $M_3=\HW$ has a flat structure \cite{Z}, and $M_n$ has a hyperbolic structure for all $n \geq 4$. However, here the manifold $M_n$ is equipped with a \emph{spherical} cone structure for all $n$, with singular set $K$ having cone angle $n\pi$. 

The case $n=2$ is special, because the singular set of $W_2$ is actually reduced to the five points $P_1,\ldots, P_5$, and $W_2=\T/(\matZ/5\matZ)$ is a flat orbifold, with $\matZ/5\matZ$ generated by $r$. The link $M_2$ of $P_i$ is the lens space $S^3/(\matZ/5\matZ) = L(5,2)$.

The case of interest here is $n=3$.

\subsection{The fiber $F$}
We can finally define the fiber $F$. We set
$$\bar F = W_3.$$

This flat cone manifold is the triple branched covering over $\Orb$ ramified along $T$. We denote the lifts of $T$ and $P_i$ from $\Orb$ to $\bar F$ with the same letters.
As in $\Orb$, the singular set $T$ is naturally stratified: 
\begin{itemize}
\item The 0-strata $P_1,\ldots, P_5$, whose link is the Hantsche -- Wendt 3-manifold $\HW$, the triple branched covering over the figure eight knot \cite{Z}; 
\item The 2-stratum is the complement $T\setminus \{P_i\}$, with cone angle $3\pi$.
\end{itemize}

The link $\HW$ of $P_i$ admits a flat structure, but it is equipped here with a \emph{spherical} cone structure, with singular set a closed geodesic $K$ with cone angle $3\pi$.

We finally define the \emph{fiber} $F$ to be the compact 4-manifold with boundary obtained from $\bar F$ by removing  some small open stars of the points $P_1, \ldots, P_5$. The fiber $F$ has 5 boundary components homeomorphic to $\HW$.

\begin{prop}
The space $\bar F\setminus\{P_i\}$ is locally $\CAT(0)$.
\end{prop}
\begin{proof}
The singular set has cone angle $3\pi$, hence the link of its points is $\CAT(1)$ (this can be proved by induction on the dimension). Since the link of every point is $\CAT(1)$, the space is locally $\CAT(0)$, see \cite[Theorem II.5.5]{BH}.
\end{proof}

\begin{cor} The interior of $F$ admits a locally $\CAT(0)$ metric.
\end{cor}

\subsection{The symmetries of $\bar F$}
The space $\bar F$ has various symmetries. First of all, 
the branched covering $\bar F \to \Orb$ has an order three deck automorphism that we denote by $\phi$. Then, all the symmetries $\varphi, \rho, \sigma, \tau, \psi$ of $\Orb$ considered in Section \ref{additional:subsection} lift to symmetries of $\bar F$, because they preserve $T$ and the branched covering is a characteristic covering (it corresponds to a characteristic subgroups of $\pi_1(\Orb \setminus T)$). 

We denote some lifts of $\varphi,\rho,\sigma, \tau,\psi$ to $\bar F$ with the same letters (lifts are not unique since they can be composed with $\phi$; we will fix them unambiguously below only when needed).

The automorphisms $\phi, \rho, \sigma, \tau$ of $\bar F$ are finite order isometries; the automorphism $\psi$ has finite order but it is not an isometry; the automorphism $\varphi$ of $\bar F$ has infinite order and it is not an isometry: this is a \emph{pseudo-Anosov transformation} for $\bar F$, in the following sense.

\subsection{Pseudo-Anosov homeomorphism of an even-dimensional manifold} \label{pA:subsection}
We propose here a notion of pseudo-Anosov homeomorphic for even-dimensional manifolds that generalizes Thurston's definition for surfaces \cite{Th-s}.

Let $X$ be a flat cone manifold such that all the strata of the intrinsic stratification \cite{McM} are of even dimension. The flat cone manifold $X$ has dimension $2n$, and each stratum is a totally geodesic flat manifold of some dimension $2k$.

\begin{defn}
A \emph{geodesic foliation} $\calF$ in $X$ is the datum of a foliation on each $2k$-stratum into geodesic flat $k$-manifolds, such that locally the closure of a leaf in a $2k$-stratum is a union of leaves in some $2h$-strata with $h\leq k$, and every leaf lies locally in the closure of only finitely many higher-dimensional leaves.
\end{defn}

A flat cone manifold $X$ of dimension two is a flat surface with some cone points. If $X$ has a geodesic foliation $\calF$, the angle at each cone point $P$ is necessarily a multiple $k\pi$ of $\pi$, and the foliation near $P$ must be of the kind shown in Figure \ref{foliation:fig}. More generally, if a flat cone manifold $X$ of dimension $2n$ has a geodesic foliation $\calF$, the cone angles of its codimension two strata are $k\pi$, and the foliation $\calF$ there looks like one as in Figure \ref{foliation:fig} multiplied by a standard foliation of $\matR^{2n-2}$. The local picture at the lower strata can be more complicated.

\begin{defn}
Let $M$ be an even dimensional compact manifold, possibly with boundary. Let $\bar M$ be obtained from $M$ by shrinking each boundary component to a \emph{point at infinity}. A \emph{pseudo-Anosov homeomorphism} $\varphi \colon M \to M$ consists of:

\begin{enumerate}
\item A flat cone structure on $\bar M$ that is locally CAT(0) everywhere except possibly at the points at infinity; 
\item Two orthogonal geodesic foliations $\calF^s$, $\calF^u$ in $\bar M$ called \emph{stable} and \emph{unstable};
\item A homeomorphism $\varphi \colon \bar M \to \bar M$
that preserves each foliation, and acts locally on the leaves of $\calF^u$ and $\calF^s$ like a homothety of a factor $\lambda > 1$ and $1/\lambda$ respectively.
\end{enumerate}
\end{defn}

The number $\lambda > 1$ is the \emph{stretch factor} of $\varphi$. For every $k$ the homeomorphism $\varphi$ preserves the union of the $2k$-strata of $\bar M$, that is a flat $2k$-manifold, and acts on it locally like an affine function that expands along $\calF^u$ and contracts along $\calF^s$. 

The existence of a geodesic foliation on $\bar M$ forces the cone angles of all the condimension-two strata to be $k \pi$ for some integer $k\geq 1, k \neq 2$, and the locally CAT(0) condition implies that $k=1$ is allowed only for points at infinity in surfaces. Thus the proposed generalization indeed coincides with the usual notion of pseudo-Anosov homeomorphisms on surfaces given by Thurston \cite{Th-s} and recalled in the introduction.

The pseudo-Anosov homeomorphism $\varphi$ of $\bar M$ induces a homeomorphism on $M$ by blowing up the points at infinity similarly as in Farb -- Looijenga \cite[Section 2.3]{FL}. This operation goes as follows. Let $M \to \bar M$ be the map that quotients every boundary component $X\subset \partial M$ to a point at infinity $p \in \bar M$. The space $\bar M$ is a cone flat manifold, hence $p$ has a spherical link, that is the set of germs of straight lines exiting from $p$ (the link is also equipped with a a spherical cone manifold structure). We can identify the spherical link at $p$ with $X$ in a natural way. Since $\varphi$ is affine on all strata, it sends germs of lines to germs of lines, and hence induces a homeomorphism between the links of $p$ and $\varphi(p)$. Therefore the homeomorphism $\varphi$ lifts canonically to $M$.

\subsection{The pseudo-Anosov homeomorphism $\varphi$}
The symmetry $\varphi$ of $\bar F$ is a pseudo-Anosov homeomorphism with stretch factor the golden ratio
$$\lambda = \frac{1+\sqrt 5}2.$$

The unstable and stable foliations $\calF^u, \calF^s$ on $\bar F$ are inherited from $\T$ and $\Orb$. These are the foliations $\matC \times \{p\}, \{p\} \times \matC$ of $\matC^2$, projected to $\Orb$ and then lifted to $\bar F$.
The homeomorphism $\varphi$ preserves the singular set $T$ where it acts as $(x,y) \mapsto (x+y,x)$. The flat torus $\T$ is not the only invariant surface in $\bar F$. Recall from Section \ref{Lagrangian:subsection} that $T$ contains four special points $Q_1,\ldots, Q_4$.

\begin{prop} There is a $\varphi$-invariant totally geodesic flat cone surface $\Sigma \subset \bar F$ of genus two that intersects $T$ orthogonally at $Q_1,\ldots, Q_4$. It intersects $\calF^u$ and $\calF^s$ into two orthogonal geodesic foliations. The restriction of $\varphi$ to $\Sigma$ is pseudo-Anosov.
\end{prop}
\begin{proof}
The surface $\Sigma$ is the pre-image of the image in $\Orb$ of the Lagrangian torus $T' \subset \T$ defined in Section \ref{Lagrangian:subsection}. Recall that $T, T' \subset \T$ intersect orthogonally in four points $Q_1, \ldots, Q_4$. The map $s\colon \T \to \T$ preserves $T'$ and projects it to a flat 2-sphere $T'/\langle s \rangle$ in $\Orb$ with four cone points $Q_1,\ldots, Q_4$ with angle $\pi$, and the pre-image $\Sigma$ is the triple branched covering over $T'/\langle s \rangle$ ramified at these four points; such a surface has genus two by a Euler characteristic count. The restriction $\varphi$ to $\Sigma$ is pseudo-Anosov because its restriction to $T$ was the Anosov map $(x,y) \mapsto (x+y,x)$.
\end{proof}

\section{Tessellations and triangulations} \label{tessellation:section}
Our aim is now to construct nice tessellations and triangulations for the flat torus $\T$, the orbifold $\Orb$, and the triple branched covering $\bar F$.  

\subsection{A tessellation $\Pi$ for $\Orb$.}
We define the \emph{simplex} $S$ as usual as the convex hull in $\matR^5$ of the 5 vertices $e_1, \ldots, e_5$. The \emph{rectified simplex} $R$ is the convex hull of the 10 vertices $e_i+e_j$ with $i\neq j$. These are the midpoints of the edges of the larger simplex $2S$. We now describe a tessellation $\Pi$ of $\Orb$ that consists of one copy of $R$ and one of $S$.

The $5+10=15$ facets of $S$ and $R$ are respectively
\begin{gather*}
x_1 = 0,\quad \ldots, \quad x_5 = 0; \\
x_1 = 0, \quad \ldots, \quad x_5 = 0, \qquad x_1 = 1, \quad \ldots, \quad x_5 = 1.
\end{gather*}
The 5 facets of $S$ are regular tetrahedra; those of $R$ of type $x_i=0$ and $x_i=1$ are regular octahedra and tetrahedra, respectively, with vertices $e_j+e_k$ with $j,k \neq i$ and $e_i+e_j$ with $j\neq i$. We assign to the three vertices $e_i$, $e_{i-1}+e_{i+1}$ and $e_{i+2}+e_{i-2}$ of $S$ and $R$ the same \emph{label} $i \in \{1,\ldots, 5\}$ (indices are always considered modulo 5). 

We now construct a metric space from $S \sqcup R$ by \emph{pairing} the 10 tetrahedral facets of $S \sqcup R$, and \emph{folding} the 5 octahedral facets of $R$, in a way that all the vertices with the same labels will be identified. 
More precisely, we do the following:
\begin{itemize}
\item We \emph{pair} each tetrahedron $x_i=0$ of $S$ with the one $x_i=1$ of $R$ via the unique isometry that preserves the labels of the vertices;
\item We \emph{fold} each octahedron $x_i = 0$ of $R$ along the \emph{middle square} $Q_i = \{x_{i+1} + x_{i-1} = x_{i+2} + x_{i-2}=1\}$. That is, we subdivide the octahedron along $Q_i$ into two square pyramids, and we identify them by reflecting along $Q_i$. 
\end{itemize}

We have \emph{paired} 5 pairs of tetrahedra and \emph{folded} 5 octahedra.
The folding of the octahedron $x_i=0$ identifies the two vertices $e_{i+2}+e_{i-2}$ and $e_{i+1}+e_{i-1}$ of the pyramids: these are opposite vertices in the octahedron and are both labeled as $i$. After pairing and folding all the facets of $S \sqcup R$ we have identified all the vertices in $S$ and $R$ having the same labels. 

\begin{prop} \label{RS:prop}
The space obtained from $S \sqcup R$ by pairing and folding is isometric to $\sqrt 2 \cdot \Orb$. The singular points $P_1, \ldots, P_5$ are the vertices labeled as $1,\ldots, 5$. The singular torus $T$ is the union of the five middle squares $Q_1,\ldots, Q_5$ as in Figure \ref{torus2:fig}. 
\end{prop}

\begin{figure}
\vspace{.4 cm}
\centering
\labellist
\small\hair 2pt
\pinlabel $T$ at 150 20
\pinlabel $Q_2$ at 30 78
\pinlabel $Q_4$ at 80 78
\pinlabel $Q_1$ at 130 78
\pinlabel $Q_3$ at 30 28
\pinlabel $Q_5$ at 80 28
\pinlabel $P_3$ at -5 95
\pinlabel $P_4$ at -5 45
\pinlabel $P_5$ at -5 11
\pinlabel $P_5$ at 45 95
\pinlabel $P_1$ at 45 45
\pinlabel $P_2$ at 45 11
\pinlabel $P_2$ at 95 95
\pinlabel $P_3$ at 95 45
\pinlabel $P_4$ at 95 11
\pinlabel $P_4$ at 145 95
\pinlabel $P_5$ at 145 45
\endlabellist
\includegraphics[width=5.5 cm]{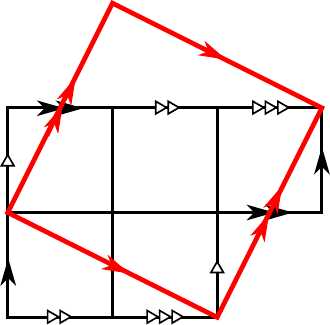}
\vspace{.3 cm}
\caption{The middle squares $Q_1,\ldots, Q_5$ form the singular torus $T$ of $\Orb$, with vertices $P_1,\ldots, P_5$. Edges with similar arrows should be identified. This is indeed a square torus: it suffices to take the red square as a fundamental domain.}\label{torus2:fig}
\end{figure}
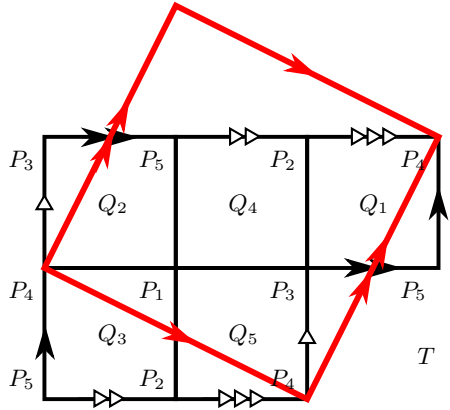

Here $\sqrt 2 \cdot \Orb$ is simply $\Orb$ rescaled by a factor $\sqrt 2$, that is irrelevant for us, so we will henceforth drop it.
Note that each folding produces a cone angle $\pi$ around each square $Q_i$, as expected, since $T$ has cone angle $\pi$. We postpone the proof of Proposition \ref{RS:prop} at the end of this section for the sake of clarity. We call $\Pi$ the resulting tessellation of $\Orb$ into one copy of $S$ and one of $R$.

\begin{cor}  \label{S4:cor}
The underlying space of $\Orb$ is piecewise-linearly homeomorphic to $S^4$. The singular set $T$ is a torus that is locally flat everywhere except at the five points $P_1,\ldots, P_5$ where it is locally a cone over the figure eight knot.
\end{cor}
\begin{proof}
By folding all the octahedral facets of the rectified simplex $R$, we have transformed it into a 4-disc with corners, homeomorphic to a simplex $S'$. The underlying space of $\Orb$ is thus triangulated into two simplexes $S$ and $S'$. A triangulation with two simplexes and 5 distinct vertices $P_1,\ldots, P_5$ is the standard triangulation of $S^4$.

We already know that the link of each $P_i$ is $S^3/D_{10}$, that is topologically $S^3$ with the figure eight knot $K$ as a singular set \cite[Example 2.30]{CHK}. Hence $T$ is everywhere flat except at these points where it is a cone over $K$.
\end{proof}

Recall that $\Orb$ has some symmetries $\varphi, \rho, \sigma, \tau, \psi$. The symmetries $\rho, \sigma, \tau$ are isometries and can be read easily on $\Pi$.

\begin{prop} \label{rst:prop}
The symmetries $\rho,\sigma,\tau$ preserve the tessellation $\Pi$, and restrict to both $S$ and $R$ as the following isometries of $\matR^5$:
\begin{align*}
\rho \colon (x_1,x_2,x_3,x_4,x_5) & \longmapsto (x_5,x_1,x_2,x_3,x_4), \\
\sigma \colon (x_1,x_2,x_3,x_4,x_5) & \longmapsto (x_1,x_5,x_4,x_3,x_2), \\
\tau \colon (x_1,x_2,x_3,x_4,x_5) & \longmapsto (x_1,x_3,x_5,x_2,x_4). 
\end{align*}
\end{prop}

We can verify that they act on $P_1,\ldots, P_5$ as the permutations 
$\rho_* = (12345), \sigma_* = (25)(34), \tau_* = (2453)$, as already noticed in Section \ref{additional:subsection}.
These symmetries generate a group of isometries of $\Orb$ of order 20. The symmetries $\varphi, \psi$ cannot preserve $\Pi$. Indeed $\varphi$ has infinite order and thus cannot preserve any tessellation, and $\psi$ cannot preserve a tessellation preserved by $\tau$ since $\varphi = \tau\psi$. We also postpone the proof of Proposition \ref{rst:prop} to the end of this section.

\subsection{Two triangulations $\Delta, \Delta'$ for $\Orb$} \label{two:subsection}
We construct a triangulation of $\Orb$ with 12 simplexes, by subdividing the rectified simplex $R$ into 11 simplexes. There are in fact two natural ways to do it.
The first is to cut $R$ along the 5 hyperplanes
$$x_{1} + x_{2} = 1, \quad x_{2} + x_{3} = 1, \quad \ldots \quad
x_{5} + x_{1} = 1.
$$

With this method we cut $R$ into 11 simplexes: a \emph{central} one
$S_-$ bounded by these 5 hyperplanes and having vertices
$$(0,1,0,1,0), \ (1,0,1,0,0),\ (0,1,0,0,1), \ (1,0,0,1,0), \ (0,0,1,0,1), $$
plus 5 more simplexes $S^i_-$ bounded by the hyperplanes
$$x_{i+1}+x_{i+2} = 1, \quad x_{i+2}+x_{i+3} = 1, \quad x_{i+3}+x_{i+4} =1, \quad  x_{i+1} = 0, \quad x_{i+4}=0$$
that have vertices
$$e_{i+2}+e_{i+3}, \quad e_{i}+e_{i+2}, \quad e_{i+1}+e_{i+3}, \quad e_{i+2}+e_{i+4}, \quad e_{i+3}+e_{i}
$$
plus 5 more simplexes $S^i_+$ bounded by the hyperplanes
$$x_{i+4}+x_{i} = 1, \quad x_{i}+x_{i+1}=1, \quad x_{i} = 1, \quad x_{i+2}=0, \quad x_{i+3}=0$$
that have vertices
$$e_{i+4}+e_{i}, \quad e_{i}+e_{i+1}, \quad e_{i+3} + e_{i}, \quad e_{i+4} + e_{i+1}, \quad e_{i} + e_{i+2}.
$$

The 5 vertices of each simplex have distinct labels $1,\ldots, 5$.
Each octahedral facet $x_i = 0$ is cut into 4 simplexes by the hyperplanes $x_{i+1}+x_{i+2} = 1$ and $x_{i+2}+x_{i+3}=1$ (the latter intersects the octahedron in the square $Q_i$).

We call $R^\Delta$ this triangulation of $R$ into 11 simplexes, and $\Delta$ the resulting triangulation of $\Orb$ into 12 simplexes that is the union of $S$ and $R^\Delta$. Each simplex in $\Delta$ has vertices $P_1,\ldots, P_5$.
We set $S_+=S$. By direct inspection we can prove:

\begin{prop}
By subdividing $R$ into 11 simplexes as $R^\Delta$ we get a triangulation $\Delta$ of $\Orb$ into 12 simplexes
$$S_-, \ S^i_-,\ S^i_+,\ S_+$$
with $i=1,\ldots, 5$. The facets are paired as follows: 
$$S_- \stackrel i \longleftrightarrow S^i_-, \quad S^i_- \stackrel  {i\pm 1} \longleftrightarrow
S^{i \mp 2}_+, \quad S^i_- \stackrel  {i\pm 2} \longleftrightarrow
S^{i \mp 2}_+,
\quad S^i_+ \stackrel i\longleftrightarrow S_+.$$
Here $A \stackrel j \longleftrightarrow B$ indicates that the $j$-th facets of $A$ and $B$ (that is those opposite to the vertex labeled with $j$) are glued along the unique isometry that preserves the labeling in $\{1,\ldots,5\}$ of the vertices. 
\end{prop}

We have a triangulation $\Delta$ of a 4-dimensional flat orbifold $\Orb$ with 12 simplexes and 5 vertices $P_1,\ldots, P_5$ that can be described in a reasonably simple way.

We obtain another triangulation $R^{\Delta'}$ of $R$ by cutting along the hyperplanes 
$$x_1+x_3 = 1, \quad x_2+x_4 = 1, \quad x_3+x_5 = 1, \quad x_4+x_1 = 1, \quad x_5+x_2 = 1.$$
This triangulation also consists of 11 simplexes $S_-', S_-^{1'}, \ldots, S_-^{5'}, S_+^{1'}, \ldots, S_+^{5'}$ where $S_-', S_-^{i'}$ and $S_+^{i'}$ have vertices respectively
\begin{gather*}
(1,1,0,0,0), \quad (0,1,1,0,0), \quad (0,0,1,1,0), \quad (0,0,0,1,1), \quad (1,0,0,0,1); \\
e_{i+4}+e_{i+1}, \quad e_{i+3}+e_{i+4}, \quad e_{i+4}+e_{i}, \quad e_{i}+e_{i+1}, \quad e_{i+1}+e_{i+2}; \\
e_{i+1}+e_{i+3}, \quad e_{i+2}+e_{i+4}, \quad e_{i+1} + e_{i+2}, \quad e_{i+2} + e_{i+3}, \quad e_{i+3} + e_{i+4}.
\end{gather*}

\begin{figure}
\vspace{.4 cm}
\centering
\labellist
\small\hair 2pt
\pinlabel $\Delta$ at 10 10
\pinlabel $(0,1,0,0,1)$ at 80 -10
\pinlabel $(0,0,1,1,0)$ at 135 220
\pinlabel $(0,1,1,0,0)$ at 20 140
\pinlabel $(0,0,1,0,1)$ at 115 135
\pinlabel $(0,1,0,1,0)$ at 115 55
\pinlabel $(0,0,0,1,1)$ at 220 60
\pinlabel $S_3^+$ at 60 180
\pinlabel $S_4^+$ at 185 150
\pinlabel $S_5^-$ at 30 50
\pinlabel $S_2^-$ at 160 35
\pinlabel $\Delta'$ at 300 10
\endlabellist
\includegraphics[width=10 cm]{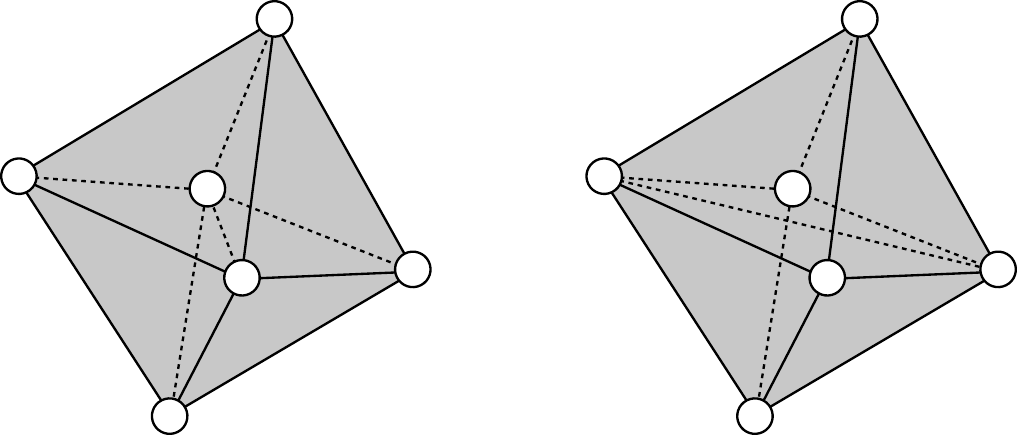}
\vspace{.3 cm}
\caption{The triangulations $R^\Delta$ and $R^{\Delta'}$ of $R$ subdivide each octahedral facet of $R$ into four tetrahedra in two different ways, that depend on the choice of a diagonal. Here we show the octahedron $x_1=0$ and the four adjacent simplexes of $\Delta$. The middle square $Q_0$ is horizontal and is subdivided into two triangles in two different ways.}\label{Octa:fig}
\end{figure}
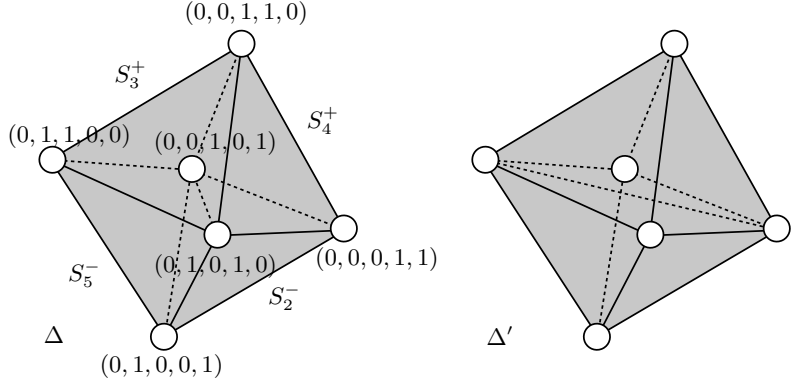

The 5 vertices of each simplex have distinct labels $1,\ldots, 5$.
Each octahedron facet of $R$ is subdivided by $R^\Delta$ and $R^{\Delta'}$ in 4 tetrahedra in two different ways, see Figure \ref{Octa:fig}.
Using this second subdivision of $R$ we get another triangulation $\Delta' = S \cup R^{\Delta'}$ of $\Orb$ with 12 simplexes. The two triangulations subdivide the singular torus $T$ in two different ways, as shown in Figure \ref{torus3:fig}. We have $\tau(\Delta) = \Delta'$ and $\tau(\Delta') = \Delta$ since
$$\tau(S_{\pm}) = S_{\pm}', \quad \tau(S_{\pm}^i) = S_{\pm}^{\tau_*(i)'}, \quad \tau(S_{\pm}') = S_{\pm}, \quad
\tau(S_{\pm}^{i'}) = S_{\pm}^{\tau_*(i)}.
$$
Here we have set $S_+' = S = S_+$ and $\tau_*=(2453)$.

\begin{figure}
\vspace{.2 cm}
\centering
\labellist
\small\hair 2pt
\pinlabel $\Delta$ at 140 20
\pinlabel $\Delta'$ at 340 20
\pinlabel $P_3$ at -5 112
\pinlabel $P_4$ at -5 62
\pinlabel $P_5$ at -5 11
\pinlabel $P_5$ at 45 112
\pinlabel $P_1$ at 45 62
\pinlabel $P_2$ at 45 11
\pinlabel $P_2$ at 95 112
\pinlabel $P_3$ at 95 62
\pinlabel $P_4$ at 95 11
\pinlabel $P_4$ at 145 112
\pinlabel $P_5$ at 145 62
\pinlabel $P_3$ at 192 112
\pinlabel $P_4$ at 192 62
\pinlabel $P_5$ at 192 11
\pinlabel $P_5$ at 242 112
\pinlabel $P_1$ at 262 62
\pinlabel $P_2$ at 262 11
\pinlabel $P_2$ at 292 112
\pinlabel $P_3$ at 312 62
\pinlabel $P_4$ at 312 11
\pinlabel $P_4$ at 342 112
\pinlabel $P_5$ at 362 62
\endlabellist
\includegraphics[width=12 cm]{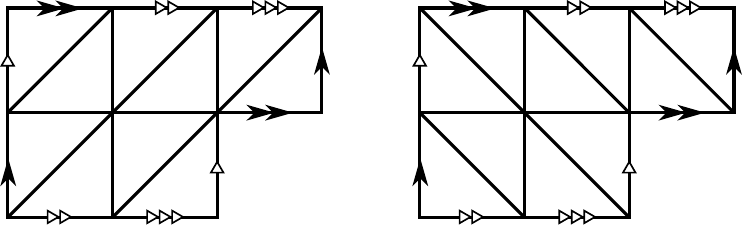}
\caption{The triangulations of $T$ induced by $\Delta$ and $\Delta'$. The isometry $\tau$ is a clockwise $\pi/2$ rotation around the point $P_1$ and sends the first triangulation to the second.}\label{torus3:fig}
\end{figure}
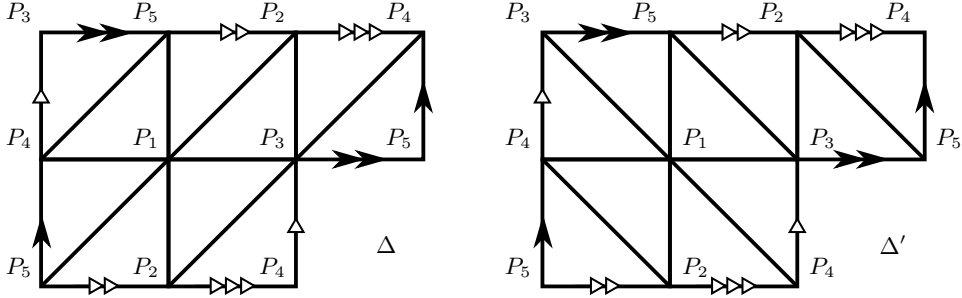
 
\begin{prop} \label{summary:prop}
The following hold:
\begin{enumerate}
\item The isometries $\rho$, $\sigma$ preserve $\Pi, \Delta$ and $\Delta'$;
\item The isometry $\tau$ preserves $\Pi$ and exchanges $\Delta$ and $\Delta'$; 
\item The map $\psi$ does not preserve $\Pi$ and $\Delta'$, but it preserves $\Delta$;
\item The map $\varphi$ does not preserve any of $\Pi,\Delta,\Delta'$, but it sends $\Delta$ to $\Delta'$.
\end{enumerate}
The map $\psi$ acts on $\Delta$ as follows: it exchanges $S_+$ with $S_-$ and $S_+^i$ with $S_-^i$ via the unique vertices-preserving isomorphism.
\end{prop}

We postpone the proof of the proposition to the end of the section. Note that $\psi$ is not an isometry! 
The isomorphisms between simplexes are not all isometries.
 
\subsection{The lifted tessellation $\Pi$ for $\bar F$} \label{Pi:subsection}
The tessellation $\Pi$ into one simplex $S$ and one rectified simplex $R$, and the triangulations $\Delta, \Delta'$ with 12 simplexes of $\Orb$, lift along the branched covering to a tessellation $\Pi$ for $\bar F$ into $3$ simplexes $S_1, S_2, S_3$ and $3$ rectified simplexes $R_1, R_2, R_3$, and to two triangulations $\Delta, \Delta'$, each with $36$ simplexes. We use the same letters $\Pi, \Delta, \Delta'$ for both $\Orb$ and $\bar F$ for simplicity.

We describe the tessellation $\Pi$ of $\bar F$. Recall that each octahedron facet $x_i=0$ in $R$ is cut by the middle square $Q_i$ into two square pyramids, that we now denote by $P_i^+, P_i^-$, with apices $e_{i+2}+e_{i-2}$ and $e_{i+1}+e_{i-1}$, respectively.

The tessellation $\Pi$ of $\bar F$ consists of $S_1,S_2,S_3,R_1,R_2,R_3$ glued as follows: 

\begin{itemize}
\item We glue the tetrahedron $x_i=0$ of $S_a$ with the one $x_i=1$ of $R_a$;
\item We glue the pyramid $P_i^+$ of $R_a$ to the pyramid $P_i^-$ of $R_{a+1}$. 
\end{itemize}

We do this for all $i=1,\ldots,5$ and all $a=1,2,3$, with indices considered cyclically. We always use the unique gluing that preserves the labeling $1,\ldots,5$ of the vertices. 

It is easy to verify that this tessellation $\Pi$ of $\bar F$ is indeed the triple branched covering of the one $\Pi$ for $\Orb$ ramified along $T$. For each $i=1,\ldots, 5$
the three middle squares $Q_i$ of the rectified simplexes $R_1, R_2, R_3$ are all identified to yield a single square, still denoted by $Q_i$. The squares $Q_1,\ldots, Q_5$ form the singular torus $T$ as in Figure \ref{torus2:fig} precisely as in $\Orb$, with the difference that now each $Q_i$ is adjacent to three rectified simplexes $R_1, R_2, R_3$ and hence the points in $T \setminus \{P_i\}$ have cone angle $3\pi$. The tessellation $\Pi$ of $\bar F$ has 5 vertices $P_1,\ldots, P_5$, like that of $\Orb$.

\subsection{The 60 isometries of $\bar F$}
Recall that $\phi$ is a deck transformation of $\bar F$ and $\rho, \sigma, \tau$ are some isometries of $\bar F$ obtained by lifting the corresponding isometries of $\Orb$. There was some ambiguity in this definition, that we now resolve by fixing once for all an explicit expression for each of them.

\begin{prop}
The isometries $\phi, \rho, \sigma, \tau$ act on the tessellation $\Pi$ of $\bar F$ as follows:
\begin{enumerate}
\item $\phi$ sends $S_a,R_a$ to $S_{a+1}, R_{a+1}$ identically;
\item $\rho$ acts on each $S_a,R_a$ via the map $(x_1,x_2,x_3,x_4,x_5)\mapsto (x_5,x_1,x_2,x_3,x_4)$;
\item $\sigma$ acts on each $S_a,R_a$ via the map $(x_1,x_2,x_3,x_4,x_5)\mapsto (x_1,x_5,x_4,x_3,x_2)$;
\item $\tau$ sends $S_a,R_a$ to $S_{2-a}, R_{2-a}$ via $(x_1,x_2,x_3,x_4,x_5)\mapsto (x_1,x_3,x_5,x_2,x_4)$.
\end{enumerate}
Indices in $R_a, S_a$ should be considered modulo 3.
\end{prop}
\begin{proof}
The expressions for $\rho, \sigma, \tau$ are obtained by choosing a lift of the corresponding expression in Proposition \ref{rst:prop} from $\Orb$ to $\bar F$.
\end{proof}

The isometries $\phi, \rho, \sigma, \tau$ have order $3, 5, 2, 4$ and generate a group of 60 symmetries of $\Pi$, which may be identified with the group $G < S_3 \times S_5$ of $60$ elements $\alpha = (\alpha_1, \alpha_2) \in G$ that act on $\Pi$ by sending $S_a, R_a$ to $S_{\alpha_1(a)}, R_{\alpha_1(a)}$ via the isometry
$$(x_1,x_2,x_3,x_4,x_5) \longmapsto (x_{\alpha^{-1}_2(1)}, x_{\alpha^{-1}_2(2)}, x_{\alpha^{-1}_2(3)}, x_{\alpha^{-1}_2(4)}, x_{\alpha^{-1}_2(5)}).$$
With this interpretation we have
$$\phi = ((123), \id), \quad \rho = (\id, (12345)), \quad \sigma = (\id, (34)(52)), \quad
\tau = ((23), (2453)).$$ 

\subsection{The triangulation $\Delta$ for $\bar F$} \label{Delta:subsection}
The triangulation $\Delta$ for $\bar F$ is obtained by lifting the triangulation $\Delta$ for $\Orb$. It consists of $36$ simplexes
$$S_{a,-}, \qquad S_{a,-}^i, \qquad S_{a,+}^i, \qquad S_{a,+}$$
with $a=1,2,3$ and $i=1,\ldots, 5$.
The facets are paired as follows: 
\begin{equation} \label{pairing:eqn}
S_{a,-} \stackrel i \longleftrightarrow S^i_{a,-}, \quad S^i_{a,-} \stackrel  {i\pm 1} \longleftrightarrow
S^{i \mp 2}_{a-1,+}, \quad S^i_{a,-} \stackrel  {i\pm 2} \longleftrightarrow
S^{i \mp 2}_{a,+},
\quad S^i_{a,+} \stackrel i\longleftrightarrow S_{a,+}.
\end{equation}

Here $A \stackrel j \longleftrightarrow B$ indicates that the $j$-th facets of $A$ and $B$ are glued along the unique isometry that preserves the labeling $1,\ldots,5$ of the vertices. By direct inspection, and by fixing a lifting of $\psi$ from Proposition \ref{summary:prop}, we find the following.

\begin{prop}
The isomorphisms $\phi, \rho, \sigma, \psi$ act on $\Delta$ as follows:
\begin{enumerate}
\item $\phi$ acts as 
$S_{a,\pm} \to S_{a+1,\pm},\ S_{a,\pm}^i \to S_{a+1,\pm}^i$
preserving the vertices;
\item $\rho$ acts as 
$S_{a,\pm } \to S_{a,\pm },\ S_{a,\pm}^i \to S_{a,\pm}^{i+1}$
permuting the vertices as $(12345)$;
\item $\sigma$ acts as
$S_{a,\pm } \to S_{a,\pm },\ S_{a,\pm}^i \to S_{a,\pm}^{\sigma_*(i)}$
permuting vertices as $\sigma_* = (25)(34)$;
\item $\psi$ acts as 
$S_{a,-} \to S_{a,+},\ S_{a,+} \to S_{a+1,-},\ 
S^i_{a,-} \to S^i_{a,+},\ S^i_{a,+} \to S^i_{a+1,-}$
preserving the vertices.
\end{enumerate}
\end{prop}

We deduce that $\phi = \psi^2$ and $\psi$ has order $6$. Recall that $\psi$ is not an isometry.
These symmetries generate a group $G' = \matZ/6\matZ \times D_{10}$ of order $60$, with $\matZ/6\matZ$ generated by $\psi$ and $D_{10}$ by $\rho,\sigma$.

Analogously the triangulation $\Delta'$ for $\bar F$ obtained by lifting that of $\Orb$ consists of the 36 simplexes
$$S_{a,-}', \qquad S_{a,-}^{i'}, \qquad S_{a,+}^{i'}, \qquad S_{a,+}'$$
with $a=1,2,3$ and $i=1,\ldots,5$. We have $S_{a,+}' = S_a = S_{a,+}$. 
\begin{prop}
The isometry $\tau$ sends $\Delta$ to $\Delta'$ as follows:
$$\tau(S_{a,\pm}) = S_{2-a,\pm}', \quad \tau(S_{a,\pm}^i) = S_{2-a,\pm}^{\tau_*(i)'}, \quad \tau(S_{a,\pm}') = S_{2-a,\pm}, \quad
\tau(S_{a,\pm}^{i'}) = S_{2-a,\pm}^{\tau_*(i)}
$$
permuting the vertices as $\tau_* = (2435)$.
\end{prop}

\subsection{Infinitely many symmetries for $\bar F$} \label{symmetries:subsection}
We summarize the many symmetries of $\bar F$ that we have written explicitly.
The isometries $\phi, \rho,  \sigma$ have order $3$, $5$, $2$, and generate a group $\matZ/3\matZ \times D_{10}$ of $30$ isometries for $\bar F$, that preserve both $\Pi$ and $\Delta$. 

If we add to this group the isometry $\tau$, we find the group $G$ of $60$ isometries of $\Pi$, and we note that $\tau^2 = \sigma$. If instead we add the non-isometry $\psi$, we find the group $G'$ of $60$ symmetries for $\Delta$, and note that $\psi^2 = \phi$.

The two groups $G, G'$ of symmetries for $\Pi$ have the same order $60$ but they do not coincide, and they generate an \emph{infinite} group of symmetries for $\Orb$, which contains the infinite-order pseudo-Anosov homeomorphism
$\varphi = \tau \psi$. By combining the expressions for $\tau$ and $\psi$ written above we find:

\begin{prop} \label{varphi:prop}
The pseudo-Anosov map $\varphi$ sends $\Delta$ to $\Delta'$ as follows:
$$\varphi(S_{a,-}) = S_{2-a,+}', \quad 
\varphi(S_{a,+}) = S_{1-a,-}', \quad
\varphi(S_{a,-}^i) = S_{2-a,+}^{\varphi_*(i)'}, \quad
\varphi(S_{a,+}^i) = S_{1-a,-}^{\varphi_*(i)'}, 
$$
permuting the vertices as $\varphi_* = (2453)$.
\end{prop}

\subsection{The $A_4$ lattice}
The rest of Section \ref{tessellation:section} is entirely devoted to the proofs of Propositions \ref{RS:prop}, \ref{rst:prop} and \ref{summary:prop}. The material introduced here will not be used in the next sections. We furnish an alternative more geometric definition of $\T$ and $\Orb$ that employs the \emph{$A_4$ lattice} and the \emph{5-cell tessellation}.
Consider the hyperplane
$$H = \{x_1+x_2+x_3+x_4+x_5=0\} \subset \matR^5.$$
The \emph{$A_4$ lattice} is
$$A_4 = \matZ^5 \cap H.$$
A basis for $A_4$ is
$$e_2-e_1, \quad e_3-e_2, \quad e_4-e_3, \quad e_5-e_4$$
and the corresponding Gram matrix is
$$\begin{pmatrix} 
2 & -1 & 0 & 0 \\
-1 & 2 & -1 & 0 \\
0 & -1 & 2 & -1 \\
0 & 0 & -1 & 2
\end{pmatrix}.$$

\subsection{The 4-torus $H/\Gamma$} \label{HGamma:subsection}
We are interested here in the index 5 sublattice $\Gamma < A_4$ generated by the vectors
$$(0,-1,1,1,-1), \quad (-1,0,-1,1,1), \quad (1,-1,0,-1,1), \quad (1,1,-1,0,-1).$$
We note that $(-1,1,1,-1,0) \in \Gamma$ and the Gram matrix is
$$
\begin{pmatrix}
4 & -1 & -1 & -1 \\
-1 & 4 & -1 & -1 \\
-1 & -1 & 4 & -1 \\
-1 & -1 & -1 & 4
\end{pmatrix}.
$$
Our interest stems from the fact that we get twice the Gram matrix $G$ of the generators of $\Lambda$, see Section \ref{T:subsection}. Therefore the flat 4-torus
$H / \Gamma$
is, after rescaling by a factor $\sqrt 2$, isometric to $\T$. We write
$$H/ \Gamma = \sqrt 2\cdot  \T.$$

We can check that the isometries $r$, $s$ of $\T$ correspond to the isometries of $H/\Gamma$ 
\begin{equation*}
\begin{aligned} 
r\colon (x_1,x_2,x_3,x_4,x_5) & \longmapsto (x_5,x_1,x_2,x_3,x_4), \\
s\colon (x_1,x_2,x_3,x_4,x_5) & \longmapsto (-x_1, -x_5,-x_4,-x_3,-x_2).
\end{aligned}
\end{equation*}

The Lagrangian tori $T$ and $T'$ in $\T$ correspond to the tori in $H/\Gamma$ generated by
$$(0,-1,-2,2,1), \quad (0,-2,1,-1,2); \qquad (0,-1,1,1,-1), \quad
(-2,1,0,0,1)$$
that have equations
$$x_2+x_5=x_3+x_4=0; \quad x_2-x_5=x_3-x_4=0.$$
The fixed points for $r$ are 
$$P_{t+1} = (0, -t, 0, t, 0) = (0,-2t,-t,t,2t)$$
for $t=0,\ldots,4$. The isometries $r,s$ generate a group $D_{10}$ acting on $H/\Gamma$, with quotient $\sqrt 2 \cdot \Orb$. 
The symmetries of $\Orb$ in Section \ref{additional:subsection} can be read here as
$$ \varphi(x) = Cx, \quad 
\rho(x) = x + P_2, \quad
\sigma(x) = -x, \quad
\tau(x) = Ax, \quad
\psi(x) = Bx
$$
with
$$
A = 
\begin{pmatrix}
1 & 0 & 0 & 0 & 0 \\
0 & 0 & 1 & 0 & 0 \\ 
0 & 0 & 0 & 0 & 1 \\
0 & 1 & 0 & 0 & 0 \\
0 & 0 & 0 & 1 & 0
\end{pmatrix}, \quad
B = 
\frac 15 \begin{pmatrix}
-1 & -1 & 4 & 4 & -1 \\
4 & -1 & -1 & -1 & 4 \\
-1 & 4 & 4 & -1 & -1 \\
-1 & -1 & -1 & 4 & 4 \\
4 & 4 & -1 & -1 & -1
\end{pmatrix},
$$
$$
C = AB = \frac 15 \begin{pmatrix}
-1 & -1 & 4 & 4 & -1 \\
-1 & 4 & 4 & -1 & -1 \\
4 & 4 & -1 & -1 & -1 \\
4 & -1 & -1 & -1 & 4 \\
-1 & -1 & -1 & 4 & 4
\end{pmatrix}.
$$
We can also recover $C$ from the representation of $\varphi$ as \eqref{matrix:eqn} in the basis of $\Lambda$.

\subsection{The 5-cell tessellation} \label{5cell:subsection}
The advantage of working with $H/\Gamma$ instead of $\T$ is that the first inherits from $H$ a natural $D_{10}$-invariant tessellation.

By intersecting $H$ with the hyperplanes $x_i = t$ with $i=1,\ldots,5$ and $t \in \matZ$ we get a tessellation of $H$ into compact polytopes that is sometimes called the \emph{5-cell tessellation}. Its vertices form precisely the $A_4$ lattice. 
The 5-cell tessellation has many symmetries: the lattice $A_4$ acts by translations, the permutation group $S_5$ acts by permuting the coordinates, and we also have the involution $\iota(x) = -x$. All these symmetries generate a group that acts isometrically on the tessellation and transitively on its vertices, with stabiliser $S_5 \rtimes \matZ_2$.

There are two orbits of polytopes under this action, consisting respectively of simplexes and rectified simplexes. One example of simplex has vertices
$$(0,0,0,0,0), \quad (-1,1,0,0,0), \quad (-1,0,1,0,0), \quad (-1,0,0,1,0), \quad (-1,0,0,0,1)$$
and is bounded by the 5 hyperplanes $x_1=-1$ and $x_i=0$ for $i=1,\ldots,4$;
one example of rectified simplex has vertices
\begin{gather*}
(-2,1,1,0,0), (-2,1,0,1,0), (-2,1,0,0,1), (-2,0,1,1,0), (-2,0,1,0,1), \\
(-2,0,0,1,1), (-1,1,0,0,0), (-1,0,1,0,0), (-1,0,0,1,0), (-1,0,0,0,1)
\end{gather*}
and is bounded by the 10 hyperplanes $x_1=-2, x_1=-1$, $x_i=0, x_i=1$ for $i=1,\ldots,4$.
The 5-cell tessellation contains simplexes and rectified simplexes with ratio 1:1.

\subsection{Proof of Propositions \ref{RS:prop}, \ref{rst:prop} and \ref{summary:prop}}
The $A_4$ lattice acts on the 5-cell tessellation, which descends to a tessellation of the 4-torus $H/A_4$. We can check that it consists of 2 simplexes and 2 rectified simplexes. 
Since $\Gamma<A_4$ has index 5, the larger 4-torus $H / \Gamma= \sqrt 2 \cdot \T$ is tessellated in 10 simplexes and 10 rectified simplexes. 

The isometries $r$ and $s$ also act on the tessellation, and generate the isometry group $D_{10}$ acting on $H/ \Gamma$, whose quotient is $\sqrt 2\cdot \Orb$. 
We deduce that $\Orb$ is tessellated into one regular simplex and one rectified simplex. By direct inspection we see that these are combined as stated in Proposition \ref{RS:prop}, and that the symmetries $\rho, \sigma, \tau, \psi$ described explicitly in Section \ref{HGamma:subsection}
act as prescribed by Propositions \ref{rst:prop} and \ref{summary:prop}.

\section{The fibered hyperbolic 5-manifold} \label{RT:section}
In the previous sections we have constructed the fiber $F$ and the monodromy $\varphi\colon F \to F$. Here we furnish the interior of the mapping torus $N^5$ of $\varphi$ with a hyperbolic structure. 

\subsection{The mapping torus}
Recall that $\bar F$ is a flat cone 4-manifold with singular set the torus $T$ containing the five points $P_1,\ldots, P_5$.  As explained in the introduction, the pseudo-Anosov homeomorphism $\varphi \colon \bar F \to \bar F$ determines (after a small isotopy) a homeomorphism $\varphi \colon F \to F$. We define $N^5$ to be the mapping torus of this homeomorphism $\varphi \colon F \to F$.
Our aim is to prove that $N^5$ is hyperbolic, that is:

\begin{teo} \label{N:teo}
The interior of $N^5$ has a complete finite volume hyperbolic metric.
\end{teo}

\subsection{The ideal 5-dimensional hyperbolic cross-polytope $C$}
As stated in the introduction, the proof of Theorem \ref{N:teo} is very much analogous to the 3-dimensional construction depicted in Figure \ref{Gieseking:fig} of a hyperbolic structure on the Gieseking 3-manifold. 
In that picture we have two distinct isomorphic ideal triangulations $\Delta, \Delta'$ of the punctured torus, and we have connected them via a regular ideal hyperbolic tetrahedron. 

Analogously, we have constructed in Section \ref{two:subsection} two isomorphic triangulations $\Delta, \Delta'$, first for $\Orb$ and then lifted to $\bar F$, and we now connect them with a \emph{regular ideal hyperbolic  5-dimensional cross-polytope} $C$. The polytope $C$ is
the convex hull of the 10 ideal points $\pm e_1, \ldots, \pm e_5$ in the Klein model for $\matH^5$. It has 32 facets
$$F_{\pm 1, \pm 1, \pm 1, \pm 1, \pm 1} = \{\pm x_1  \pm x_2  \pm x_3  \pm x_4  \pm x_5 = 1\}.$$
This is an interesting hyperbolic polytope because its dihedral angles are $2\pi/3$. Indeed the link of an ideal vertex is a  Euclidean 4-dimensional cross-polytope, whose dihedral angle is known to be $2\pi/3$. The polytope $C$ is not a Coxeter polytope (the dihedral angles do not divide $\pi$), but by reflecting it along its facets we get a tessellation of $\matH^5$. This polytope is also described by Ratcliffe and Tschantz in \cite{RT5}.

\subsection{A cusped hyperbolic 5-orbifold}
We start by pairing isometrically the facets of $C$ following the instructions listed in Table \ref{pairing:table}. 
Here and below, an isometry between two facets of $C$ is indicated by a permutation $\sigma \in S_5$, which prescribes that the isometry should be the unique one 
between the facets that sends $\pm e_1, \ldots, \pm e_5$ to $\pm e_{\sigma(1)}, \ldots, \pm e_{\sigma(5)}$. We denote by $F_{\dots, 0, \ldots}$ the 3-face of $C$ that is the intersection of the facets $F_{\ldots, -1, \ldots}$ and $F_{\ldots, 1, \ldots}$. Such a 3-face is a regular ideal tetrahedron.

\begin{table}
\hrule
\begin{align*}
F_{1,1,1,1,1} & \stackrel{(34)(25)}\longrightarrow F_{-1,-1,-1,-1,-1} &
F_{1,-1,1,1,-1} & \stackrel{(2453)}\longrightarrow F_{1,-1,-1,-1,-1} \\
F_{1,-1,1,-1,1} & \stackrel{(2453)}\longrightarrow F_{-1,1,-1,-1,-1}&
F_{-1,1,1,-1,1} & \stackrel{(2453)}\longrightarrow  F_{-1,-1,1,-1,-1}\\
F_{-1,1,-1,1,1} & \stackrel{(2453)}\longrightarrow  F_{-1,-1,-1,1,-1}&
F_{1,1,-1,1,-1} & \stackrel{(2453)}\longrightarrow F_{-1,-1,-1,-1,1}\\
F_{1,-1,1,1,1} & \stackrel{(2453)}\longrightarrow  F_{1,1,-1,-1,-1} &
F_{1,1,1,-1,1} & \stackrel{(2453)}\longrightarrow  F_{-1,1,1,-1,-1}\\
F_{-1,1,1,1,1} & \stackrel{(2453)}\longrightarrow  F_{-1,-1,1,1,-1}&
F_{1,1,-1,1,1} & \stackrel{(2453)}\longrightarrow  F_{-1,-1,-1,1,1}\\
F_{1,1,1,1,-1} & \stackrel{(2453)}\longrightarrow  F_{1,-1,-1,-1,1}&
F_{1,1,1,-1,-1} & \stackrel{\id}\longrightarrow F_{1,-1,1,-1,-1}\\
F_{-1,1,1,1,-1} & \stackrel{\id}\longrightarrow F_{-1,1,-1,1,-1} &
F_{-1,-1,1,1,1} & \stackrel{\id}\longrightarrow F_{-1,-1,1,-1,1} \\
F_{1,-1,-1,1,1} & \stackrel{\id}\longrightarrow F_{1,-1,-1,1,-1} &
F_{1,1,-1,-1,1} & \stackrel{\id}\longrightarrow F_{-1,1,-1,-1,1} 
\end{align*}
\hrule
\caption{This face-pairing of the regular ideal hyperbolic 5-dimensional  cross-polytope $C$ produces a hyperbolic 5-orbifold.}\label{pairing:table}
\end{table}

\begin{prop} The resulting space is a hyperbolic 5-orbifold with singular set
$$\Sigma = F_{1,0,1,-1,-1} \cup F_{-1,1,0,1,-1} \cup \ldots \cup F_{0,1,-1,-1,1}$$
that is a totally geodesic hyperbolic 3-manifold with cone angle $2\pi/3$.
\end{prop}
\begin{proof}
The cross-polytope $C$ has $80$ three-dimensional faces, and the pairing partitions them into cycles. With some patience we may check that it produces 25 cycles of order 3 as in Table \ref{cycles:table}, plus 5 cycles of order 1 as in Table \ref{cycles2:table}.

\begin{table}
\hrule
\begin{gather*}
F_{0,-1,-1,-1,-1} \stackrel{(34)(25)}\longrightarrow F_{0,1,1,1,1} 
\stackrel{(2453)}\longrightarrow F_{0,-1,1,1,-1}
\stackrel{(2453)}\longrightarrow F_{0,-1,-1,-1,-1} \\
F_{-1,0,-1,-1,-1} \stackrel{(34)(25)}\longrightarrow F_{1,1,1,1,0} 
\stackrel{(2453)}\longrightarrow F_{1,-1,0,-1,1}
\stackrel{(2453)}\longrightarrow F_{-1,0,-1,-1,-1} \\
\vdots \\
F_{-1,-1,-1,-1,0} \stackrel{(34)(25)}\longrightarrow F_{1,0,1,1,1} 
\stackrel{(2453)}\longrightarrow F_{1,1,-1,0,-1}
\stackrel{(2453)}\longrightarrow F_{-1,-1,-1,-1,0} \\
F_{1,0,-1,-1,-1} \stackrel{(3542)}\longrightarrow F_{1,-1,0,1,-1}
\stackrel{\id} \longrightarrow F_{1,-1,0,1,1}
\stackrel{(2453)} \longrightarrow F_{1,0,-1,-1,-1} \\
F_{-1,1,0,-1,-1} \stackrel{(3542)}\longrightarrow F_{1,-1,1,-1,0}
\stackrel{\id} \longrightarrow F_{1,1,1,-1,0}
\stackrel{(2453)} \longrightarrow F_{-1,1,0,-1,-1} \\
\vdots \\
F_{0,-1,-1,-1,1} \stackrel{(3542)}\longrightarrow F_{0,1,-1,1,-1}
\stackrel{\id} \longrightarrow F_{0,1,1,1,-1}
\stackrel{(2453)} \longrightarrow F_{0,-1,-1,-1,1} \\
F_{1,-1,0,-1,-1} \stackrel{(3542)}\longrightarrow F_{1,-1,1,1,0}
\stackrel{(2453)} \longrightarrow F_{1,1,0,-1,-1}
\stackrel{\id} \longrightarrow F_{1,-1,0,-1,-1} \\
F_{-1,1,-1,0,-1} \stackrel{(3542)}\longrightarrow F_{1,0,1,-1,1}
\stackrel{(2453)} \longrightarrow F_{-1,1,1,0,-1}
\stackrel{\id} \longrightarrow F_{-1,1,-1,0,-1} \\
\vdots \\
F_{-1,0,-1,-1,1} \stackrel{(3542)}\longrightarrow F_{1,1,0,1,-1}
\stackrel{(2453)} \longrightarrow F_{1,0,-1,-1,1}
\stackrel{\id} \longrightarrow F_{-1,0,-1,-1,1} \\ 
F_{1,-1,-1,0,-1} \stackrel{(3542)}\longrightarrow F_{1,0,1,1,-1}
\stackrel{(2453)} \longrightarrow F_{1,-1,-1,0,1}
\stackrel{\id} \longrightarrow F_{1,-1,-1,0,-1} \\
F_{-1,1,-1,-1,0} \stackrel{(3542)}\longrightarrow F_{1,-1,1,0,1}
\stackrel{(2453)} \longrightarrow F_{1,1,-1,-1,0}
\stackrel{\id} \longrightarrow F_{-1,1,-1,-1,0} \\
\vdots \\
F_{-1,-1,0,-1,1} \stackrel{(3542)}\longrightarrow F_{1,1,-1,1,0}
\stackrel{(2453)} \longrightarrow F_{-1,-1,0,1,1}
\stackrel{\id} \longrightarrow F_{-1,-1,0,-1,1} \\
F_{1,-1,-1,-1,0} \stackrel{(3542)}\longrightarrow F_{1,-1,1,0,-1}
\stackrel{\id} \longrightarrow F_{1,1,1,0,-1}
\stackrel{(2453)} \longrightarrow F_{1,-1,-1,-1,0} \\
F_{0,1,-1,-1,-1} \stackrel{(3542)}\longrightarrow F_{0,-1,1,-1,1}
\stackrel{\id} \longrightarrow F_{0,-1,1,1,1}
\stackrel{(2453)} \longrightarrow F_{0,1,-1,-1,-1} \\
\vdots \\
F_{-1,-1,-1,0,1} \stackrel{(3542)}\longrightarrow F_{1,0,-1,1,-1}
\stackrel{\id} \longrightarrow F_{1,0,-1,1,1}
\stackrel{(2453)} \longrightarrow F_{-1,-1,-1,0,1} 
\end{gather*}
\hrule
\vspace{.35 cm}
\caption{The 25 cycles of 3-faces of order 3.}\label{cycles:table}
\end{table}

\begin{table}
\hrule
\begin{align*}
F_{1,0,1,-1,-1} & \stackrel \id \longrightarrow F_{1,0,1,-1,-1} \\
F_{-1,1,0,1,-1} & \stackrel \id \longrightarrow F_{-1,1,0,1,-1} \\
& \vdots \\
F_{0,1,-1,-1,1} & \stackrel \id \longrightarrow F_{0,1,-1,-1,1}
\end{align*}
\hrule
\caption{The 5 cycles of 3-faces of order 1.}\label{cycles2:table}
\end{table}

The dihedral angles in the 25 cycles of order 3 sum nicely to $2\pi$. The 3-faces of the 5 cycles of order 1 form the singular set $\Sigma$, which turns out to be a hyperbolic 3-manifold tessellated into 5 regular ideal tetrahedra. Its cone angle is $2\pi/3$.
\end{proof}

\begin{table}
\hrule
\begin{align*}
F^{a}_{1,1,1,1,1} & \stackrel{(34)(25)}\longrightarrow F^{a-1}_{-1,-1,-1,-1,-1} &
F^{a}_{1,-1,1,1,-1} & \stackrel{(2453)}\longrightarrow F^{1-a}_{1,-1,-1,-1,-1} \\
F^{a}_{1,-1,1,-1,1} & \stackrel{(2453)}\longrightarrow F^{1-a}_{-1,1,-1,-1,-1}&
F^{a}_{-1,1,1,-1,1} & \stackrel{(2453)}\longrightarrow  F^{1-a}_{-1,-1,1,-1,-1}\\
F^{a}_{-1,1,-1,1,1} & \stackrel{(2453)}\longrightarrow  F^{1-a}_{-1,-1,-1,1,-1}&
F^{a}_{1,1,-1,1,-1} & \stackrel{(2453)}\longrightarrow F^{1-a}_{-1,-1,-1,-1,1}\\
F^{a}_{1,-1,1,1,1} & \stackrel{(2453)}\longrightarrow  F^{2-a}_{1,1,-1,-1,-1} &
F^{a}_{1,1,1,-1,1} & \stackrel{(2453)}\longrightarrow  F^{2-a}_{-1,1,1,-1,-1}\\
F^{a}_{-1,1,1,1,1} & \stackrel{(2453)}\longrightarrow  F^{2-a}_{-1,-1,1,1,-1}&
F^{a}_{1,1,-1,1,1} & \stackrel{(2453)}\longrightarrow  F^{2-a}_{-1,-1,-1,1,1}\\
F^{a}_{1,1,1,1,-1} & \stackrel{(2453)}\longrightarrow  F^{2-a}_{1,-1,-1,-1,1}&
F^{a}_{1,1,1,-1,-1} & \stackrel{\id}\longrightarrow F^{a-1}_{1,-1,1,-1,-1}\\
F^{a}_{-1,1,1,1,-1} & \stackrel{\id}\longrightarrow F^{a-1}_{-1,1,-1,1,-1} &
F^{a}_{-1,-1,1,1,1} & \stackrel{\id}\longrightarrow F^{a-1}_{-1,-1,1,-1,1} \\
F^{a}_{1,-1,-1,1,1} & \stackrel{\id}\longrightarrow F^{a-1}_{1,-1,-1,1,-1} &
F^{a}_{1,1,-1,-1,1} & \stackrel{\id}\longrightarrow F^{a-1}_{-1,1,-1,-1,1} 
\end{align*}
\hrule
\caption{This face-pairing of $C_1, C_2, C_3$ produces a hyperbolic 5-manifold $N^5$.}\label{pairing2:table}
\end{table}

We can now easily construct a hyperbolic 5-manifold three-fold cover of this hyperbolic 5-orbifold. Pick three copies $C_1, C_2, C_3$ of $C$. Let $F^a_{\pm 1, \pm 1, \pm 1, \pm 1, \pm 1}$ denote the facets of $C_a$. We pair isometrically these facets as prescribed in Table \ref{pairing2:table}. 

\begin{prop}
The result is a complete finite-volume hyperbolic 5-manifold.
\end{prop}
\begin{proof}
We can verify that the cycles are as those of Table \ref{cycles:table}, each repeated three times, plus those of Table \ref{cycles2:table} with triple length: so they all have length 3 and the dihedral angles sum to $2\pi$ everywhere.
\end{proof}

\subsection{Proof of Theorem \ref{N:teo}}
We prove that the interior of $N^5$ is homeomorphic to the hyperbolic manifold just constructed by gluing three copies $C_1,C_2,C_3$ of $C$.

Recall the two triangulations $R^\Delta$ and $R^{\Delta'}$ of the rectified simplex $R$, each consisting of 11 simplexes. We now define two simplicial embeddings 
\begin{equation} \label{i:eqn}
i\colon R^\Delta \hookrightarrow \partial C, \qquad i'\colon R^{\Delta'} \hookrightarrow \partial C.
\end{equation}
We first send bijectively the 10 vertices of $R$ to the 10 vertices of $C$ as follows:
\begin{align*}
(1,1,0,0,0) & \longmapsto -e_4 \qquad (0,0,1,0,1) \longmapsto e_4 \\
(0,1,1,0,0) & \longmapsto -e_5 \qquad (1,0,0,1,0) \longmapsto e_5 \\
(0,0,1,1,0) & \longmapsto -e_1 \qquad (0,1,0,0,1) \longmapsto e_1 \\
(0,0,0,1,1) & \longmapsto -e_2 \qquad (1,0,1,0,0) \longmapsto e_2 \\
(1,0,0,0,1) & \longmapsto -e_3 \qquad (0,1,0,1,0) \longmapsto e_3 
\end{align*}
and then extend this map to two simplicial maps $i,i'$. The maps send a vertex labeled with $i$ to $\pm e_i$. The central simplexes $S_-, S_-'$ of $R^{\Delta}$ and $R^{\Delta'}$ are sent to the opposite facets $F_{1,1,1,1,1}$ and $F_{-1,-1,-1,-1,-1}$ of $C$. The remaining 10 facets $S^1_-,\ldots,S^5_-,S^1_+,\ldots,S^5_+$ of $R^\Delta$ are sent to
\begin{gather*}
F_{-1,1,1,1,1}, \ F_{1,-1,1,1,1},\ \ldots,\ F_{1,1,1,1,-1}, \\
F_{1,-1,1,1,-1}, \ F_{-1,1,-1,1,1},\ \ldots,\ F_{-1,1,1,-1,1}.
\end{gather*}
and the remaining 10 facets $S^{1'}_-,\ldots,S^{5'}_-,S^{1'}_+,\ldots,S^{5'}_+$ of $R^{\Delta'}$ are sent to
\begin{gather*} 
F_{1,-1,-1,-1,-1}, \ F_{-1,1,-1,-1,-1},\ \ldots,\ F_{-1,-1,-1,-1,1}, \\
F_{-1,-1,1,1,-1}, \ F_{-1,-1,-1,1,1},\ \ldots,\ F_{-1,1,1,-1,-1}
\end{gather*}
We identify $R^\Delta$ and $R^{\Delta'}$ with their images in $\partial C$ via $i$ and $i'$. Now
$R^\Delta$ and $R^{\Delta'}$ are two 4-discs in $\partial C$, with disjoint interiors, whose boundaries intersect in 5 tetrahedra
\begin{equation} \label{tetrahedra:eqn}
F_{-1,0,-1,1,1}, \ F_{1,-1,0,-1,1}, \ldots, F_{0,-1,1,1,-1}.
\end{equation}
These 5 tetrahedra are the images of the tetrahedral facets $x_i=1$ in $R^\Delta$ and $R^{\Delta'}$.

We have constructed a \emph{cobordism} between two triangulations $R^\Delta$ and $R^{\Delta'}$ of the same Euclidean polyhedron $R$, by means of a 5-dimensional hyperbolic polyhedron $C$, similar to Figure \ref{Gieseking:fig}. A notable difference here is that our cobordism $C$ has also some \emph{vertical} facets: it has 11 \emph{horizontal} facets at the \emph{top} in $R^\Delta$, 11 more horizontal facets at the \emph{bottom} in $R^{\Delta'}$, and 10 more \emph{vertical} facets
\begin{gather*}
F_{1,-1,1,-1,-1}, \quad F_{-1,1,-1,1,-1}, \ldots, \quad F_{-1,1,-1,-1,1} \\
F_{1,1,1,-1,-1}, \quad F_{-1,1,1,1,-1}, \ldots, \quad F_{1,1,-1,-1,1}.
\end{gather*}

The total number of facets of $C$ is $11+11+10=32$. We now take three copies $C_1, C_2, C_3$ of $C$, and glue them along their vertical facets via the last 5 maps of Table \ref{pairing2:table}. This produces a hyperbolic 5-manifold with corners $C^*$, that contains only horizontal facets. Moreover one checks that these maps glue the facets $S^i_{a,-} \subset R_a^\Delta$ and $S^{i\mp 2}_{a-1,+} \subset R_{a-1}^\Delta$ precisely as in \eqref{pairing:eqn}; therefore the top horizontal facets $R_1^\Delta \cup R_2^\Delta \cup R_3^\Delta$
of $C_1\cup C_2 \cup C_3$ are glued in $C^*$ like in the triangulation $\Delta$ of $\bar F$, that is the top facets of $C^*$ form in fact the triangulation $\Delta$ minus $S_1 \cup S_2 \cup S_3$. Analogously the bottom facets of $C^*$ are ${\Delta'}$ minus $S_1 \cup S_2 \cup S_3$.
 
The boundary of the missing $S_a$ in both the bottom and top triangulations of $C^*$ consists of the 5 facets \eqref{tetrahedra:eqn} in $\partial C_a$. So we attach abstractly $S_a$ to these facets to fill these gaps: the result is $C^{**} = C^* \cup S_1 \cup S_2 \cup S_3$, where $S_1\cup S_2 \cup S_3$ should be considered as a thin 4-dimensional part lying both at the bottom and at the top of $C^{**}$. The bottom and top facets of $C^{**}$ are now respectively $\Delta'$ and $\Delta$, intersecting in the thin part $S_1 \cup S_2\cup S_3$ that belongs to both the bottom and the top. We have constructed a cobordism $C^{**}$ from $\Delta'$ to $\Delta$, that is everywhere fully 5-dimensional except on the thin 4-dimensional part $S_1 \cup S_2 \cup S_3$.

Recall that $\varphi = \tau\psi$ and $\varphi(\Delta) = \Delta'$. To construct the interior of the mapping torus $N^5$ of $\varphi$ it suffices to remove the vertices from $C^{**}$ and glue the top and bottom facets via $\varphi$. We can then verify using Proposition \ref{varphi:prop} that we get precisely the facets paired as in Table \ref{pairing2:table}, and hence the interior of $N^5$ is the hyperbolic 5-manifold constructed above. This concludes the proof.

A particular case is the top facet $F_{1,1,1,1,1}^a = S_{a,-}$ of $\Delta$, that is sent via $\varphi$ to $S'_{2-a,+}=S_{2-a}$, considered as a bottom facet. Therefore $\varphi$ glues $F^a_{1,1,1,1,1}$ to the thin simplex $S_{2-a}$. Moreover, $S_{2-a}=S_{2-a,+}$ considered as a top facet is sent via $\varphi$ to the bottom facet $S'_{a-1,-} = F_{-1,-1,-1,-1,-1}^{a-1}$. Summing all, we glue $F_{1,1,1,1,1}^a$ to $F_{-1,-1,-1,-1,-1}^{a-1}$, via the permutation $(\varphi_*)^2 = (34)(25)$.

\subsection{Is it the Ratcliffe -- Tschantz manifold?}
We think so. In fact, the Ratcliffe -- Tschantz hyperbolic 5-manifold decomposes as a union $C_1 \cup C_2 \cup C_3$ of three copies of the regular ideal cross-polytope $C$ as follows: the manifold is obtained by attaching two copies $P^5_1, P^5_2$ of the right-angled polytope $P^5$ according to \cite[Table 1]{IMM}, and the convex hull of the ideal points of each $P^5_a$ is a regular cross-polytope $C_a \subset P^5_a$. Moreover, what is left is a third cross-polytope $C_3$.

Unfortunately, this decomposition into $C_1 \cup C_2 \cup C_3$ is not isomorphic to the one we provided here. However, we think that there is a move connecting the two decompositions. We do not pursue this argument further here.

\section{Topological invariants of the fiber} \label{topological:section}
We now compute some topological invariants of the fiber $F$ and the action of the monodromy $\varphi$ on these. We prove in particular Theorems \ref{pi:F:teo} and \ref{H:F:teo}.

\subsection{The strata of the ideal tessellation $\Pi$}
In Section \ref{Pi:subsection} we constructed a tessellation $\Pi$ of $\bar F$, which
can also be interpreted as an \emph{ideal tessellation} for $F$.
We know that $\Pi$ is made of 3 simplexes $S_1, S_2, S_3$ and 3 rectified simplexes $R_1, R_2, R_3$. By analysing the orbits of all the strata produced by the gluings, we discover that the strata of $\Pi$ are:
\begin{itemize}
\item 5 ideal vertices $P_1,\ldots, P_5$;
\item 10 edges $e_{ij}$ with $\{i,j\} \subset \{1,\ldots, 5\}$, and
$$\partial e_{ij} = P_i \cup P_j;$$
\item 30 triangles $T_a^{ijk}$ with $\{i,j,k\} \subset \{1,\ldots, 5\}$, $a \in \{1,2,3\}$, and
$$\partial T_a^{ijk} = e_{ij} \cup e_{jk} \cup e_{ki};$$
\item 5 squares $Q_i$ with $i \in \{1,\ldots,5\}$ and 
$$\partial Q_i = e_{i+1,i+2} \cup e_{i+2,i+4} \cup e_{i+4,i+3} \cup e_{i+3,i+1};$$
\item 15 tetrahedra $T_a^{ijkl}$ with $\{i,j,k,l\} \subset \{1,\ldots, 5\}$, $a\in \{1,2,3\}$, and
$$\partial T_a^{ijkl} = T_a^{ijk} \cup T_a^{jkl} \cup T_a^{kli} \cup T_a^{lij};$$
\item 15 pyramids $P_a^i$ with $i \in \{1,\ldots, 5\}$, $a\in \{1,2,3\}$, and
$$\partial P_a^i = Q_i \cup T^{i,i+1,i+2}_{a+1} \cup T^{i,i+2,i+4}_{a-1} \cup T^{i,i+4,i+3}_{a+1} \cup T^{i,i+3,i+1}_{a-1};$$
\item 3 simplexes $S_a$ with $a \in \{1,2,3\}$, and
$$\partial S_a = T^{1234}_a \cup T^{2345}_a \cup T^{3451}_a \cup T^{4512}_a \cup T^{5123}_a.$$
\item 3 rectified simplexes $R_a$ with $a\in \{1,2,3\}$, and
\begin{align*}
\partial R_a & = T^{1234}_a \cup T^{2345}_a \cup T^{3451}_a \cup T^{4512}_a \cup T^{5123}_a \cup \\
& \qquad P_{a+1}^1 \cup P_{a+1}^2 \cup P_{a+1}^3 \cup P_{a+1}^4 \cup P_{a+1}^5 \cup
P_{a-1}^1 \cup P_{a-1}^2 \cup P_{a-1}^3 \cup P_{a-1}^4 \cup P_{a-1}^5.
\end{align*}
\end{itemize}
We have
$$\chi(F) = -10 +35-30+6 = 1.$$

The 5 squares $Q_i$ form the torus $T$ as in Figure
\ref{torus2:fig}.

\subsection{Symmetries of $\Pi$} \label{G:subsection}
It is shown in Section \ref{Pi:subsection} that $\Pi$ has 60 symmetries, that form the group $G < S_3 \times S_5$ generated by the elements
$$\rho = (\id, (12345)), \quad \sigma = (\id, (34)(52)), \quad
\tau = ((23), (2453)), \quad \phi = ((123), \id).$$ 

Each $\alpha = (\alpha_1, \alpha_2) \in G$ modifies the indices of the strata as follows:
\begin{gather*}
\alpha(P_i) = P_{i'}, \quad \alpha (e_{ij}) = e_{i'j'}, \quad
\alpha (Q_i) = Q_{i'}, \\
 \alpha(T_a^{ijk}) = T_{a'}^{i'j'k'}, \quad
 \alpha(T_a^{ijkl}) = T_{a'}^{i'j'k'l'} , \quad \alpha(S_a) = S_{a'}, \quad \alpha(R_a) = R_{a'}
\end{gather*}
where $a' = \alpha_1(a)$ and $i' = \alpha_2(i),j'=\alpha_2(j),k'=\alpha_2(k),l'=\alpha_2(l)$.

\subsection{The dual spine $X$}
We can easily construct a \emph{spine} $X$ of the fiber $F$ by dualising the ideal tessellation: we take the barycenter subdivision of the tessellation, and then we remove the open stars of the ideal vertices. The spine $X$ is a 3-dimensional object, and each stratum $S$ of the tessellation is dual to a stratum $S^*$ of $X$. The spine $X$ consists of:

\begin{itemize}
\item 6 vertices $S_1^*, S_2^*, S_3^*, R_1^*, R_2^*, R_3^*$;
\item 30 edges $(T_a^{ijkl})^*$ and $(P_a^i)^*$;
\item 30 squares $(T_a^{ijk})^*$ and 5 triangles $Q_i^*$;
\item 10 polyhedra $e_{ij}^*$.
\end{itemize}

Note that the dual of a square is a triangle, and the dual of a triangle is a square (because every square $Q_i$ is contained in three pyramids, while every triangle $T_a^{ijk}$ is contained in two tetrahedra and two pyramids).
The manifold $F$ collapses onto the spine $X$. We describe more explicitly the 2-skeleton $X^2$ of the spine: for simplicity, denote the 30 edges of $X$ as $e_a^m = (T_a^{ijkl})^*$ and $\hat e_a^i = (P_a^i)^*$, where $m$ is such that $\{i,j,k,l,m\}= \{1,2,3,4,5\}$. We have 
$$\partial e_a^i = R_a^* \cup S_a^*, \qquad
\partial \hat e_a^i = R_{a-1}^* \cup R_{a+1}^*.
$$

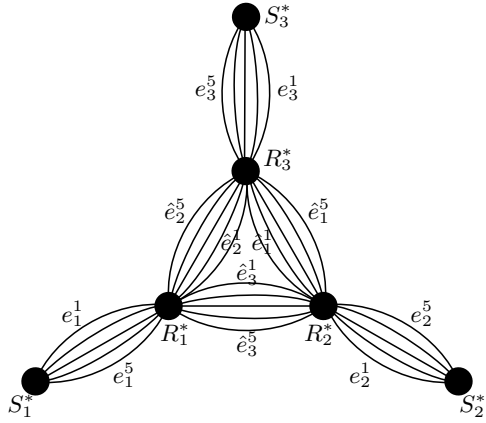
\begin{figure}
\centering
\vspace{.2 cm}
\labellist
\small\hair 2pt
\pinlabel $S_1^*$ at 0 -10
\pinlabel $S_2^*$ at 440 -10
\pinlabel $S_3^*$ at 250 370
\pinlabel $R_1^*$ at 150 60
\pinlabel $R_2^*$ at 290 60
\pinlabel $R_3^*$ at 250 230
\pinlabel $e_1^1$ at 50 80
\pinlabel $e_1^5$ at 100 20
\pinlabel $e_2^5$ at 390 80
\pinlabel $e_2^1$ at 330 20
\pinlabel $e_3^1$ at 260 300
\pinlabel $e_3^5$ at 180 300
\pinlabel $\hat e_3^1$ at 220 120
\pinlabel $\hat e_3^5$ at 220 50
\pinlabel $\hat e_1^1$ at 235 150
\pinlabel $\hat e_1^5$ at 290 180
\pinlabel $\hat e_2^1$ at 205 150
\pinlabel $\hat e_2^5$ at 150 180
\endlabellist
\includegraphics[width=6 cm]{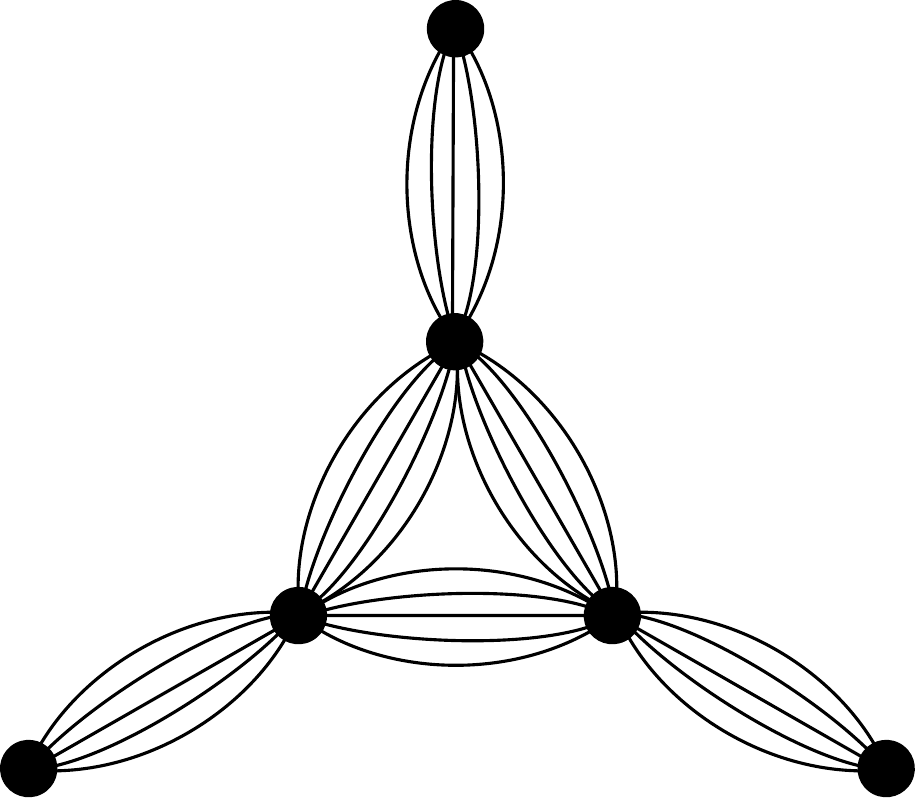}
\caption{The 1-skeleton $X^1$ of the spine $X$.}\label{spine:fig}
\end{figure}

\begin{figure}
\centering
\vspace{.2 cm}
\labellist
\small\hair 2pt
\pinlabel $Q_i^*$ at 65 45
\pinlabel $R_1^*$ at -5 -10
\pinlabel $R_2^*$ at 130 -10
\pinlabel $R_3^*$ at 64 125
\pinlabel $\hat e_3^i$ at 64 -5
\pinlabel $\hat e_1^i$ at 104 60
\pinlabel $\hat e_2^i$ at 23 60
\pinlabel $(T_a^{lmi})^*$ at 250 60
\pinlabel $R_{a+1}^*$ at 185 -10
\pinlabel $R_{a}^*$ at 310 -10
\pinlabel $R_{a}^*$ at 180 125
\pinlabel $S_{a}^*$ at 310 125
\pinlabel $\hat e_{a-1}^l$ at 245 -5
\pinlabel $\hat e_{a-1}^i$ at 175 60
\pinlabel $e_a^k$ at 245 123
\pinlabel $e_a^j$ at 310 60
\pinlabel $(T_a^{jkm})^*$ at 434 60
\pinlabel $R_{a-1}^*$ at 370 -10
\pinlabel $R_{a}^*$ at 495 -10
\pinlabel $R_{a}^*$ at 365 125
\pinlabel $S_{a}^*$ at 495 125
\pinlabel $\hat e_{a+1}^k$ at 430 -5
\pinlabel $\hat e_{a+1}^j$ at 362 60
\pinlabel $e_a^i$ at 430 123
\pinlabel $e_a^l$ at 495 60
\endlabellist
\includegraphics[width=10 cm]{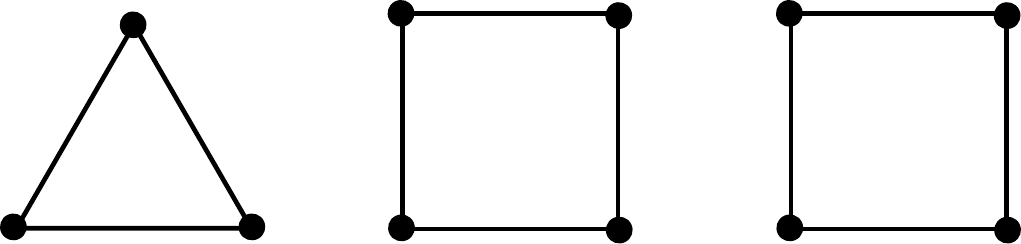}
\vspace{.2 cm}
\caption{The 35 two-cells of the spine $X$. There are 5 triangles $Q_i^*$ parametrized by $i\in \{1,2,3,4,5\}$. We suppose that $(i,j,k,l,m)$ is a cyclic permutation of $(1,2,3,4,5)$ and $a \in \{1,2,3\}$, so there are 15 squares of each kind $(T_a^{lmi})^*$ and $(T_a^{jkm})^*$, and 30 in total.}\label{2cells:fig}
\end{figure}

The 1-skeleton $X^1$ is in Figure \ref{spine:fig}. The 5 triangles $Q_i^*$ and the 30 squares $(T_a^{ijk})^*$ should be attached to $X^1$ as indicated in Figure \ref{2cells:fig}. There are two types of squares, depending on whether the indices $i,j,k$ are consecutive or not: we suppose in the notation that $(i,j,k,l,m)$ is a cyclic permutation of $(1,2,3,4,5)$, so the two types are $(T_a^{ijk})^*$ and $(T_a^{ijl})^*$, and there are 15 squares of each type as in the figure.

\subsection{The ideal triangulation $\Delta$}
As explained in Section \ref{Delta:subsection}, the ideal tessellation $\Pi$ can be subdivided into an ideal triangulation $\Delta$ for $F$ with $36$ simplices. 
By direct inspection we find that the the strata of $\Delta$ are:

\begin{itemize}
\item 5 ideal vertices $P_1,\ldots, P_5$;
\item 15 edges $e_{i,i+1}^+, e_{i,i+1}^-, e_{i,i+2}$ with $i\in \{1,\ldots, 5\}$, and
$$\partial e_{i,i+1}^+ = \partial e_{i,i+1}^- = P_i\cup P_{i+1}, \quad 
\partial e_{i,i+2} = P_i \cup P_{i+2};$$
\item 10 triangles $T^{i,i+1,i+2}_\pm$, 30 triangles $T^{i,i+1,i+2}_{a,\pm}$, and 30 triangles $T^{i,i+1,i+3}_{a,\pm}$ with $i\in \{1,\ldots,5\}$, $a\in\{1,2,3\}$, and
\begin{align*}
\partial T^{i,i+1,i+2}_\pm &  = e_{i,i+1}^\pm \cup e_{i+1,i+2}^\mp \cup e_{i+2,i}, \\
\partial T^{i,i+1,i+2}_{a,\pm} & = e_{i,i+1}^\pm \cup e_{i+1,i+2}^\pm \cup e_{i+2,i}, \\
\partial T^{i,i+1,i+3}_{a,\pm}&  = e_{i,i+1}^\pm \cup e_{i+1,i+3} \cup e_{i+3,i};
\end{align*}
\item 90 tetrahedra $T^{i}_{a,s_1,s_2,s_3}$ where $i\in \{1,\ldots,5\}, a \in \{1,2,3\}$ and $s_j = \pm$ is a sign, with $(s_1,s_2,s_3) \neq (\pm, \mp, \pm)$ and 
\begin{align*}
\partial T^{i}_{a,-,-,+} & = T^{i,i+1,i+2}_{a+1,-} \cup T^{i,i+1,i+3}_{a-1,-} \cup T^{i,i+2,i+3}_{a,+} \cup T^{i+1,i+2,i+3}_-, \\
\partial T^{i}_{a,+,+,-} & = T^{i,i+1,i+2}_{a+1,+} \cup T^{i,i+1,i+3}_{a,+} \cup T^{i,i+2,i+3}_{a,-} \cup T^{i+1,i+2,i+3}_+, \\
\partial T^{i}_{a,+,-,-} & = T^{i,i+1,i+2}_{+} \cup T^{i,i+1,i+3}_{a,+} \cup T^{i,i+2,i+3}_{a-1,-} \cup T^{i+1,i+2,i+3}_{a+1,-}, \\
\partial T^{i}_{a,-,+,+} & = T^{i,i+1,i+2}_{-} \cup T^{i,i+1,i+3}_{a,-} \cup T^{i,i+2,i+3}_{a,+} \cup T^{i+1,i+2,i+3}_{a+1,+}, \\
\partial T^{i}_{a,-,-,-} & = T^{i,i+1,i+2}_{a,-} \cup T^{i,i+1,i+3}_{a,-} \cup T^{i,i+2,i+3}_{a,-} \cup T^{i+1,i+2,i+3}_{a,-}, \\
\partial T^{i}_{a,+,+,+} & = T^{i,i+1,i+2}_{a,+} \cup T^{i,i+1,i+3}_{a+1,+} \cup T^{i,i+2,i+3}_{a+1,+} \cup T^{i+1,i+2,i+3}_{a,+};
\end{align*}
\item 6 simplexes $S_{a,\pm}$ and 30 simplexes $S_{a,\pm}^i$ where $i\in \{1,\ldots, 5\}, a \in \{1,2,3\}$, 
\begin{align*}
\partial S_{a,\pm} & = T_{a,\pm,\pm,\pm}^1 \cup T_{a,\pm,\pm,\pm}^2 \cup T_{a,\pm,\pm,\pm}^3 \cup T_{a,\pm,\pm,\pm}^4 \cup T_{a,\pm,\pm,\pm}^5, \\
\partial S_{a,-}^i & = T_{a, -, -, -}^i \cup T_{a-1,-,-,+}^{i+1} \cup T_{a,-,+,+}^{i+2} \cup T_{a,+,+,-}^{i+3} \cup T_{a-1,+,-,-}^{i+4}, \\
\partial S_{a,+}^i & = T_{a, +, +, +}^i \cup T_{a-1,+,+,-}^{i+1} \cup T_{a+1,+,-,-}^{i+2} \cup T_{a+1,-,-,+}^{i+3} \cup T_{a-1,-,+,+}^{i+4}.
\end{align*}
\end{itemize}

We confirm that $\chi(F) = -15+70-90+36 = 1$.
The 1-skeleton of $\Pi$ is contained in the 1-skeleton of $\Delta$ as shown in Figure \ref{ideal:fig}.
The ideal triangulation $\Delta$ is obtained by subdividing the ideal tessellation $\Pi$ as follows: 
\begin{itemize}
\item 
The 15 edges of $\Delta$ are the 10 edges $e_{i,i+1}^- = e_{i,i+1}$ and $e_{i,i+2}$ of $\Pi$, plus the 5 diagonals $e_{i,i+1}^+$ of the squares $Q_{i-2}$, see Figure \ref{ideal:fig};
\item 
The 70 triangles of $\Delta$ are the 30 triangles $T_{a,-}^{ijk} = T_a^{ijk}$ of $\Pi$, plus the 10 triangles $T_\pm^{i,i+1,i+2}$ obtained by subdividing the squares $Q_j$, the 15 triangles $T_{a,+}^{i,i+1,i+3}$ that subdivide the pyramids, and the 15 triangles $T_{a,+}^{i,i+1,i+2}$ that lie in the interior of a rectified simplex $R_a$;
\item The 90 tetrahedra of $\Delta$ are the 15 tetrahedra $T^i_{a,-,-,-}=T^{jklm}_a$ of $\Pi$, plus 30 tetrahedra $T^i_{a,+,-,-}, T^i_{a,-,+,-}$ obtained by subdividing the pyramids, plus 45 more that lie in the interior of some $R_a$.
\end{itemize}

\begin{figure}
\centering
\vspace{.2 cm}
\labellist
\small\hair 2pt
\pinlabel $P_1$ at 22 10
\pinlabel $P_2$ at 125 10
\pinlabel $P_3$ at 150 100
\pinlabel $P_4$ at 80 150
\pinlabel $P_5$ at -7 100
\pinlabel $P_1$ at 230 10
\pinlabel $P_2$ at 333 10
\pinlabel $P_3$ at 358 100
\pinlabel $P_4$ at 283 153
\pinlabel $P_5$ at 201 100
\pinlabel $e_{12}$ at 72 5
\pinlabel $e_{23}$ at 140 50
\pinlabel $e_{34}$ at 120 125
\pinlabel $e_{45}$ at 23 125
\pinlabel $e_{51}$ at 5 50
\pinlabel $e_{13}$ at 60 25
\pinlabel $e_{14}$ at 30 55
\pinlabel $e_{24}$ at 113 55
\pinlabel $e_{25}$ at 85 25
\pinlabel $e_{35}$ at 72 88
\pinlabel $e_{12}^+$ at 280 -7
\pinlabel $e_{12}^-$ at 280 27
\pinlabel $e_{13}$ at 261 42
\pinlabel $e_{14}$ at 238 55
\pinlabel $e_{24}$ at 321 55
\pinlabel $e_{25}$ at 298 42
\pinlabel $e_{35}$ at 280 88
\pinlabel $e_{23}^-$ at 328 70
\pinlabel $e_{23}^+$ at 360 50
\pinlabel $e_{34}^+$ at 337 137
\pinlabel $e_{34}^-$ at 308 115
\pinlabel $e_{45}^+$ at 223 137
\pinlabel $e_{45}^-$ at 252 117
\pinlabel $e_{51}^+$ at 202 50
\pinlabel $e_{51}^-$ at 230 70
\endlabellist
\includegraphics[width=8 cm]{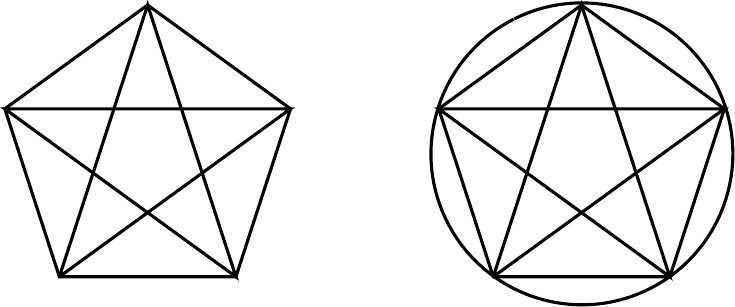}
\caption{The 1-skeleta of the ideal tessellation $\Pi$ and of the ideal triangulation $\Delta$. The latter contains the former, since $e_{i,i+1}^- = e_{i,i+1}$ for every $i$.}\label{ideal:fig}
\end{figure}
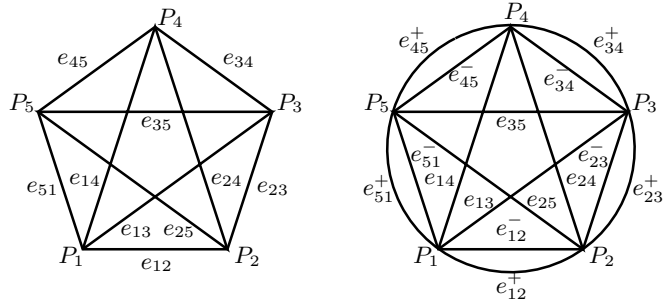

\subsection{Symmetries of $\Delta$}
As explained in Section \ref{Delta:subsection}, the ideal triangulation $\Delta$ has 60 symmetries, which form the group $G' = \matZ_6 \times D_{10}$. The group $\matZ_6$ is generated by the order-6 symmetry $\psi$ that 
acts on all the simplexes of $\Delta$ as follows:
\begin{itemize}
\item It fixes the ideal vertices $P_1,\ldots,P_5$ and $e_{i,i+2}$, and sends $e_{i,i+1}^\pm$ to $e_{i,i+1}^\mp$;
\item It acts on the triangles as follows: 
\begin{align*}
T^{i,i+1,i+2}_\pm & \longrightarrow T^{i,i+1,i+2}_\mp, \\
T^{i,i+1,i+2}_{a,-} & \longrightarrow T^{i,i+1,i+2}_{a,+} \longrightarrow T^{i,i+1,i+2}_{a+1,-}, \\ T^{i,i+1,i+3}_{a,-} & \longrightarrow T^{i,i+1,i+3}_{a+1,+} \longrightarrow T^{i,i+1,i+3}_{a+1,-};
\end{align*}
\item It acts on the tetrahedra as follows:
\begin{align*}
T^i_{a,-,-,+} & \longrightarrow T^i_{a,+,+,-} \longrightarrow T^i_{a+1,-,-,+} \\
T^i_{a,+,-,-} & \longrightarrow T^i_{a,-,+,+} \longrightarrow T^i_{a+1,+,-,-} \\
T^i_{a,-,-,-} & \longrightarrow T^i_{a,+,+,+} \longrightarrow T^i_{a+1,-,-,-} 
\end{align*}
\item It acts on the 4-simplexes as follows:
$$S_{a,-} \longrightarrow S_{a,+} \longrightarrow S_{a+1,-}, \qquad 
S_{a,-}^i \longrightarrow S_{a,+}^i \longrightarrow S_{a+1,-}^i.$$
\end{itemize}

The map $\psi^2 = \phi$ changes the label from $a$ to $a+1$ in all the simplexes.
The dihedral group $D_{10}$ acts dihedrally on the vertices and on the labels of type $i$ of all the simplexes, preserving the labels of type $a$ and all the signs. 

\subsection{The fundamental group}
Having determined the 2-skeleton $X^2$ of the spine $X$ of $F$, we can use it to determine a presentation for $\pi_1(F)$.

\begin{teo} \label{pi:teo}
We have
$$\pi_1(F) = \langle a_i, b_i \ |\ a_{i+2} = a_{i}a_{i+1}, \ b_{i+2} = b_{i}b_{i+1}, \ a_{i}^{-1}b_{i+1}a_{i+2} = b_i^{-1}a_{i+1}b_{i+2} \rangle$$
where $i=1,\ldots, 6$ is considered modulo 6. 
\end{teo}
\begin{proof}
We work with the spine $X$.
We orient $e_a^i$ from $R_a^*$ to $S_a^*$ and $\hat e_a^i$ from $R_{a+1}^*$ to $R_{a-1}^*$.
We fix $R_1$ as a basepoint and the maximal tree of $X^1$
$$\hat e_2^1 \cup \hat e_3^1 \cup e_1^1 \cup e_2^1 \cup e_3^1.$$
The fundamental group $\pi_1(F) = \pi_1(X^2)$ is generated by the 25 remaining oriented edges. It is notationally more convenient to consider all the 30 edges 
$$e_a^i, \quad \hat e_a^i$$
with $a \in \{1,2,3\}$ and $i\in \{1,2,3,4,5\}$ as generators, and to kill the elements $\hat e_2^1, \hat e_3^1, e_1^1, e_2^1, e_3^1$ by adding them as relators. We get
$3+3+5+15+15=41$ relators
$$
e_a^1, \quad \hat e_a^1, \quad
\hat e_1^n \hat e_2^n \hat e_3^n, \quad
e_a^j(e_a^k)^{-1} \hat e_{a-1}^i (\hat e_{a-1}^l)^{-1}, \quad
e_a^l(e_a^i)^{-1} (\hat e_{a+1}^j)^{-1} \hat e_{a+1}^k
$$
parametrized by $a \in \{1,2,3\}, n \in \{1,2,3,4,5\}$, and by the cyclic permutations $(i,j,k,l,m)$ of $(1,2,3,4,5)$. The last $5+15+15$ relators arise from the $35$ two-cells of $X^2$, that is from the $5$ triangles and $30$ squares shown in Figure \ref{2cells:fig}.
The first $3+3$ are there to eliminate the additional generators.

Some 19 relations can be easily transformed as follows:
\begin{gather*}
e_a^1 = e, \quad
e_a^2 = \hat e_{a-1}^5( \hat e_{a-1}^3)^{-1}, \quad
e_a^3 = ( \hat e_{a+1}^4)^{-1} \hat e_{a+1}^5, \\
e_a^4 = ( \hat e_{a+1}^3)^{-1} \hat e_{a+1}^2, \quad
e_a^5 = \hat e_{a-1}^2( \hat e_{a-1}^4)^{-1}, \quad \hat e^1_a = e.
\end{gather*}

Therefore we can restrict our generators set to the 12 elements $\hat e^i_a$ with $a\in \{1,2,3\}$ and $i\in \{1,2,3,4\}$.
If we substitute the expressions for $e_a^i$ in the remaining 22 relators we get 7 types of relators 
\begin{gather*}
(\hat e_a^3)^{-1} \hat e_a^2 (\hat e_a^4)^{-1} ( \hat e_{a-1}^2)^{-1} \hat e^3_{a-1}, \quad
\hat e_{a+1}^5 (\hat e_{a+1}^3)^{-1} (\hat e_a^2)^{-1} \hat e_a^3 (\hat e_a^5)^{-1}, \\
(\hat e_{a}^4)^{-1} \hat e_{a}^5 (\hat e_{a}^3)^{-1} ( \hat e_{a-1}^5)^{-1} \hat e^4_{a-1}, \quad
\hat e_{a+1}^2 (\hat e_{a+1}^4)^{-1} (\hat e_a^5)^{-1} \hat e_a^4(\hat e_a^2)^{-1} , \\
(\hat e^4_{a})^{-1} \hat e_{a}^5 ( \hat e_{a}^2)^{-1} \hat e_{a}^3 \hat e_{a+1}^2(\hat e_{a+1}^5)^{-1}, \quad
\hat e_a^2(\hat e_a^4)^{-1} \hat e_a^3 (\hat e_a^5)^{-1} (\hat e_{a-1}^3)^{-1} \hat e_{a-1}^4, \\
\hat e_1^n \hat e_2^n \hat e_3^n.
\end{gather*}
We use the relators 1, 4 to express $\hat e_a^4, \hat e_a^5$ in term of $\hat e_a^2, \hat e_a^3$, and introduce 
\begin{gather*}
a_1 = (\hat e^2_1)^{-1}, a_2 = \hat e_1^2(\hat e_3^2)^{-1},
a_3 = (\hat e^2_3)^{-1}, a_4 = \hat e_3^2(\hat e_2^2)^{-1},
a_5 = (\hat e^2_2)^{-1}, a_6 = \hat e_2^2(\hat e_1^2)^{-1} \\
b_1 = \hat e^3_2 (\hat e^3_1)^{-1}, b_2 = (\hat e^3_1)^{-1},
b_3 = \hat e^3_1 (\hat e^3_3)^{-1}, b_4 = (\hat e^3_3)^{-1},
b_5 = \hat e^3_3 (\hat e^3_2)^{-1}, b_6 = (\hat e^3_2)^{-1}.
\end{gather*} 
These elements generate $\pi_1(X)$ and the remaining relators 2, 3, 5, 6, 7 are
\begin{gather*}
b_{i}^{-1}a_{i+5}^{-1}b_ia_{i+1}b_{i+2}^{-1}a_{i+1}b_{i+2}a_{i+2}^{-1}, \quad
a_{i+2}^{-1}b_{i+1}a_{i+1}^{-1}b_{i+2}a_{i+3}b_{i+3}^{-1}, \quad b_{i+3}^{-1}b_{i+1}b_{i+2}, \\ 
b_{i+3}a_{i+3}^{-1}b_{i+2}^{-1}a_{i+1}b_{i+2}a_{i+3}^{-1}b_{i+4}a_{i+5}b_{i+5}^{-1}a_{i+3}^{-1}, \quad
a_{i+1}a_{i+3}a_{i+5},\quad b_ib_{i+2}b_{i+4}
\end{gather*}
where $i\in \{0,2,4\}$ and indices are considered modulo 6. The second relator
$$a_{i+1}^{-1}b_{i+2}a_{i+3} = b_{i+1}^{-1}a_{i+2}b_{i+3}$$
can be used to simplify the other relators, and we end with the relations
$$a_{i+2} = a_ia_{i+1}, \quad b_{i+2} = b_ib_{i+1}, \quad a_{i}^{-1}b_{i+1}a_{i+2} = b_{i}^{-1}a_{i+1}b_{i+2}$$
parametrized by $i\in \{1,\ldots,6\}$.
\end{proof}

This presentation is convenient because it is simple and symmetric. It is well-known that the fundamental group of the Hantsche -- Wendt 3-manifold $\HW$ is the Fibonacci group with presentation
$$\pi_1(\HW) = \langle a_i\ |\ a_{i+2} = a_{i}a_{i+1}\rangle.$$
Therefore $\pi_1(F)$ is obtained from the free product $\pi_1(\HW)*\pi_1(\HW)$ with generators $a_i,b_i$ by adding the 6 relations 
$$a_{i}^{-1}b_{i+1}a_{i+2} = b_i^{-1}a_{i+1}b_{i+2}.$$ 
By some manipulation we may substitute these relations with the following ones:
$$[a_2,b_2] = [a_4,b_4] = [a_6,b_6], \qquad [a_1,b_1] = [a_3,b_3] = [a_5,b_5].$$

\subsection{The actions of the symmetries on $\pi_1(F)$.}
We have defined various symmetries $\rho, \sigma, \tau, \psi, \phi, \varphi$ for $F$ in the previous pages, and we have $\sigma = \tau^2$, $\phi = \psi^2$, and $\varphi = \tau\psi$. 
Each symmetry defines an element of $\Out(\pi_1(F))$, that is an automorphism of $\pi_1(F)$ well-defined only up to conjugation.

\begin{prop}
The automorphism $\psi_*$ acts as:
$$\psi_*(a_i) = a_{i-1}, \qquad \psi_*(b_i) = b_{i-1}.$$
The automorphism $\tau_*$ acts as:
\begin{alignat*}5
a_1 & \longmapsto a_3b_3^{-1}a_1^{-1}, \qquad a_2 &\ \longmapsto a_1b_3a_2^{-1}b_1^{-1}a_5^{-1}, \qquad b_1 & \longmapsto a_3a_1^{-1}, \qquad b_2 &\ \longmapsto a_1^{-1},
\\
a_3 & \longmapsto a_1b_1^{-1}a_5^{-1}, \qquad a_4 &\ \longmapsto a_5b_1a_6^{-1}b_5^{-1}a_3^{-1}, \qquad  
b_3 & \longmapsto a_1a_5^{-1}, 
\qquad b_4 & \longmapsto a_5^{-1}, \\
a_5 & \longmapsto a_5b_5^{-1}a_3^{-1}, \qquad a_6 & \longmapsto a_3b_5a_4^{-1}b_3^{-1}a_1^{-1}, \qquad 
b_5 & \longmapsto a_5a_3^{-1}, \qquad b_6 & \longmapsto a_3^{-1}. 
\end{alignat*}
The automorphism $\rho_*$ acts as:
\begin{alignat*}5
a_1 & \longmapsto a_5b_2a_3, \qquad a_2 & \longmapsto b_3^{-1}a_2^{-1}b_1b_4, \qquad b_1 & \longmapsto a_3^{-1}b_1^{-1}a_1b_3, \qquad\ \ b_2 & \longmapsto a_3^{-1}b_3,
\\
a_3 & \longmapsto a_3^{-1}b_4, \qquad a_4 & \longmapsto b_4^{-1}a_3b_6a_5^{-1}, \qquad  
b_3 & \longmapsto b_3^{-1}a_3b_5a_5^{-1}, 
\qquad\ \ b_4 & \longmapsto b_5a_5^{-1}, \\
a_5 & \longmapsto b_6a_5^{-1}, \qquad a_6 & \longmapsto a_5b_6^{-1}a_5b_2a_3, \qquad 
b_5 & \longmapsto a_5b_5^{-1}a_5b_1a_3, \qquad b_6 & \longmapsto a_5b_1a_3. 
\end{alignat*}

\end{prop}
\begin{proof}
We refer to the proof of Theorem \ref{pi:teo}. The symmetries $\phi, \tau$ and $\rho$ preserve the spine and act as follows on the edges $\hat e_a^i$:
$$\phi(\hat e_a^i) =\hat  e_{a+1}^{i}, \qquad \rho(\hat e_a^i) = \hat e_a^{i+1},
\qquad
\tau(\hat e_a^i) =(\hat  e_{a'}^{i'})^{-1}
$$
where $a' = (23)(a)$, $i' = (2453)(i)$. The symmetry $\phi$ preserves the chosen basepoint and spanning tree, hence we easily deduce that $\phi_*(a_i) = a_{i-2}$ and $\phi_*(b_i) = b_{i-2}$. The calculation of $\tau_*$ and $\rho_*$ needs more care because it does not preserve the spanning tree, and we omit the details.

The automorphism $\psi_*$ is a square root of $\phi_*$, so it is reasonable to expect that it acts as stated, and it can be verified by inspecting its action on the ideal triangulation (recall that $\psi$ does not preserve the tessellation). We omit the details.
\end{proof}

We can deduce the action of the monodromy $\varphi = \tau\psi$.
\begin{cor}
The automorphism $\varphi_*$ acts as:
\begin{alignat*}5
a_1 & \longmapsto a_3b_5a_4^{-1}b_3^{-1}a_1^{-1}, \qquad a_2 &\ \longmapsto a_3b_3^{-1}a_1^{-1}, \qquad b_1 & \longmapsto a_3^{-1}, \qquad b_2 &\ \longmapsto a_3a_1^{-1},
\\
a_3 & \longmapsto  a_1b_3a_2^{-1}b_1^{-1}a_5^{-1}, \qquad a_4 &\ \longmapsto a_1b_1^{-1}a_5^{-1}, \qquad  
b_3 & \longmapsto  a_1^{-1}, 
\qquad b_4 & \longmapsto  a_1a_5^{-1}, \\
a_5 & \longmapsto a_5b_1a_6^{-1}b_5^{-1}a_3^{-1}, \qquad a_6 & \longmapsto a_5b_5^{-1}a_3^{-1}, \qquad 
b_5 & \longmapsto a_5^{-1}, \qquad b_6 & \longmapsto  a_5a_3^{-1}. 
\end{alignat*}
\end{cor}

The patient reader may also verify that the automorphism $\rho^2\sigma\rho^3$ is conjugate to the automorphism that exchanges $a_i$ with $b_i$. 
Recall that $\partial F = \HW_1 \sqcup \cdots \sqcup \HW_5$ where $\HW_i$ is the Hantsche-Wendt 3-manifold that is the link of $P_i$. Each boundary component gives rise to a peripheral subgroup of $\pi_1(F)$, well-defined only up to conjugation. We can determine their generators.

\begin{prop} \label{peri:prop}
The peripheral subgroups have the following generators: 
\begin{gather*}
\pi_1(\HW_1) = \langle a_2b_1^{-1}, a_6^{-1}b_5 \rangle, \quad
\pi_1(\HW_2) = \langle b_1,b_2 \rangle, \quad \pi_1(\HW_3) = \langle a_3^{-1}b_3, b_5a_5^{-1}\rangle, \\ \
\pi_1(\HW_4) = \langle a_1,a_2 \rangle \quad
\pi_1(\HW_5) = \langle a_3^{-1}b_4, b_6a_5^{-1}\rangle
\end{gather*}
\end{prop}
\begin{proof}
The subgroup of $\pi_1(F)$ generated by $a_1,\ldots, a_6$ has indeed the Fibonacci presentation 
$\langle a_i | a_{i+2} = a_ia_{i+1}\rangle$ (without any additional relators) because $\pi_1(F)$ retracts onto it by sending $b_i$ to 1. This subgroup is necessarily peripheral, and is generated by any two of the six elements $a_1,\ldots, a_6$. By acting iteratively on the generators via $\rho$ and by conjugating we find some simple generators of all the peripheral subgroups. Using the fact that $\tau(\HW_2) = \HW_4$ we deduce that the $b_i$'s generate $\pi_1(\HW_2)$ and the $a_i$'s generate $\pi_1(\HW_4)$.
\end{proof}

\subsection{Homology groups}
We determine the homology groups of $F$. If not otherwise mentioned, all the homology groups are defined over $\matZ$.

\begin{prop} \label{chi:prop}
The fiber $F$ is orientable and mirrorable. We have $\chi(F)=1$ and 
$$H_1(F) = (\matZ/4\matZ)^4, \quad H_2(F) = \matZ^4, \quad H_3(F) = \matZ^4.$$
The intersection form $Q$ on $H_2(F)$ has signature $\sigma = 0$ and $\det Q = 16$.
\end{prop}
\begin{proof}
The fiber $F$ is orientable because it is a branched covering over $S^4$ that is orientable. It is mirrorable because the isometry $\tau$ inverts its orientation, and hence $\sigma = 0$. We already know that $\chi(F) =1$.
We easily get $H_1(F)$ from $\pi_1(F)$, generated by $a_1,a_2,b_1,b_2$. We have
$$H_1(\partial F) \stackrel {i_*}\longrightarrow H_1(F) \longrightarrow H_1(F,\partial F) \longrightarrow H_0(\partial F) \longrightarrow H_0(F) \longrightarrow 0.$$
The map $i_*$ is surjective, $H_0(\partial F) = \matZ^5, H_0(F) = \matZ$, hence $H^3(F) = H_1(F,\partial F) = \matZ^4$. Then $H_3(F) = \matZ^4$ and $H_2(F)$ has no torsion. 
Then $H_2(F) = \matZ^4$. We have
$$0 \longrightarrow H_2(F) \stackrel {j_*} \longrightarrow H_2(F,\partial F) \stackrel \partial \longrightarrow H_1(\partial F) \stackrel{i_*^1}\longrightarrow H_1(F) \longrightarrow 0$$
because $i_*^1$ is surjective and $H_2(\partial F)=0$. We have $H_2(F,\partial F) = H^2(F) = \matZ^4 \times (\matZ/4\matZ)^4$, $H_1(\partial F) = (\matZ/4\matZ)^{10}$ and $H_1(F) = (\matZ/4\matZ)^4$.
Therefore $H_2(F)$ injects in  
$H_2(F,\partial F)/{\rm Tors}$ as a subgroup of index $4^{10-4-4}=16$. 

The intersection form $Q_{F}$ is the restriction of the unimodular bilinear form
$$H_2(F) \times H_2(F,\partial F)/{\rm Tors} \longrightarrow \matZ$$
furnished by Poincar\'e duality. Since the form is unimodular and $H_2(F)$ has index 16 we have $\det Q_{F} = \pm 16$, with a positive sign because the signature vanishes.
\end{proof}

These data are not enough to determine the isomorphism type of $Q$. We now patiently compute $Q$ on some generators.

\subsection{Some interesting surfaces} We determine explicit generators for the homology groups $H_2(F)$ and $H_2(F,\partial F)/{\rm Tors}$. The spine $X$ and the ideal triangulation $\Delta$ contain some relatively simple surfaces that generate both groups.

\subsubsection{Genus two surfaces}
The spine $X$ contains the 6 closed surfaces
\begin{align*}
\Sigma_a^1 & = (T_a^{123})^* \cup (T_a^{234})^* \cup (T_a^{345})^* \cup (T_a^{451})^* \cup
(T_a^{512})^*, \\
\Sigma_a^2 & = (T_a^{124})^* \cup (T_a^{235})^* \cup (T_a^{341})^* \cup (T_a^{452})^* \cup
(T_a^{513})^*
\end{align*}
with $a\in \{1,2,3\}$. Recall that each $(T_a^{ijk})^*$ is a square as in Figure \ref{2cells:fig}. The edges of the five squares in the definition of $\Sigma_a^i$ match to produce a surface of genus two, tessellated into 3 vertices (that are $S_a^*, R_a^*, R_{a \mp 1}^*$ with the sign depending on $i=1,2$), 10 edges, and 5 squares, that we orient like the squares from Figure \ref{2cells:fig}. 

The 2-skeleton $X^2$ is in fact the union of these 6 surfaces $\Sigma_a^j$, plus the 5 triangles $Q_i^*$. We also note that
\begin{equation} \label{sigma:eqn}
\rho(\Sigma_a^i) = \Sigma_a^i, \quad \phi(\Sigma_a^i) = \Sigma_{a+1}^i, \quad
\tau(\Sigma_a^i) = \Sigma_{2-a}^{3-i}.
\end{equation}

\subsubsection{Pairs of pants}
The ideal triangulation $\Delta$ contains the 60 surfaces
$$P_{a,\pm}^{ijk} = T_{a+1,\pm}^{ijk} \cup T_{a-1,\pm}^{ijk}$$
with $a\in \{1,2,3\}$ and $\{i,j,k\} \subset \{1,\ldots,5\}$.
Each such surface $P_{a,\pm}^{ijk}$ is a thrice-punctured sphere, union of two ideal triangles that share the same edges; this becomes a properly embedded pair of pants in the compact manifold $F$ with
$$\partial P_{a,\pm}^{ijk} \subset \HW_i \sqcup \HW_j \sqcup \HW_k.$$

We assign to each pair of pants $P_{a,\pm}^{ijk}$ the orientation of the triangle $T_{a+1, \pm}^{ijk}$ induced by the cyclic ordering of its vertices. We note that
$$
\rho(P_{a,\pm}^{ijk}) = P_{a,\pm}^{i+1,j+1,k+1}, \quad
\sigma(P_{a,\pm}^{ijk}) = P_{a,\pm}^{i'j'k'}, \quad
\phi(P_{a,\pm}^{ijk})  = P_{a+1,\pm}^{ijk}
$$
where $(i',j',k') = (25)(34)(i,j,k)$. The order-6 symmetry $\psi$ acts as follows:
\begin{equation} \label{psi:eqn}
\begin{aligned} 
P^{i,i+1,i+2}_{a,-} & \longrightarrow P^{i,i+1,i+2}_{a,+} \longrightarrow P^{i,i+1,i+2}_{a+1,-}, \\ P^{i,i+1,i+3}_{a,-} & \longrightarrow P^{i,i+1,i+3}_{a+1,+} \longrightarrow P^{i,i+1,i+3}_{a+1,-}.
\end{aligned}
\end{equation}

\subsubsection{Algebraic intersections}
We fix once for all an orientation for $F$ as follows: we orient the rectified simplex $R_a$ by choosing the vector $(1,1,1,1,1)$ as a positive normal to the hyperplane containing them, and we orient $F$ like $R_a$.

The genus-two surfaces $\Sigma_a^j$ are contained in the spine $X$, which is dual to the tessellation $\Pi$, while the pants $P_{a,-}^{ijk}$ are contained in the ideal tessellation $\Pi$. By construction the genus two surfaces and the pants intersect transversely, and algebraic intersections are easily calculated.

\begin{prop} \label{Sigma:P:prop}
The following algebraic intersections hold: 
\begin{align*}
\Sigma_a^1 \cdot P_{a,-}^{i,i+1,i+2} = 0, \ & 
\Sigma_a^1 \cdot P_{a \pm 1,-}^{i,i+1,i+2} = \pm 1, & 
\Sigma_a^1 \cdot P_{a,-}^{i,i+1,i+3} = 0, \ & 
\Sigma_a^1 \cdot P_{a \pm 1,-}^{i,i+1,i+3} = 0, \\
\Sigma_a^2 \cdot P_{a,-}^{i,i+1,i+2} = 0, \ & 
\Sigma_a^2 \cdot P_{a \pm 1,-}^{i,i+1,i+2} = 0, & 
\Sigma_a^2 \cdot P_{a,-}^{i,i+1,i+3} = 0, \ & 
\Sigma_a^2 \cdot P_{a \pm 1,-}^{i,i+1,i+3} = \pm 1 
\end{align*}
\end{prop}
\begin{proof}
The surface $\Sigma_a^1$ is disjoint from all the pants except $P_{a \pm 1,-}^{i,i+1,i+2}$, that intersects $\Sigma_a^1$ transversely in one point, with sign $\pm 1$. The case of $\Sigma_a^2$ is analogous.
\end{proof}

The intersection between the genus two surfaces $\Sigma_a^j$ and the pants $P_{a,+}^{ijk}$ is more complicated to calculate since the latter ones are not contained in $\Pi$.

\subsection{The second homology group}
Every (properly embedded) oriented surface (with boundary) $S \subset \bar F$ determines a class in the homology group $H_2(F,\matZ)$ (respectively $H_2(F,\partial F, \matZ)$), that we denote with the same symbol $S$ for simplicity.

\begin{prop}
The group $H_2(F)$ is generated by the genus two surfaces 
$$\Sigma_1^1,\quad \Sigma_2^1,\quad \Sigma_3^1,\quad \Sigma_1^2,\quad \Sigma_2^2,\quad \Sigma_3^2$$ 
with relations
$$\Sigma_1^1+\Sigma_2^1+\Sigma_3^1 = 0, \qquad
\Sigma_1^2+\Sigma_2^2+\Sigma_3^2=0.$$
We have
\begin{gather*}
\Sigma_a^1 \cdot \Sigma_a^1 = 2, \quad \Sigma_a^2 \cdot \Sigma_a^2 = -2, \quad \Sigma_a^1 \cdot \Sigma_{a+1}^1 = -1, \quad \Sigma^2_a \cdot \Sigma^2_{a+1} = 1, \\
\Sigma_a^1 \cdot \Sigma^2_a = -1, \quad \Sigma_a^1 \cdot \Sigma^2_{a+1} = 1, \quad \Sigma_a^1 \cdot \Sigma^2_{a-1} = 0
\end{gather*}
Therefore the intersection form with respect to the basis
$\Sigma_1^1, \Sigma_2^1, \Sigma_1^2, \Sigma_2^2$ is 
$$Q = \begin{pmatrix}
2 & -1 & -1 & 1 \\
-1 & 2 & 0 & -1 \\
-1 & 0 & -2 & 1 \\
1 & -1 & 1 & -2
\end{pmatrix}.
$$
Analogously, the group $H_2(F,\partial F)/{\rm Tors}$ is generated by the pairs of pants
$$P_{1,-}^{123}, \quad P_{2,-}^{123}, \quad P_{3,-}^{123}, \quad P_{1,-}^{124}, \quad P_{2,-}^{124}, \quad P_{3,-}^{124}$$ 
with relations
$$P_{1,-}^{123} + P_{2,-}^{123} + P_{3,-}^{123} = 0, \qquad P_{1,-}^{124} + P_{2,-}^{124} + P_{3,-}^{124} = 0.$$
The basis $P_{2,-}^{123}, -P_{1,-}^{123}, P_{2,-}^{124}, -P_{1,-}^{124}$ is dual to 
$\Sigma_1^1,\Sigma_2^1,\Sigma_1^2,\Sigma_2^2$.
\end{prop}
\begin{proof}
We know that $H_2(F) \cong H_2(F,\partial F)/{\rm Tors} \cong \matZ^4$ and from Proposition \ref{Sigma:P:prop} we deduce that $\Sigma_1^1$, $\Sigma_2^1$, $\Sigma_1^2$, $\Sigma_2^2$ and $P_{2,-}^{123}, -P_{1,-}^{123}, P_{2,-}^{124}, -P_{1,-}^{124}$ are dual basis. 

From Proposition \ref{Sigma:P:prop} we also deduce that 
the classes $\Sigma_1^1+\Sigma_2^1+\Sigma_3^1 $,
$\Sigma_1^2+\Sigma_2^2+\Sigma_3^2$,
$P_{1,-}^{123} + P_{2,-}^{123} + P_{3,-}^{123}$, $P_{1,-}^{124} + P_{2,-}^{124} + P_{3,-}^{124} $ have zero intersection with all the elements of the dual space and are hence zero.

The surface $\Sigma_a^1$ intersects transversely $\Sigma_{a+1}^1$ and $\Sigma_{a+1}^2$ in one point, that is respectively $R_a^*$ and $R_{a-1}^*$. By calculating the sign we find $\Sigma_a^1 \cdot \Sigma_{a+1}^1 = -1$ and $\Sigma_a^1 \cdot \Sigma_{a+1}^2 = 1$. Since $\Sigma_1^1+\Sigma_2^1+\Sigma_3^1=0$ we deduce that $\Sigma_a^1 \cdot \Sigma_a^1 = 2$. Since $\tau$ is orientation-reversing and sends $\Sigma_a^1$ to $\Sigma_{a'}^2$ we get $\Sigma_a^2\cdot \Sigma_a^2 = -2$ and $\Sigma_a^2 \cdot \Sigma_{a+1}^2 = 1$. 

It is possible to isotope $\Sigma_a^1$ away from $\Sigma_{a-1}^2$, hence $\Sigma_a^1\cdot \Sigma_{a-1}^2 = 0$ and from $\Sigma_1^2+\Sigma_2^2+\Sigma_3^2=0$ we finally deduce also that $\Sigma_a^1\cdot \Sigma_a^2 = -1$. 
\end{proof}

We note that indeed $\det Q = 16$ and $\sigma = 0$ as predicted by Proposition \ref{chi:prop}.

\subsection{The actions of the symmetries on the homology}
The homology group $H_1(F)$ is finite, and the action of the symmetries of $F$ on $H_1(F)$ can be easily deduced from the action on $\pi_1(F)$. We turn our attention to the more interesting second homology group $H_2(F) \cong \matZ^4$. We fix
$\Sigma_1^1, \Sigma_2^1, \Sigma_1^2, \Sigma_2^2$ as a basis for $H_2(F)$. 

\begin{teo}
The symmetries of $F$ act on $H_2(F)$ as $\rho_* = \id, \sigma_* = -\id$ and
$$\phi_* =
\begin{pmatrix} 
0 & -1 & 0 & 0 \\
1 & -1 & 0 & 0 \\
0 & 0 & 0 & -1 \\
0 & 0 & 1 & -1
\end{pmatrix}, \quad
\tau_* = 
\begin{pmatrix}
0 & 0 & -1 & 1 \\
0 & 0 & 0 & 1 \\
1 & -1 & 0 & 0 \\
0 & -1 & 0 & 0
\end{pmatrix}, \quad
\psi_* = 
\begin{pmatrix}
-1 & 0 & -1 & 1 \\
0 & -1 & -1 & 0 \\ 
-1 & 2 & 1 & 0 \\
-2 & 1 & 0 & 1
\end{pmatrix}
$$
\end{teo}
\begin{proof}
The symmetries $\rho, \phi, \tau$ permute the surfaces $\Sigma_a^i$ as in \eqref{sigma:eqn} and we deduce the state expressions for $\rho_*, \phi_*, \tau_*$, noticing that $\rho, \sigma$ preserve all the orientations of $\Sigma_a^i$, while $\tau: \Sigma_a^i \to \Sigma_{2-a}^{3-i}$ is orientation preserving if and only if $i=1$. Therefore $\sigma_* = \tau_*^2 = -\id$. The calculation of $\psi_*$ is less obvious since $\psi$ is not a symmetry for $X$, and it is preferable to use the dual basis 
$$P_{2,-}^{123},\quad -P_{1,-}^{123},\quad P_{2,-}^{124},\quad -P_{1,-}^{124}$$ 
for 
$H_2(F,\partial F)/{\rm Tors}$. By \eqref{psi:eqn} the automorphism $\psi$ sends these elements to
$$P_{2,+}^{123},\quad -P_{1,+}^{123},\quad P_{3,+}^{124},\quad -P_{2,+}^{124}.$$ 
By expressing each of these homology classes in the original basis we get the desired expression for $\psi_*$. (We omit the calculation, which is standard.)
\end{proof}

Since $\rho, \sigma, \phi$ preserves the orientation of $F$ and $\tau, \psi$ inverts it, we can check that indeed $\rho_*,\sigma_*,\phi_*$ preserve $Q$ while $\tau_*, \psi_*$ send $Q$ to $-Q$.

\begin{cor}
The monodromy $\varphi$ acts as
$$\varphi_* =
\begin{pmatrix}
-1 & -1 & -1 & 1 \\
-2 & 1 & 0 & 1 \\
-1 & 1 & 0 & 1 \\
0 & 1 & 1 & 0
\end{pmatrix}
$$
\end{cor}
\begin{proof}
We have $\varphi_* = \psi_*\tau_*$.
\end{proof}

The characteristic polynomial of $\varphi_*$ is $x^4-6x^2+1$, with eigenvalues $\pm 1 \pm \sqrt 2$.

\subsection{Spin manifold} We prove the following.

\begin{prop}
The manifold $F$ is spin.
\end{prop}
\begin{proof}
To prove that $F$ is spin is equivalent to show that $\alpha \cdot \alpha = 0$ for every class $\alpha \in H_2(F,\matZ/2\matZ)$. 
We argue similarly as in the proof of Proposition \ref{chi:prop} with $\matZ/2\matZ$ coefficients: we know $H_i(\partial F, \matZ/2\matZ)$, $H_1(F, \matZ/2\matZ)$, $\chi(F)$, and that $H_1(\partial F, \matZ/2\matZ) \to H_1(F,\matZ/2\matZ)$ is surjective; by looking at the long exact sequence of the pair $(F,\partial F)$ we easily deduce that 
$H_2(F, \matZ/2\matZ) = (\matZ/2\matZ)^8$ and that the image of $H_2(\partial F) \to H_2(F)$ has dimension 6. Let $\alpha_1,\ldots,\alpha_6$ be a basis of the image. We can complete it to a basis of $H_2(F)$ by adding $\Sigma_1^1,\Sigma_2^1$ because we have $\alpha_i\cdot \Sigma_j^1 = 0$, $\Sigma_i^1\cdot\Sigma_i^1=0$, and $\Sigma_1^1\cdot \Sigma_2^1=1$. Now $\alpha_i \cdot \alpha_i=0$ because the boundary classes can be pushed inside $F$ and $\Sigma_i^1\cdot\Sigma_i^1=0$ implies that $\alpha\cdot \alpha=0$ for every $\alpha \in H_2(F, \matZ/2\matZ)$.
\end{proof}

\subsection{Proof of Theorem \ref{H:F:teo}}
To complete the proof, it only remains to make a couple of remarks.
First, no non-trivial class in $H_2(F)$ may be represented by any immersed sphere or tori, since these can be homotoped into $\partial F$ by Theorem \ref{F:teo}, that is a union of rational homology spheres. Second, the homology group $H_3(F)$ is generated by the boundary components of $F$ and hence the action of $\varphi_*$ on $H_3(F)$ has finite order. The proof is complete.

\section{The Flat Closing Problem} \label{closing:section}
In this long section we prove Theorem \ref{closing:teo}. Recall that the link $X$ of a point $P_i$ in $\bar F$ is the Hantsche -- Wendt 3-manifold $\HW$, corresponding to the $i$-th boundary component of $F$, endowed with a spherical cone structure. Its universal cover $\tilde X$ is homeomorphic to $\matR^3$, equipped with the lifted spherical cone structure. The main ingredient of the proof of Theorem \ref{closing:teo} is the following

\begin{teo} \label{CAT1:teo}
The space $\tilde X$ is $\CAT(1)$.
\end{teo}

We will in fact use the following corollary.

\begin{cor} \label{CAT1:cor}
There is a finite covering $X_*$ of $X$ that is $\CAT(1)$. Moreover, every covering of $X_*$ is also $\CAT(1)$.
\end{cor}
\begin{proof}
We will construct below a triangulation for $X$, consisting of spherical tetrahedra. This lifts to a triangulation of any cover $X_*$. By \cite[Theorem 5.4]{BH} the space $X_*$ is $\CAT(1)$ if and only if the triangulation satisfies the \emph{link condition} (the link of every vertex should be $\CAT(1)$) and the space contains no closed geodesic of length $<2\pi$. The link condition is easily verified for $X$ and every cover of $X$, see the proof of Proposition \ref{link:prop}. So it only remains to exclude short closed geodesics.

By Theorem \ref{CAT1:teo} there are no closed geodesics of length $< 2\pi$ in $\tilde X$. Any cover $X_*$ of $X$ where metric balls of radius $\pi$ lift isometrically to $\tilde X$ also does not contain short closed geodesics. Since $\pi_1(X)$ is residually finite we can find such $X_*$.
\end{proof}

Every finite covering $F_*$ of $F$ induces, by shrinking boundary components to points, a branched covering $\bar F_* \to \bar F$, ramified along the points at infinity. The space $\bar F_*$ inherits a cone manifold flat structure from $\bar F$. 

\begin{cor}
There is a finite covering $F_*$ of $F$ such that $\bar F_*$ is locally $\CAT(0)$. Moreover, for every covering $F_{**}$ of $F_*$ the space $\bar F_{**}$ is locally $\CAT(0)$.
\end{cor}
\begin{proof}
The space $\bar F$ is locally $\CAT(0)$ everywhere except possibly at the points at infinity, and hence so is any branched covering $\bar F_*$. Since $\pi_1(F)$ is residually finite, we can find a finite covering $F_*$ such that the link of every point at infinity in $\bar F_*$ is a covering of $X_*$, and is hence $\CAT(1)$. Therefore the link condition is valid everywhere and $\bar F_*$ is $\CAT(0)$. The same argument applies to any cover of $F_*$.
\end{proof}

\subsection{Proof of Theorem \ref{closing:teo}}
We now prove Theorem \ref{closing:teo}, assuming Theorem \ref{CAT1:teo}. Every finite covering $N_*^5 \to N^5$ fibers over $S^1$ with some fiber $F_*$ that covers $F$.
Since $\pi_1(N^5)$ is residually finite, using the previous corollary we can find a covering $N^5_*$ such that:
\begin{enumerate}
\item The cusps of $N^5$ have disjoint sections with systole $> 2 \pi$;
\item The space $\bar F_*$ is locally $\CAT(0)$.
\end{enumerate}

The monodromy $\varphi \colon \bar F \to \bar F$ lifts to a monodromy $\varphi_* \colon \bar F_* \to \bar F_*$. The mapping torus of $\varphi_*$ is the pseudo-manifold $\bar N^5$ obtained by shrinking every boundary component of $N^5$ to a circle. Since the systoles are $> 2\pi$, by \cite[Theorem 2.7]{FM} the space $\bar N^5$ has a locally $\CAT(-1)$ complete path metric. In particular it is aspherical and $\pi_1(\bar N^5)$ is hyperbolic and torsion-free. As in \cite[Section 3]{IMM} we deduce that $\Out(\pi_1(\bar F_*))$ is infinite and hence $\pi_1(\bar F_*)$ is not hyperbolic. It does not contain $\matZ \times \matZ$ because it is contained in the hyperbolic group $\pi_1(\bar N^5)$.

We have found a compact locally $\CAT(0)$ space $Y = \bar F_*$ such that $\pi_1(Y)$ is not hyperbolic and does not contain $\matZ \times \matZ$. The proof of Theorem \ref{closing:teo} is complete.

\subsection{Start of the proof of Theorem \ref{CAT1:teo}}
It only remains to prove Theorem \ref{CAT1:teo}, and this demanding task will occupy the rest of this long section. The hard part will be to show that $\tilde X$ does not contain any closed geodesic of length $< 2\pi$. To achieve this we will need a number of estimates, and we will conclude with a computer assisted analysis.

We first provide a satisfying description for $\tilde X$.

\subsection{A periodic set of blue lines in $\matR^3$}

\begin{figure}
\centering
\labellist
\pinlabel $x$ at 40 40
\pinlabel $y$ at 350 120
\pinlabel $z$ at 135 315
\pinlabel $O$ at 117 115
\small\hair 2pt
\endlabellist
\includegraphics[width=12.5 cm]{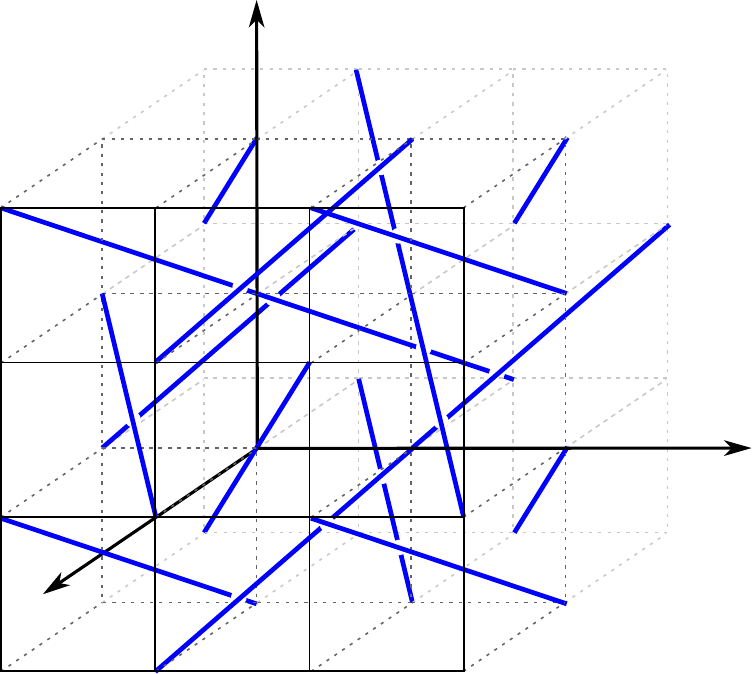}
\caption{A periodic set of lines in $\matR^3$. The edge length of every cube is 2.}
\label{grid4:fig}
\end{figure}
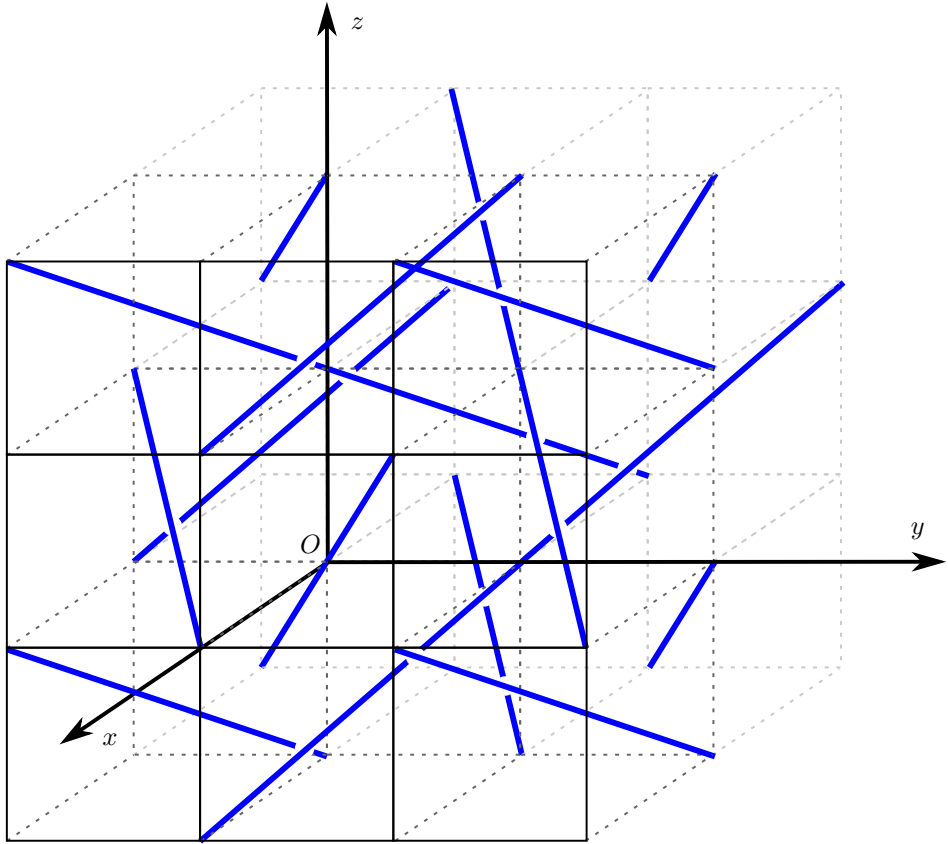

Consider the \emph{blue lines} in $\matR^3$ shown in Figure \ref{grid4:fig}. These are the four lines
$$ l_1 = \left\{ \begin{pmatrix} t \\ t \\ t \end{pmatrix} \right\}, \quad
l_2 = \left\{ \begin{pmatrix} t \\ -t \\ 2+t \end{pmatrix} \right\}, \quad
l_3 = \left\{ \begin{pmatrix} 2+t \\ t \\ -t \end{pmatrix} \right\}, \quad
l_4 = \left\{ \begin{pmatrix} -t \\ 2+t \\ t \end{pmatrix} \right\}
$$
plus all those obtained from these by translating via the group $H$ generated by the vectors $(\pm 2, \pm 2, \pm 2)$, with all possible signs. The group $G$ of isometries of $\matR^3$ that preserve these lines is a particularly rich \emph{crystallographic group}, that is a discrete group of isometries with compact quotient $\matR^3/G$. Figure \ref{grid4:fig} can be found in a paper of 
Conway -- Friedrichs -- Huson -- Thurston \cite{CFHT} that classifies the three-dimensional crystallographic groups. 

Figure \ref{grid3:fig} shows also the segments that realize the minimum distance $\sqrt 2$ between two distinct blue lines. These segments are of course perpendicular to the blue lines, and they are parallel to the 12 vectors obtained by permuting the coordinates of $(\pm 1, \pm 1, 0)$.
The segments form a \emph{trivalent graph} in $\matR^3$ with two connected components, drawn in red and green in the figure, with vertices contained in $(\frac 12\matZ)^3$. The angles between the three segments exiting from any vertex are $2\pi/3$, so the trivalent graph is a periodic \emph{minimal network} in space.

\begin{figure}
\includegraphics[width=12.5 cm]{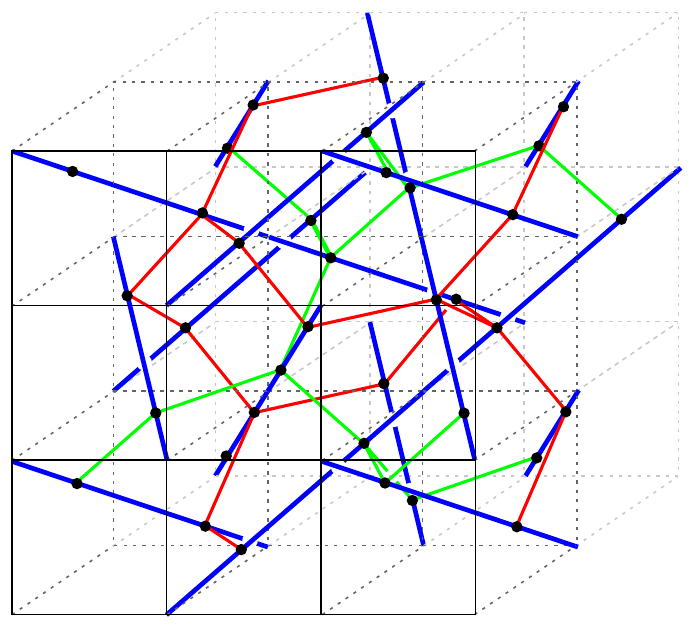}
\caption{The red and green segments are the shortest ones connecting two distinct blue lines. They form a periodic minimal network in space with two connected components (red and green).}
\label{grid3:fig}
\end{figure}

As we said the group $G$ contains various types of isometries, and 
we describe some.
The \emph{light gray} and \emph{dark gray lines} in $\matR^3$ are the lines parallel to the coordinate axis that cut in half the face of some cubes of Figure \ref{grid4:fig} as in the left and right picture of Figure \ref{reflection_lines:fig}, respectively. The following isometries are element of $G$:
\begin{itemize}
\item Translation via a vector in $H$;
\item Reflection with respect to a vertex of a cube;
\item Rotation of angle $2\pi/3$ along a blue line; 
\item Reflection with respect to a light gray line;
\item Rototranslation with respect to a dark gray line, with angle $\pi$ and step 2;
\item Reflection with respect to a line containing a green or red segment.
\end{itemize}

The rotational part of these isometries generate the full symmetry group $O_{48} < {\rm O}(3)$ of the octahedron. In fact we have an exact sequence
$$0 \longrightarrow H \longrightarrow G \longrightarrow O_{48} \longrightarrow 0.$$

The group $G$ of course preserves the minimal network, and one checks that an element $g\in G$ exchanges the red and green components of the minimal network if and only if it reverses the orientation of $\matR^3$. 

\begin{figure}
\centering
\labellist
\small\hair 2pt
\endlabellist
\includegraphics[width=10 cm]{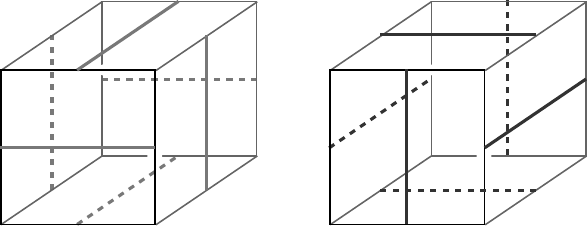}
\caption{A light (dark) gray line in $\matR^3$ is any line parallel to some axis that intersects some cube in Figure \ref{grid4:fig} as in the left (right) picture. The reflections (that is, $\pi$ rotations) along the light gray lines and the rototranslations along the dark gray lines with angle $\pi$ and step 2 preserve the blue lines of Figure \ref{grid4:fig} and are therefore elements of the group $G$.}
\label{reflection_lines:fig}
\end{figure}

The rototranslations generate a subgroup $G' < G$ that acts freely on $\matR^3$ and quotients it to the Hantsche -- Wendt flat 3-manifold $\HW$. Any two adjacent cubes (sharing a face) form altogether a fundamental domain for $G'$. 

\subsection{A periodic triangulation $\Delta$ of $\matR^3$} \label{triangulation:subsection}
We construct a $G$-invariant triangulation of $\matR^3$ whose 1-skeleton contains the blue lines and the minimal network. Every vertex $v$ of the network has:
\begin{enumerate}
\item three vertices at distance $\sqrt 2$, connected to $v$ via red or green edges;
\item two vertices at distance $\sqrt 3$ contained in the same blue line as $v$;
\item six vertices at distance $\sqrt 5$.
\end{enumerate}

Every other vertex is at distance larger than $\sqrt 5$. For instance, the 3, 2, 6 vertices that are of distance $\sqrt 2, \sqrt 3, \sqrt 5$ from $v=\frac 12(1,1,1)$ are obtained by adding to $v$ the following vectors
\begin{gather*}
(-1,0,1), (1,-1,0), (0,1,-1), \quad (1,1,1), (-1,-1,-1), \\
(0,2,1), (1,0,2), (2,1,0), (0,-2,-1), (-1,0,-2), (-2,-1,0).
\end{gather*}

Let $X^1\subset \matR^3$ be the periodic graph obtained as the union of all the blue lines, the red and green edges, and all the segments of length $\sqrt 5$ connecting two vertices, that we draw in black. The periodic graph $X^1$ has edges of length $\sqrt 2$, $\sqrt 3$, $\sqrt 5$, coloured respectively in blue, red or green, and black.

Let $X^2$ be the simplicial complex obtained by completing $X^1$, that is by adding a triangle at every triple of cyclically concatenated edges. There are two types of triangles, with edge lengths $(\sqrt 2, \sqrt 3, \sqrt 5)$ and $(\sqrt 2, \sqrt 5, \sqrt 5)$. One can check that the complement of $X^2$ in $\matR^3$ consists of tetrahedra of two kinds, as in Figure \ref{grid5:fig}. Therefore $X^2$ is the 2-skeleton of a $G$-invariant triangulation $\Delta$ of $\matR^3$.

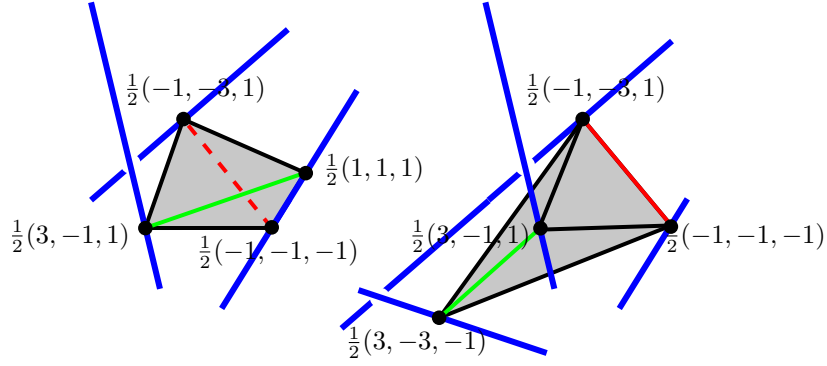
\begin{figure}
\centering
\labellist
\pinlabel $\frac 12(-1,-1,-1)$ at 71 39
\pinlabel $\frac 12(1,1,1)$ at 107 70
\pinlabel $\frac 12(3,-1,1)$ at -7 45
\pinlabel $\frac 12(-1,-3,1)$ at 40 100
\pinlabel $\frac 12(-1,-1,-1)$ at 246 45
\pinlabel $\frac 12(3,-1,1)$ at 143 45
\pinlabel $\frac 12(-1,-3,1)$ at 190 100
\pinlabel $\frac 12(3,-3,-1)$ at 123 5
\small\hair 2pt
\endlabellist
\includegraphics[width = 8 cm]{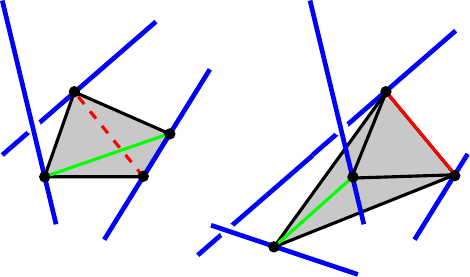}
\caption{Two kinds of tetrahedra.}
\label{grid5:fig}
\end{figure}

Summing up, we have constructed a triangulation $\Delta$ for $\matR^3$ with isometry group $G$.
It has four types of edges, coloured in blue, green, red, and black, of length $\sqrt 2$, $\sqrt 3$, $\sqrt 3$, $\sqrt 5$. The valence of these edges (that is the number of incident tetrahedra) is 6, 10, 10, 4, respectively. The triangulation has two isometry types of tetrahedra, shown in Figure \ref{grid5:fig}.

\subsection{A description of $\tilde X$} \label{Xtilde:subsection}
We have constructed a triangulation $\Delta$ of $\matR^3$. Every tetrahedron of $\Delta$ has two opposite red and green edges.
We now assign to each tetrahedron of $\Delta$ the structure of a \emph{spherical} tetrahedron with dihedral angles $\pi/5,\pi/5,\pi/2,\pi/2,\pi/2,\pi/2,\pi/2$, where the angles $\pi/5$ are assigned to the green and red edges. 

In this way we have equipped $\matR^3$ with a \emph{spherical cone manifold} structure. Since the blue, green, red, and black edges have valence 6, 10, 10, 4, and are assigned the dihedral angles $\pi/2, \pi/5, \pi/5, \pi/2$ at each tetrahedron, they have cone angles $3\pi, 2\pi, 2\pi, 2\pi$ overall. Therefore the singular set is the union of the blue lines, each with cone angle $3\pi$. 

\begin{lemma} \label{Xtilde:lemma}
The space $\matR^3$ with the spherical cone structure just defined is isometric to the universal cover $\tilde X$ of $X$. We have $\Isom(\tilde X) = G$ and $X= \tilde X/G'$.
\end{lemma}

The next few pages are fully dedicated to proving of this lemma. The reader who is not interested in the proof may jump directly to Section \ref{link:subsection}.

\subsection{Spherical join}
The \emph{spherical join} of two subsets $I,J \subset S^1\subset \matC$ is
$$I*J = \big\{(z\cos \theta, w\sin \theta) \ \big|\ (z,w) \in I\times J,\, \theta \in [0,\pi/2] \big\} \subset S^3 \subset \matC^2.$$

The spherical join $I*J$ consists of $I\times \{0\}$, $\{0\}\times J$, and all the geodesics of length $\pi/2$ joining them.
The join of two arcs in $S^1$ of length $\alpha,\beta <\pi$ is a \emph{spherical tetrahedron} with edge lengths $\alpha, \beta, \pi/2, \pi/2, \pi/2, \pi/2$. In this case the dihedral angle of an edge equals the length of the opposite edge. 
The join of a point with $S^1$ is a \emph{half-sphere}. The join of an arc of length $\alpha < \pi$ with $S^1$ is a \emph{lens}, bounded by two half-spheres meeting at an angle $\alpha$.

\subsection{The orbifold $S^3/D_{10}$}
Recall that the link of a singular point $P_i$ in the 4-dimensional flat orbifold $\Orb = \T /D_{10}$ is the spherical 3-orbifold
$$S^3 / D_{10}$$
where the dihedral group $D_{10}$ is generated by
\begin{equation} \label{rs:eqn}
r\colon (z,w) \longmapsto (\zeta z, \zeta^2 w), \quad
s \colon (z,w) \longmapsto (\bar z, \bar w).
\end{equation}
This action is not exactly that of \eqref{D10:eqn}, but it is conjugate to it and notationally easier, so we decide to employ it here. The singular set of the orbifold is a closed geodesic $\gamma$, image of the fixed points set of $s$, having length $2\pi$ and cone angle $\pi$.

The isometries
$$\tau(z,w) = (-\bar w, z), \qquad \nu_1(z,w) = (-z,w), \qquad \nu_2(z,w) = (z,-w)$$
of $S^3$ normalize $D_{10}$ and hence
induce isometries of $S^3/D_{10}$ of order 4, 2, and 2. They act on $\gamma$
as a $\pi/2$ rotation and as reflections, respectively. They generate a group $K$ of 8 isometries of $S^3/D_{10}$. The isometry $\tau$ is the one from Section \ref{additional:subsection}.

We would like to describe fundamental domains for the actions of $D_{10}$ and $K$. 

\subsection{A fundamental domain for $D_{10}$} \label{lens:subsection}
Given $z_1,z_2 \in S^1$, let $[z_1,z_2] \subset S^1$ be the counterclockwise arc from $z_1$ to $z_2$. We set $\eta = e^\frac{\pi i}5$ and $\zeta = e^\frac {2 \pi i}5$ and define 
$$L_j = [\eta^j, \eta^{j+1}] * S^1 \subset S^3, \quad \pi_j = \{\eta^j\} * S^1$$
for all $j\in \matZ/10\matZ$. 
Here $L_j$ is a lens, with boundary two half-spheres $\pi_j$ and $\pi_{j+1}$ meeting with an angle $\pi/5$ as in Figure \ref{lens:fig}. The 3-sphere $S^3$ is tessellated into the 10 lenses $L_0,\ldots, L_9$. We have
\begin{gather*}
r(I*J) = \zeta I * \zeta^2 J, \qquad s(I*J) = \bar I* \bar J,  \\
\tau(I*J) = (-\bar J)*I, \quad \nu_1(I*J) = (-I) * J, \quad \nu_2(I*J) = I * (-J). \\
r(\pi_j) = \pi_{j+2}, \ r(L_j) = L_{j+2}, \qquad s(\pi_j) = \pi_{-j}, \ s(L_j) = L_{-j-1}, \\
\nu_1(\pi_j) = \pi_{j+5}, \ \nu_1(L_j) = L_{j+5}, \qquad \nu_2(\pi_j) = \pi_j, \ \nu_2(L_j) = L_j.
\end{gather*}

\begin{figure}
\vspace{.3 cm}
\centering
\labellist
\pinlabel $(0,-1)$ at -15 40
\pinlabel $(0,1)$ at 162 40
\pinlabel $(0,\eta)$ at 140 60
\pinlabel $(0,\eta^2)$ at 103 65
\pinlabel $(0,\eta^3)$ at 42 65
\pinlabel $(0,\eta^4)$ at 10 60
\pinlabel $(0,\eta^6)$ at 10 15
\pinlabel $(0,\eta^7)$ at 42 12
\pinlabel $(0,\eta^8)$ at 103 14
\pinlabel $(0,\eta^9)$ at 140 20
\pinlabel $(1,0)$ at 75 -5
\pinlabel $(\eta,0)$ at 75 82
\pinlabel $\gamma_0$ at 100 1
\pinlabel $\gamma_1$ at 65 45
\pinlabel $B_0^-$ at 65 13
\pinlabel $B_1^-$ at 100 45
\pinlabel $B_1^+$ at 40 45
\small\hair 2pt
\endlabellist
\includegraphics[width = 6 cm]{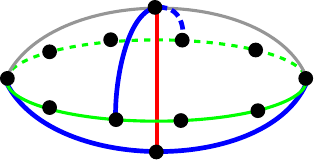}
\caption{The lens $L_0 = [1, \eta] * S^1$. The subsets $[1,\eta] \times \{0\}$ and $\{0\} \times S^1$ are drawn in red and green respectively. The arcs $\gamma_0$ and $\gamma_1$ are drawn in blue. They have length $\pi$ and cut each boundary half-sphere $\pi_0, \pi_1$ into two right-angled bigons $B_0^\pm$, $B_1^\pm$.}
\label{lens:fig}
\end{figure}
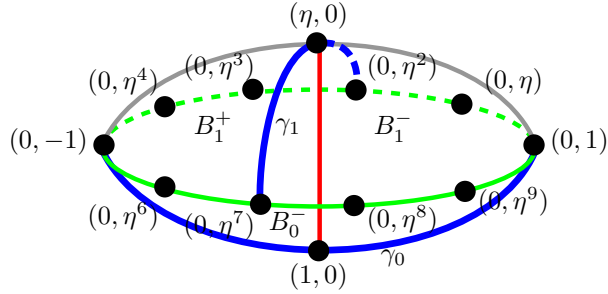

The group $D_{10}$ acts freely and transitively on the lenses $L_0, \ldots, L_9$. The geodesic
$$\gamma_j = \{\eta^j\} * \{\eta^{2j}, -\eta^{2j}\} \subset \pi_j$$
has length $\pi$ and cuts the half-sphere $\pi_j$ into two \emph{right-angled bigons} 
$$B_j^+ = \{\eta^j\} \times [\eta^{2j}, -\eta^{2j}], \qquad
B_j^- = \{\eta^j\} \times [-\eta^{2j}, \eta^{2j}].$$
See Figure \ref{lens:fig}. We have
$$r(\gamma_j) = \gamma_{j+2}, \quad r(B_j^\pm) = B_{j+2}^\pm, \qquad s(\gamma_j) = \gamma_{-j}, \quad s(B_j^\pm) = B_{-j}^\mp.$$

Each lens $L_j$ is a fundamental domain for the action of $D_{10}$, and the orbifold quotient $S^3/D_{10}$ is obtained from $L_j$ by folding its bigons along the separating geodesics: the bigon $B_j^+$ is identified with its companion $B_j^-$ via a $\pi$-rotation along $\gamma_j$, and the same is done for $B_{j+1}^+$ and $B_{j+1}^-$. 
The singular set $\gamma$ of $S^3/D_{10}$ is the image of the two geodesics $\gamma_j\cup \gamma_{j+1}$.

\subsection{A fundamental domain for $K$}
For every $j,k \in \matZ/10\matZ$ we define
$$T_{j,k} = [\eta^j, \eta^{j+1}] * [\eta^k, \eta^{k+1}].$$

This is a spherical tetrahedron with edge lengths $\pi/5,\pi/5,\pi/2,\pi/2,\pi/2,\pi/2$. Each lens $L_j$ is tessellated into 10 tetrahedra $T_{j,0}, \ldots, T_{j,9}$. We use $L_0$ as a fundamental domain for $D_{10}$ and get a tessellation of $S^3/D_{10}$ into the tetrahedra $T_{0,0}, \ldots, T_{0,9}$.
The group $K$ acts on $S^3/D_{10}$ preserving this tessellation, acting as 
\begin{gather*}
\nu_1(T_{0,k}) = T_{0,-k+1}, \quad \nu_2(T_{0,k}) = T_{0,k+5}, \\
\tau\colon T_{0,0} \to T_{0,2} \to T_{0,6} \to T_{0,4} \to T_{0,0}, \quad \tau(T_{0,8}) = T_{0,8}, \\
\tau\colon T_{0,1} \to T_{0,7} \to T_{0,5} \to T_{0,9} \to T_{0,1}, \quad \tau(T_{0,3}) = T_{0,3}.
\end{gather*}
Therefore we get two orbits 
$$\big\{T_{0,0},\ T_{0,1},\ T_{0,2},\ T_{0,4},\ T_{0,5},\ T_{0,6},\ T_{0,7},\ T_{0,9}\big\}, \quad
\big\{T_{0,3}, T_{0,8}\big\}.$$

\begin{figure}
\vspace{.3 cm}
\centering
\labellist
\pinlabel $(1,0)$ at -5 60
\pinlabel $(\eta,0)$ at 25 33
\pinlabel $(0,\eta^7)$ at 20 0
\pinlabel $(0,\eta^8)$ at 67 60
\pinlabel $(1+\eta,\eta^8+\eta^9)/\|1+\eta\|$ at 35 76
\small\hair 2pt
\endlabellist
\includegraphics[width = 4 cm]{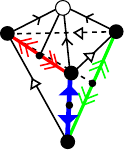}
\caption{A fundamental domain for the action of $K$ on $S^3/D_{10}$ is the union of two tetrahedra. Some faces are subdivided into two triangles, so there are 8 triangles overall. The group $K$ pairs the triangles with matching labeled edges. }
\label{fundamentalK:fig}
\end{figure}
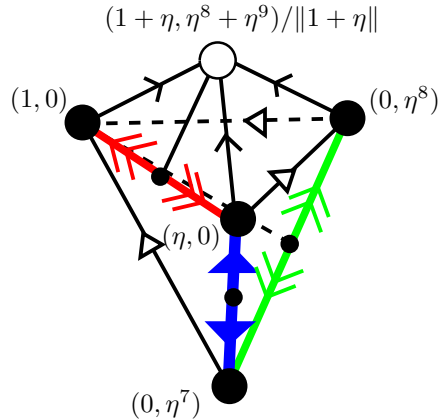

We can check that a fundamental domain for the action of $K$ consists of $T_{0,7}$ and one fourth of $T_{0,8}$ (that is, the subtetrahedron of $T_{0,8}$ with vertices the center of $T_{0,8}$ and three vertices of $T_{0,8}$), as in Figure \ref{fundamentalK:fig}. The group $K$ pairs the triangles in the unique way that matches the labeled edges.

\subsection{A fundamental domain for $G$}
We now determine a fundamental domain for the action of $G$ on the Euclidean space $\matR^3$.
A fundamental domain for the translation subgroup $H$ consists of four cubes, with total volume 32. Since $H<G$ has index 48, a fundamental domain for $G$ should have volume $32/48 = 2/3$. Using Piero della Francesca's formula (that is in modern terms the Cayley -- Menger determinant) we find that the two tetrahedra of Figure \ref{grid5:fig} have volume $1/2$ and $2/3$. The group $G$ acts freely and transitively on the tetrahedra of the first kind, while it acts on those of the second kind transitively but with stabilisers a cyclic group of order 4. A fundamental domain is the union of a tetrahedron of the first kinds and one fourth of one of the second kind, with total volume $1/2+1/6=2/3$ as requested. Such a fundamental domain is shown in Figure \ref{fundom:fig}.

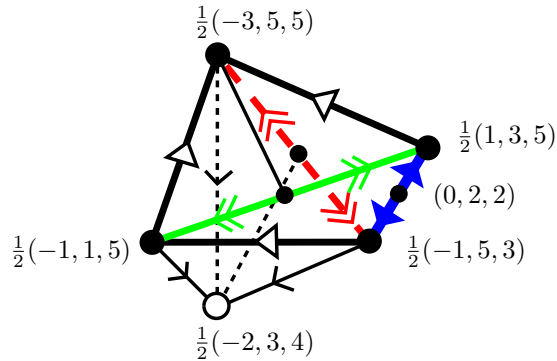
\begin{figure}
\vspace{.3 cm}
\centering
\labellist
\pinlabel $\frac 12(-1,5,3)$ at 71 15
\pinlabel $\frac 12(1,3,5)$ at 80 40
\pinlabel $\frac 12(-1,1,5)$ at -15 15
\pinlabel $\frac 12(-3,5,5)$ at 25 65
\pinlabel $\frac 12(-2,3,4)$ at 25 -5
\pinlabel $(0,2,2)$ at 73 27
\small\hair 2pt
\endlabellist
\includegraphics[width = 4 cm]{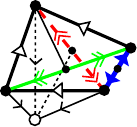}
\caption{A fundamental domain for the action of $G$ on $\matR^3$.}
\label{fundom:fig}
\end{figure}

\subsection{Proof of Lemma \ref{Xtilde:lemma}}
The Figures \ref{fundamentalK:fig} and \ref{fundom:fig} are combinatorially isomorphic and produce the orbifolds
$$(S^3/D_{10})/K, \qquad \matR^3/G.$$
The first orbifold is spherical, and the second is flat. By computing angles we see that all the dihedral angles of the identified edges in both fundamental domains sum to $2 \pi$, except the blue ones, which sum to $\pi$ in the first orbifold and to $2 \pi/3$ in the second. The two orbifolds are homeomorphic, and they differ only by the cone angles on the blue singular stratum that are $\pi$ and $2\pi/3$ respectively. 

We assign to the flat orbifold $\matR^3/G$ the elliptic structure of $(S^3/D_{10})/K$. At the universal cover level, this corresponds to giving to each tetrahedron of the triangulation $\Delta$ of $\matR^3$ the structure of a spherical tetrahedron with dihedral angles $\pi/5, \pi/5, \pi/2, \pi/2, \pi/2, \pi/2, \pi/2$ as we did in Section \ref{Xtilde:subsection}. This concludes the proof of Lemma \ref{Xtilde:lemma}.

\subsection{The link condition} \label{link:subsection}
Lemma \ref{Xtilde:lemma} describes a triangulation $\Delta$ of $\tilde X$ into spherical polyhedra. Recall that $\Delta$ is the triangulation of Euclidean space $\matR^3$ defined in Section \ref{triangulation:subsection}, where each tetrahedron is assigned a spherical structure. 

By \cite[Theorem 5.4]{BH} the space $\tilde X$ is $\CAT(1)$ if and only if the triangulation $\Delta$ satisfies the \emph{link condition} (the link of every vertex should be $\CAT(1)$) and $\tilde X$ contains no closed geodesic of length $<2\pi$. 

\begin{prop} \label{link:prop}
The triangulation of $\tilde X$ satisfies the link condition.
\end{prop}
\begin{proof}
The singular set of $\tilde X$ consists of lines with cone angle $3 \pi$, hence the link of a point in $\tilde X$ is either a standard $S^2$ or a $S^2$ with two antipodal points with cone angle $3 \pi$, and in both cases this is a $\CAT(1)$ space. 
\end{proof}

To prove that $\tilde X$ is CAT(1), it remains to show that $\tilde X$ has no closed geodesic of length $< 2\pi$. This is not obvious since, as we now see, the space $\tilde X$ contains infinitely many closed geodesics of length $2\pi$. 

\subsection{Closed geodesics in $\tilde X$}
The lens $L_0$ of Figure \ref{lens:fig} is a fundamental domain for $S^3/D_{10}$, and being simply connected it lifts to infinitely many lenses in $\tilde X$. The space $\tilde X$ is tessellated into  these infinitely many lenses, each consisting of 10 tetrahedra of the triangulation $\Delta$. The boundary of each lens in $\tilde X$ consists of two half-spheres, that intersect in a \emph{circular ridge}, that is a closed geodesic in $\tilde X$ made of 10 green edges. By acting with an isometry of $G$ that interchanges the green and red graphs we find another (somehow dual) isometric tessellation of $\tilde X$ into lenses, whose circular ridges are closed geodesics in $\tilde X$ made of 10 red edges.

In fact, any locally injective path in either the green or the red graph in Figure \ref{grid3:fig}, parametrized by arc length, is a geodesic, because the red and green edges are geodesics (of length $\pi/5$) that meet at an angle $\pi$ at every common endpoint. The shortest closed path consists of 10 edges, and it has length $2\pi$. 

We have just shown that $\tilde X$ contains infinitely many closed geodesics of length $2\pi$. As we said, we must prove that there are no closed geodesics in $\tilde X$ of length $< 2\pi$. We first rule out those that are disjoint from the singular set.

\begin{prop} \label{no:disjoint:prop}
The space $\tilde X$ contains no closed geodesic disjoint from the singular set with length $< 2 \pi$.  
\end{prop}
\begin{proof}
One such geodesic would project to a closed geodesic $\alpha$ with length $< 2 \pi$ in $S^3/D_{10}$ disjoint from the singular closed geodesic $\gamma$. We now show that there is only one such geodesic in $S^3/D_{10}$, of length $\pi$, that however does not lift to $\tilde X$.

The closed geodesic $\alpha$ lifts to a geodesic arc $\tilde \alpha$ in $S^3$, contained in a real vector plane $W \subset \matC^2$ preserved by some non-trivial $g \in D_{10}$ that acts on $W$ as a rotation. Up to conjugation we may suppose that either $g=r^i$ or $g=s$, see \eqref{rs:eqn}. If $g=r^i$ then $W =\matC \times \{0\}$ or $\{0\} \times \matC$, hence $\alpha$ intersects the singular set $\gamma$, a contradiction. If $g=s$, we must have $W = i\matR \times i \matR$ where $s$ acts as a rotation of angle $\pi$. 

Therefore $\alpha$ has length $\pi$, and it is in fact determined: it is the union of two segments of length $\pi/2$ that connect the midpoints of a red and a green edge in two tetrahedra of the triangulation of $S^3/D_{10}$. By analysing the triangulation of Figure \ref{grid3:fig} we see that such segments, if pursued, yield infinite lines in $\tilde X$. 
\end{proof}

It remains to consider the closed geodesics in $\tilde X$ that intersect the singular set. 

\subsection{Dihedral sectors}
Consider a blue singular line $\ell$ in $\tilde X$. A small tubular neighbourhood of $\ell$ is parametrized as an infinite cylinder $\ell \times D$ where $D$ is a disc with some small radius $r>0$ and cone angle $3 \pi$ at its center. The disc $D$ contains naturally six consecutive radii $r_1, \ldots, r_6$ that subdivide it into six \emph{sectors} $S_1,\ldots, S_6$ of angle $\pi/2$ each. The six radii are the projections of the green and red edges incident to $\ell$ (there are infinitely many such edges, but they project to six radii only).

The tubular neighbourhood of a blue edge $e \subset \ell$ is parametrized as $e \times D$, and it inherits a subdivision into six \emph{dihedral sectors} $e \times S_1, \ldots, e \times S_6$, each with dihedral angle $\pi/2$. Two consecutive sectors $e \times S_i$ and $e \times S_{i+1}$ intersect in a \emph{wall} $e \times r_i$.

We denote a dihedral sector at $e$ via the vector, based at the midpoint of $e$, of smallest integer coordinates, that points towards the dihedral sector and lies half-way between its two boundary walls. As an example, the origin $O$ in Figure \ref{chords:fig} is the midpoint of the edge $e$ with vertices $(1,1,1)$ and $(-1,-1,-1)$.  The walls of the 6 sectors at $e$ are parallel to the vectors
$$(-1,0,1), \quad (0,-1,1), \quad (1,-1,0), \quad (1,0,-1), \quad (0,1,-1), \quad (-1,1,0).$$ 

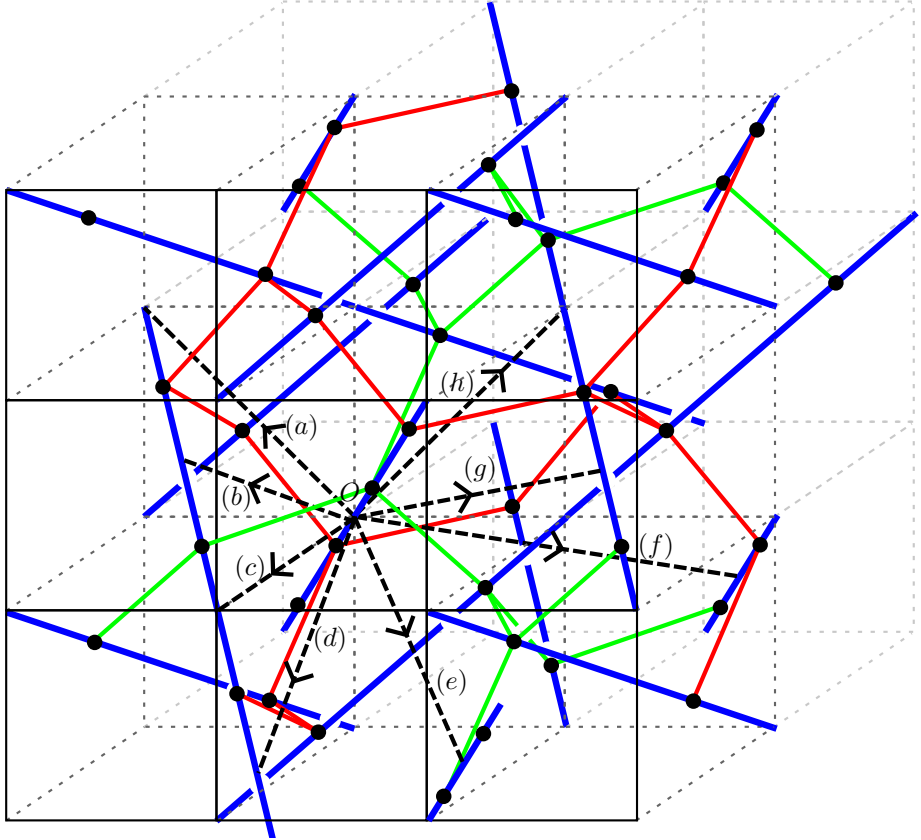
\begin{figure}
\centering
\labellist
\pinlabel $O$ at 127 122
\pinlabel $(c)$ at 92 95
\pinlabel $(b)$ at 87 120
\pinlabel $(a)$ at 110 145
\pinlabel $(d)$ at 120 70
\pinlabel $(h)$ at 165 160
\pinlabel $(g)$ at 173 130
\pinlabel $(f)$ at 235 103
\pinlabel $(e)$ at 163 55
\small\hair 2pt
\endlabellist
\includegraphics[width=12.5 cm]{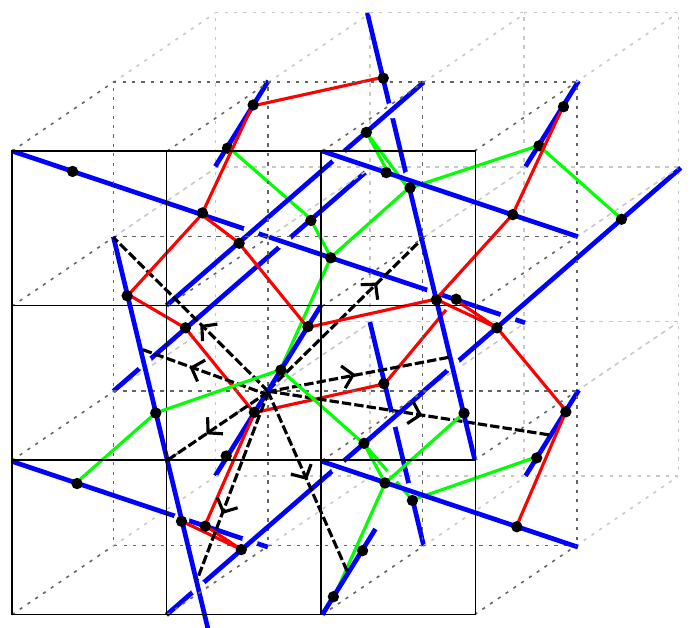}
\caption{Some oriented chords $(a), \ldots (f)$ exiting from the origin $O$. Every oriented geodesic chord in $\tilde X$ of length $< \pi$ transforms into one of these 8 chords after an isometry of $\tilde X$ and a small isotopy.}
\label{chords:fig}
\end{figure}

The sectors lying between $(-1,0,1)$ and $(0,-1,1)$, $(0,-1,1)$ and $(1,-1,0)$, $\ldots$ $(-1,0,1)$ and $(1,0,-1)$ are denoted by the vectors
$$(-1,-1,2), \quad (1,-2,1), \quad (2,-1,-1), \quad (1,1,-2), \quad (-1,2,-1), \quad (-2,1,1).$$

\subsection{Chords}
Let a \emph{chord} in $\tilde X$ be an arc that intersects the singular set at its endpoints. For instance, each red and green edge in Figure \ref{grid3:fig} is a \emph{geodesic chord} of minimal length $\pi/5$. Figure \ref{chords:fig} displays eight (non geodesic) oriented chords $(a), (b), \ldots, (h)$ connecting the midpoints of some blue edges. They are represented by the vectors
\begin{equation} \label{vectors:eqn}
\begin{gathered} 
(0,-2,2), \quad (1,-1,1), \quad (2,0,0), \quad (3, 1, -1), \\ 
(3, 3, -1), \quad (-1,3,-1), \quad (1,3,1), \quad (0,2,2).
\end{gathered}
\end{equation}
The tangent vectors at $O$ of these chords lie in the dihedral sectors
\begin{gather*}
(-1,-1,2), \quad (1,-2,1), \quad (2,-1,-1), \quad (1,1,-2), \\ 
(1,1,-2), \quad (-1,2,-1), \quad (-1,2,-1), \quad (-2,1,1).
\end{gather*}
Actually, those tangent to $(a)$ and $(d)$ lie in the intersection of two adjacent sectors, but we decide to assign them to the sectors indicated above. The other endpoint $P$ of these chords lie in some other blue edge, and its tangent vector at $P$ (pointing towards the interior of the chord) lies in the sectors
\begin{gather*}
(-2,1,-1), \quad (-2,1,-1), \quad (-2,1,-1), \quad (-2,1,-1), \\
(-1,-1,2), \quad (1,-2,1), \quad (-1,-1,-2), \quad (-1,-1,-2).
\end{gather*}
Similarly as above, each of the last two tangent vectors actually belongs to two adjacent sectors, and we decide to assign it to $(-1,-1,-2)$.

A \emph{small isotopy} between two chords $\alpha_0, \alpha_1$ is an isotopy $\alpha_t$ where each $\alpha_t$ is a chord, each endpoint of $\alpha_t$ is allowed to move only in a single edge of some blue line in Figure \ref{grid3:fig}, and the vectors tangent to $\alpha_t$ at both endpoints (directed towards the interior of $\alpha$) are each allowed to move only in one sector.

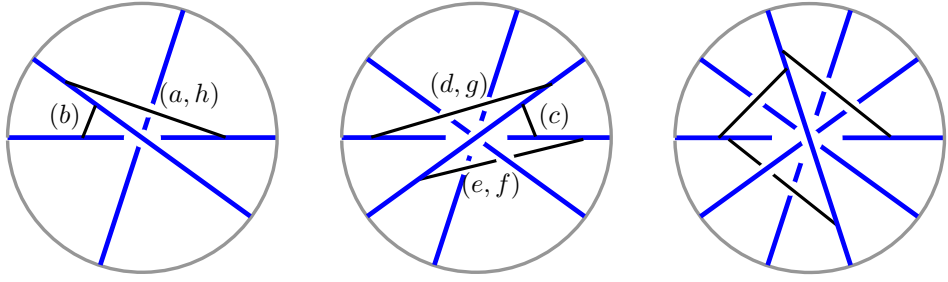
\begin{figure}
\centering
\labellist
\pinlabel $(b)$ at 28 75
\pinlabel $(a,h)$ at 87 85
\pinlabel $(e,f)$ at 230 42
\pinlabel $(d,g)$ at 215 90
\pinlabel $(c)$ at 260 75
\small\hair 2pt
\endlabellist
\includegraphics[width=12.5 cm]{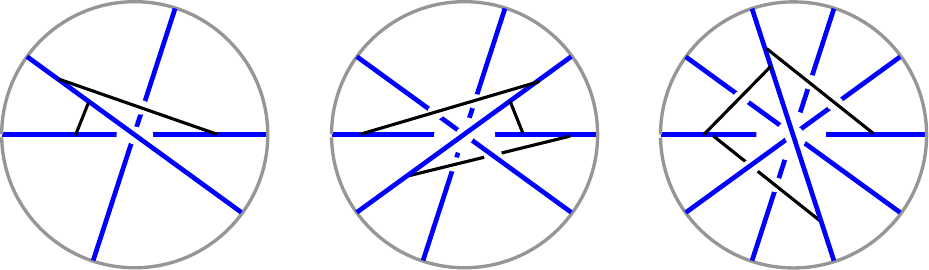}
\caption{The chords in $S^3$ of length $< \pi$ are of these types. The types in the right figure are actually equivalent to those of the left and central picture after an isometry of $S^3$.}
\label{chordsS3:fig}
\end{figure}

\begin{lemma} \label{abcdefgh:lemma}
The following hold:
\begin{enumerate}
\item There are no geodesic chords in $\tilde X$ of length $> \pi$. 
\item Every geodesic chord in $\tilde X$ of length $\pi$ is homotopic with fixed endpoints through geodesic chords of length $\pi$ to a concatenation of shorter geodesic chords (with total length $\pi$).
\item Every oriented geodesic chord $\alpha$ in $\tilde X$ of length $< \pi$ is transformed into one of the eight chords $(a), \ldots (h)$ from Figure \ref{chords:fig} after an isometry of $\tilde X$ and a small isotopy.  The chord $\alpha$ can be projected to $S^3/D_{10}$, lifted to $S^3$, and small isotoped until it looks as the corresponding chord $(a), \ldots, (h)$ drawn in Figure \ref{chordsS3:fig}. 
\end{enumerate} 
\end{lemma}
\begin{proof}
Recall that the singular set $\gamma$ of $S^3/D_{10}$ is a closed geodesic of length $2\pi$, whose preimage in $S^3$ is the union of 5 disjoint blue closed geodesics $\ell_1, \ldots, \ell_5$. Let a \emph{chord} in $S^3$ be an arc which intersects these geodesics in its endpoints.

Every chord in $\tilde X$ projects to one in $S^3/D_{10}$ and then lifts to one in $S^3$. Conversely, every chord in $S^3$ can be projected to $S^3/D_{10}$ and then lifted to a chord in $\tilde X$. Therefore everything reduces to classifying the geodesic chords in $S^3$ up to small isotopy ($S^3$ is also tessellated into geodesic tetrahedra as $\tilde X$, so the notions of blue edge, sector, and small isotopy apply also to chords in $S^3$). We deduce (1), since there are no geodesic chords longer than $\pi$ in $S^3$. If a geodesic chord in $S^3$ has length $\pi$, it connects two antipodal points of one $\ell_i$, and it can be isotoped through chords of length $\pi$ until it touches some other $\ell_j$. This proves (2).

It remains to classify the geodesic chords in $S^3$ of length $< \pi$ up to small isotopy.
Recall that the red and green closed geodesics $S^1 \times \{0\}$ and $\{0\} \times S^1$ are subdivided each into 10 red and green geodesic chords of length $\pi/5$.

\emph{Claim: Every geodesic chord in $S^3$ can be small-isotoped until it transforms (as a limit) into a concatenation of red or green geodesic chords}.

To prove the claim we will use this property: the only closed geodesics in $S^3$ that intersect at least four of the $\ell_j$ are the red and green closed geodesics. We deduce this property from the standard fact that there are always at most two projective lines intersecting 4 pairwise disjoint lines in $\matRP^3$.

Let $\alpha\subset S^3$ be a geodesic chord. It is easy to show that is always possible to small-isotope $\alpha$ by sliding its endpoints until either (i) one endpoint reaches the endpoint $v$ of some blue edge, or (ii) the geodesic chord transforms into a geodesic arc that intersects four lines $\ell_j$, and hence (by the stated property) it is a subarc of the red or green closed geodesic. In case (ii) the claim is proved. In case (i), we note that, seen from $v$, the other four geodesics $\ell_j$ look like 4 closed geodesics in $S^2$ intersecting at the poles seen from the origin of $\matR^3$. Therefore we can slide the other endpoint towards one pole until (ii) holds also in that case. 

The claim implies that every geodesic chord $\alpha$ is obtained (after a small isotopy) by perturbing a subarc $\alpha'$ of either the red or green closed geodesic. Up to symmetries we may suppose that $\alpha' = [1, \eta^k]\times\{0\}$ for some $k \in \{1,2,3,4,5\}$. We cannot have $k> 5$ because any perturbation of $\alpha'$ would have length $>\pi$. If $k=5$, the perturbed endpoints lie in the same $\ell_j$, hence the geodesic connecting them also lies in $\ell_j$, therefore this case can also be excluded.

If $k=1$, the arc $\alpha'$ is a single red arc. Therefore the chord is obtained by perturbing a red arc. The corresponding oriented chord in $\tilde X$ is, up to some symmetries of $\tilde X$, a perturbation of the green arc in Figure \ref{chords:fig} with endpoints  $\frac 12(1,1,1)$ and $\frac 12 (3,-1,1)$. There are two possible perturbations, giving (b) and (c).

Figure \ref{chordsS3:fig} displays the cases $k=2,3,4$. In each case we juxtapose $k$ lenses and look at them from above, so the $k$ red arcs in $\alpha'$ are orthogonal to the picture and projected to the central point, while the green closed geodesic is the boundary of the circle. Up to symmetries we get the perturbations listed in the figure, and each perturbation gives some chord in Figure \ref{chords:fig} of the corresponding type. Up to switching the red and green geodesics, the cases found with $k=4$ were already included in the cases with $k=2$ and $k=3$ so they can be ignored.
\end{proof}

We say that an oriented chord is \emph{of type} $(a), (b), \ldots, (h)$ if it can be transformed into the corresponding chord $(a), (b), \ldots, (h)$ in Figure \ref{chords:fig} after an isometry of $\tilde X$ and a small isotopy. By Lemma \ref{abcdefgh:lemma}, every geodesic chord $\alpha$ of length $< \pi$ is of some type. If both endpoints of $\alpha$ lie in the interior of two blue edges, the type of $\alpha$ is unique. If one endpoint lies at the intersection of two blue edges $e_1,e_2$, the chord $\alpha$ has two types, depending on the choice of the edge $e_1$ and $e_2$. If both endpoints lie in the intersection of two edges, we can check from Figure \ref{chordsS3:fig} that the chord is either a red or green segment, and it has four types $(b), (c), (c), (b)$ according to the four possible choices of edges.

\begin{lemma}
Let $\alpha$ be a geodesic chord in $\tilde X$ of length $< \pi$, with endpoints $P,Q$ contained in two blue edges $e_1,e_2$. We can isotope $\alpha$ through geodesic chords in a unique way by moving the endpoint $P$ in $e_1$ in some direction, until one of the following occurs: 
\begin{itemize}
\item the endpoint $P$ reaches an endpoint of $e_1$, or 
\item the geodesic chord $\alpha$ intersects some other blue edge in its interior; this case does not occur if $\alpha$ is of type $(b)$ or $(c)$.
\end{itemize}
\end{lemma}
\begin{proof}
The chord can be transported in $S^3$ where it becomes as in Figure \ref{chordsS3:fig}, and in $S^3$ the assertion is evident.
\end{proof}

We call this operation \emph{sliding} a geodesic chord of length $< \pi$. We can slide any geodesic chord $\alpha$ in $\tilde X$ of length $< \pi$ by moving one of its endpoints $P$ until either $P$ reaches the endpoint of a blue edge, or $\alpha$ crosses some other blue edge.

\begin{figure}
\centering
\labellist
\pinlabel $(b)$ at 28 75
\pinlabel $(a,h)$ at 227 85
\pinlabel $(e,f)$ at 500 42
\pinlabel $(d,g)$ at 485 90
\pinlabel $(c)$ at 380 75
\small\hair 2pt
\endlabellist
\includegraphics[width=12.5 cm]{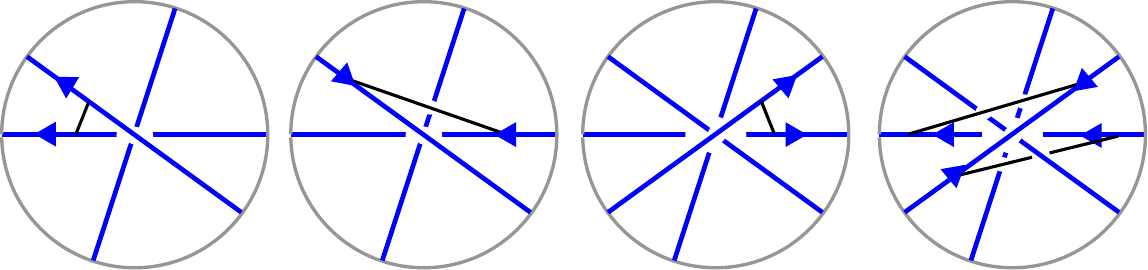}
\caption{Every chord of type $(a), \ldots, (h)$ induces an orientation on the blue edges containing its endpoints.}
\label{chords_orient_e:fig}
\end{figure}
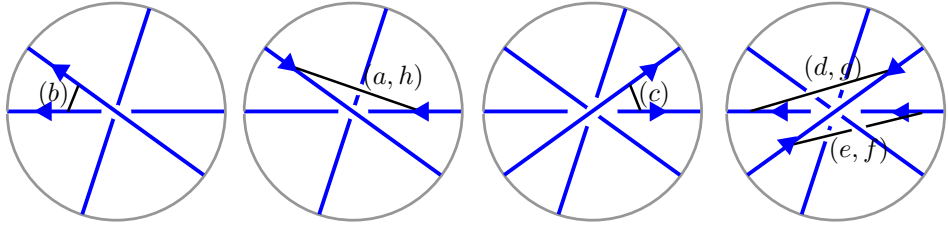

Let $\alpha$ be a geodesic chord of length $<\pi$ and let $e_1, e_2$ be two blue edges containing its endpoints (if both endpoints of $\alpha$ lie in the interior of some edges, then $e_1, e_2$ are determined; if not, the choice of $e_1,e_2$ fixes a type for $\alpha$). The chord $\alpha$ induces an orientation of both $e_1,e_2$ as follows: we transport $\alpha, e_1, e_2$ in $S^3$ as in Figure \ref{chords_orient_e:fig}, and orient the edges $e_1$ and $e_2$ as shown in the figure. This choice of orientation is in fact motivated by the following fact.

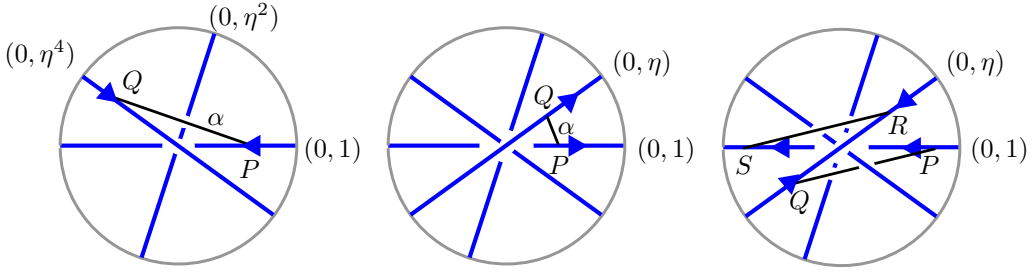
\begin{figure}
\centering
\vspace{.3 cm}
\labellist
\pinlabel $P$ at 103 53
\pinlabel $Q$ at 40 100
\pinlabel $\alpha$ at 85 80
\pinlabel $(0,1)$ at 148 63
\pinlabel $(0,\eta^2)$ at 100 135
\pinlabel $(0,\eta^4)$ at -10 115
\pinlabel $(0,1)$ at 328 63
\pinlabel $(0,\eta)$ at 315 110
\pinlabel $P$ at 270 55
\pinlabel $Q$ at 262 92
\pinlabel $\alpha$ at 274 75
\pinlabel $(0,1)$ at 508 63
\pinlabel $(0,\eta)$ at 495 110
\pinlabel $P$ at 470 55
\pinlabel $Q$ at 400 35
\pinlabel $R$ at 453 77
\pinlabel $S$ at 370 55
\small\hair 2pt
\endlabellist
\includegraphics[width= 12 cm]{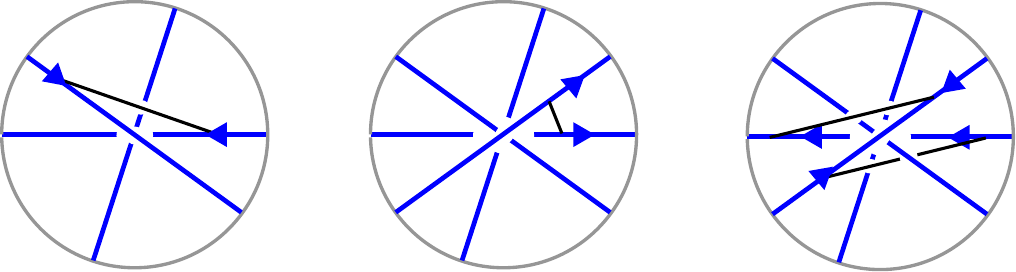}
\caption{A chord of type $(a)$ or $(h)$, one of type $(c)$, one of type $(e)$ or $(f)$, and one of type $(d)$ or $(g)$.}
\label{slide:fig}
\end{figure}

\begin{lemma} \label{slide:lemma}
Suppose that $\alpha, e_1,e_2$ is not of type $(b)$. 
If we slide $\alpha$ by moving one endpoint in the direction of the edge $e_1,e_2$ that contains it, the length of $\alpha$ strictly decreases.
\end{lemma}
\begin{proof}
We look at Figure \ref{slide:fig}. If $\alpha$ is of type $(a)$ or $(h)$, its endpoints are
$$P = (\cos \theta, \sin \theta), \quad Q = (\eta^2 \cos \varphi, \eta^4 \sin \varphi)$$
for some $\theta, \varphi \in [0, \pi/2]$. The geodesic chord $\alpha$ has length
\begin{align*}
f(\theta, \varphi) = d(P,Q) & = \arccos (P\cdot Q) = \arccos (\Re (\eta^2) \cos \theta \cos \varphi + \Re (\eta^4) \sin \theta \sin \varphi) \\
& = \arccos \big(\tfrac {\sqrt 5 -1}4 \cos \theta \cos \varphi - \tfrac{\sqrt 5 + 1}4 \sin \theta \sin \varphi).
\end{align*}
We have
\begin{equation} \label{ah:eqn}
\frac {\partial f}{\partial \theta} = 
- \frac{-\tfrac {\sqrt 5 -1}4 \sin \theta \cos \varphi - \tfrac{\sqrt 5 + 1}4 \cos \theta \sin \varphi}{\sin d(P,Q) } > 0
\end{equation}
for every $\theta \neq 0,\frac \pi 2$, and similarly for $\partial f/\partial \varphi$, so this proves the claim. If $\alpha$ is of type $(c)$, its endpoints are
$$P = (\cos \theta, \sin \theta ), \qquad Q = (\eta^3\cos \varphi, \eta\sin \varphi )$$ 
for some $\theta, \varphi \in [0, \pi/2]$. The geodesic chord $\alpha$ has length
\begin{align*}
f(\theta, \varphi) = d(P,Q) & = \arccos (P\cdot Q) = \arccos (\Re (\eta^3) \cos \theta \cos \varphi + \Re \eta \sin \theta \sin \varphi) \\
& = \arccos \big(-\tfrac {\sqrt 5 -1}4 \cos \theta \cos \varphi + \tfrac{\sqrt 5 + 1}4 \sin \theta \sin \varphi).
\end{align*}
We have
\begin{equation} \label{c:eqn}
\frac {\partial f}{\partial \theta} = 
- \frac{\tfrac {\sqrt 5 -1}4 \sin \theta \cos \varphi + \tfrac{\sqrt 5 + 1}4 \cos \theta \sin \varphi}{\sin d(P,Q) } < 0
\end{equation}
for every $\theta \neq 0,\frac \pi 2$, and similarly for $\partial f/\partial \varphi$, so this proves the claim. If $\alpha$ is of type 
$(e)$ or $(f)$, its endpoints are
$$P = (\cos \theta, \sin \theta ), \qquad Q = (\eta^3\cos \varphi, \eta^6\sin \varphi )$$ 
The chord $\alpha$ passes below one blue edge and above another blue edge in Figure \ref{slide:fig}. One checks that this holds precisely when 
\begin{equation} \label{sym:eqn}
\frac {3-\sqrt 5}2 \tan\theta < \tan\varphi < \frac{3+\sqrt 5}2 \tan \theta.
\end{equation}
This condition is symmetric in $\theta$ and $\varphi$. The geodesic chord $\alpha$ has length
\begin{align*}
f(\theta, \varphi) = d(P,Q) & = \arccos (P\cdot Q) = \arccos (\Re (\eta^3) \cos \theta \cos \varphi + \Re \eta^6 \sin \theta \sin \varphi) \\
& = \arccos \big(-\tfrac {\sqrt 5 -1}4 \cos \theta \cos \varphi - \tfrac{\sqrt 5 + 1}4 \sin \theta \sin \varphi).
\end{align*}
We have
\begin{equation} \label{ef:eqn}
\frac {\partial f}{\partial \theta} = 
- \frac{\tfrac {\sqrt 5 -1}4 \sin \theta \cos \varphi - \tfrac{\sqrt 5 + 1}4 \cos \theta \sin \varphi}{\sin d(P,Q) } > 0
\end{equation}
for every $\theta \neq 0,\frac \pi 2$, because $\tan \varphi > \tan \theta (3-\sqrt 5)/2 $,
and similarly for $\partial f/\partial \varphi$, so this proves the claim. Finally, if $\alpha$ is of type $(d)$ or $(g)$, its endpoints are 
$$R = (\eta^3 \cos \theta, \eta\sin \theta ), \qquad S = (\cos \varphi, -\sin \varphi ).$$ 
The chord $\alpha$ passes above two blue edges in Figure \ref{slide:fig}. This holds precisely when 
\begin{equation*}
\tan\varphi > \frac{3+\sqrt 5}2 \tan \theta.
\end{equation*}
The geodesic chord $\alpha$ has length
\begin{align*}
f(\theta, \varphi) = d(R,S) & = \arccos (R\cdot S) = \arccos (\Re (\eta^3) \cos \theta \cos \varphi + \Re \eta^4 \sin \theta \sin \varphi) \\
& = \arccos \big(-\tfrac {\sqrt 5 -1}4 \cos \theta \cos \varphi - \tfrac{\sqrt 5 + 1}4 \sin \theta \sin \varphi).
\end{align*}
We have
\begin{equation} \label{dg:eqn}
\begin{aligned} 
\frac {\partial f}{\partial \theta} & = 
- \frac{\tfrac {\sqrt 5 -1}4 \sin \theta \cos \varphi - \tfrac{\sqrt 5 + 1}4 \cos \theta \sin \varphi}{\sin d(P,Q) } > 0 \\
\frac {\partial f}{\partial \varphi} & = 
- \frac{\tfrac {\sqrt 5 -1}4 \cos \theta \sin \varphi - \tfrac{\sqrt 5 + 1}4 \sin \theta \cos \varphi}{\sin d(P,Q) } < 0
\end{aligned}
\end{equation}
for every $\theta, \varphi \neq 0,\frac \pi 2$, because $\tan \varphi > \tan \theta (3+\sqrt 5)/2 $, so this proves the claim. 
The proof is complete. 
\end{proof}

\begin{cor} \label{abcdefgh:cor}
Every geodesic chord of type $(a), (b), (c), (d), (e), (f), (g), (h)$ has correspondingly length at least
$$ 2\pi/5, \quad \pi/5, \quad \pi/5, \quad \pi/2, \quad 3\pi/5, \quad 3\pi/5, \quad \pi/2, \quad 2\pi/5.$$
\end{cor}
\begin{proof}
First, we note that any geodesic chord has length $\geq \pi/5$, and the equality is reached by the green and red chords: this can be deduced easily by computing the distance between points in distinct blue geodesics $\ell_j$ in $S^3$.

By the previous lemma, we can slide any chord of type $(a), (c), (e), (f), (h)$ decreasing its length until we either reach the final endpoints of the directed blue edges, or we cross some other blue edge. In the case $(c)$ we end up with a green or red chord, of length $\pi/5$. In the cases $(a), (h)$ we end up with a composition of two green or red chords, with total length $2\pi/5$. 

In the cases $(e), (f)$ we slide the endpoints until the chord intersects some blue edge in its interior. The endpoints of the chord are
$$P = (\cos \theta, \sin \theta ), \qquad Q = (\eta^3\cos \varphi, \eta^6\sin \varphi )$$ 
and we have seen in the proof of Lemma \ref{slide:lemma} that this new intersection with a blue edge arises precisely when $\tan \varphi = \tan \theta (3+\sqrt 5)/2$ (up to exchanging $P$ and $Q$). From this equality we deduce that
$$ Q = \frac{(\eta^3\cos \theta, \frac{3+\sqrt 5}2\eta^6 \sin \theta)}{\sqrt{\cos^2 \theta + \left(\frac{3+\sqrt 5}{2}\right)^2 \sin^2 \theta}}.$$
The chord has then length
\begin{align*}
g(\theta) =d(P,Q) & = \arccos (P\cdot Q) = \arccos \frac{\Re \eta^3 \cos^2 \theta + \Re \eta^6 \frac{3+\sqrt 5}2 \sin^2 \theta}
{\sqrt{\cos^2 \theta + \left(\frac{3+\sqrt 5}{2}\right)^2 \sin^2 \theta}} \\
& =
\arccos \frac{-(\sqrt 5 - 1) \cos^2 \theta - (2\sqrt 5 + 4) \sin^2 \theta}
{4\sqrt{\cos^2 \theta + \left(\frac{3+\sqrt 5}{2}\right)^2 \sin^2 \theta}}.
\end{align*}
We have $g(\theta) \geq 3\pi/5$ for every $\theta \in [0,\pi/2]$. Indeed $g(0) = \arccos(\Re \eta^3) = 3\pi/5$ and $g'(\theta) \geq 0$ for all $\theta \in [0,\pi/2]$. 

In the cases $(d)$, $(g)$ we slide the endpoints $R,S$ until they reach the endpoints $(\eta^3,0)$, $(0,-1)$ of the blue edges, that lie at distance $\pi/2$.
\end{proof}

\subsection{Geodesic multichords}
Let a \emph{multichord} $\alpha$ be a curve in $\tilde X$ that is a concatenation $\alpha = \alpha_1 * \cdots * \alpha_k$ of finitely many chords $\alpha_1, \ldots, \alpha_k$. A multichord $\alpha = \alpha_1 * \cdots * \alpha_k$ parametrized by arc length is geodesic precisely when each $\alpha_i$ is a geodesic chord, and two consecutive chords $\alpha_j, \alpha_{j+1}$ meet at an angle $\geq \pi$.
For instance, every locally injective path in the red or green graph, parametrized by arc length, is a geodesic multichord. 

A multichord is \emph{closed} if its endpoints coincide. A closed multichord that is parametrised by arc length is a \emph{closed geodesic multichord} precisely when each chord is geodesic, and every two consecutive chords meet at an angle $\geq \pi$, interpreted cyclically.
We already know that the red and green graphs contain infinitely many closed geodesic multichords of length $2 \pi$.

\begin{prop} \label{gm:prop}
Every closed geodesic in $\tilde X$ of length smaller than $2\pi$ is a closed geodesic multichord.
\end{prop}
\begin{proof}
Let $\alpha$ be a closed geodesic shorter than $2\pi$. By Proposition \ref{no:disjoint:prop} it intersects the singular set. If it is at some point tangent to it, it is entirely contained in it, and is hence not closed. Therefore it intersects the blue lines transversely at finitely many points, and is thus a closed geodesic multichord.
\end{proof}

To prove Theorem \ref{CAT1:teo}, it remains to check that there are no closed geodesic multichords in $\tilde X$ shorter than $2\pi$. To this purpose we now find some restrictions on geodesic multichords. Since there are many types $(a), \ldots, (h)$ of geodesic chords, and these can be combined in many ways, unfortunately there will be a certain number of cases to consider in the next pages. The techniques used in the proofs will sometimes be very similar, however.

Let $\alpha = \alpha_1 * \alpha_2$ be a geodesic multichord, consisting of two geodesic arcs $\alpha_1, \alpha_2$ concatenated at some point $P$ in some blue line. Let $e$ be a blue edge containing $P$, and $s_1, s_2$ be two sectors at $e$ containing the vectors tangent to $\alpha_1, \alpha_2$ at $P$ directed towards the interiors of $\alpha_1$ and $\alpha_2$. 

\begin{lemma} \label{sectors:lemma}
The two sectors $s_1, s_2$ do not coincide. If they are adjacent, either
\begin{enumerate}
\item the chords $\alpha_1, \alpha_2$ are two consecutive red or green segments, or
\item the multichord $\alpha$ has length $\pi$ and is homotopic with fixed endpoints through geodesic chords of length $\pi$ to a concatenation of 5 red or green segments.
\end{enumerate}
\end{lemma}
\begin{proof}
Since $\alpha$ is a geodesic, the two chords $\alpha_1$ and $\alpha_2$ meet with an angle $\theta \geq \pi$ at $P$. Remember that each sector has dihedral angle $\pi/2$. Therefore we must have $s_1 \neq s_2$, and if $s_1,s_2$ are adjacent, the chords $\alpha_1,\alpha_2$ must lie in the distant (that is, different from $s_1\cap s_2$) boundaries of the two sectors $s_1,s_2$, so that $\theta = \pi$.

In that case $\alpha_1 * \alpha_2$ is a geodesic arc contained in one boundary half-sphere of some lens containing the blue edge $e$. It intersects the interior of the half-sphere, it has length $\pi$, and it can be homotoped with fixed endpoints through geodesic arcs until it lies in the boundary of the half-sphere, where it has become a concatenation of 5 red or green arcs. If it is already contained in the boundary, then $\alpha_1, \alpha_2$ are two consecutive red or green segments.
\end{proof}

Recall that every geodesic chord induces an orientation on the two edges containing its endpoints, as prescribed by Figure \ref{chords_orient_e:fig}. 

\begin{lemma} \label{nob:lemma}
Let $\alpha = \alpha_1*\alpha_2$ be the concatenation of two geodesic chords, each of some type $(a), (c), (d), (e), (f), (g) $ or $(h)$, which induce the same orientation on the blue edge $e$ containing their common endpoint. One of the following holds:
\begin{enumerate}
\item Both $\alpha_1$ and $\alpha_2$ are green or red edges;
\item The multichord $\alpha$ can be homotoped, with fixed endpoints and monotonically decreasing length, to a shorter multichord $\alpha_1'*\alpha_2'$.
\end{enumerate}
\end{lemma}
\begin{proof}
Let $P$ be the common endpoint of the two chords $\alpha_1,\alpha_2$, contained in a blue edge $e$ with endpoints $v_1,v_2$, oriented from $v_1$ to $v_2$. If $P\neq v_2$, we can slide $P$ towards $v_2$ and decrease the length of both $\alpha_1,\alpha_2$ monotonically by Lemma \ref{slide:lemma}. 

If $P = v_2$, and both $\alpha_1$ and $\alpha_2$ are green or red edges, we are done. If $P=v_2$ and at least one of $\alpha_1, \alpha_2$ is not a green or red edge, we actually have $\partial f/\partial \theta \neq 0$ also at $P$ for a least one chord in the formula \eqref{ah:eqn}, \eqref{c:eqn} \eqref{ef:eqn} or \eqref{dg:eqn}, hence by sliding $P$ past $v_2$ to the adjacent edge we may further decrease the length of $\alpha$.  
\end{proof}

The type $(b)$ is the only type excluced in the previous lemma. We get a similar conclusion with three geodesic chords, with the middle one being of type $(b)$.

\begin{lemma} \label{xby:lemma}
Let $\alpha = \alpha_1*\alpha_2*\alpha_3$ be the concatenation of three geodesic multichords, where $\alpha_2$ is of type $(b)$ and each of $\alpha_1,\alpha_3$ is of some type $(a), (c), (d), (e), (f), (g)$, or $(h)$. Suppose that both $\alpha_1, \alpha_2$ and $\alpha_2, \alpha_3$ induce
the same orientation on the blue edge containing their common endpoint. One of the following holds:
\begin{enumerate}
\item The chords $\alpha_1$, $\alpha_2$ and $\alpha_3$ are all green or red edges;
\item The multichord $\alpha$ can be homotoped, with fixed endpoints and monotonically decreasing length, to a shorter multichord $\alpha_1'*\alpha_2'*\alpha_3'$.
\end{enumerate}
\end{lemma}
\begin{proof}
Let $P$ and $Q$ be the common endpoints of $\alpha_1, \alpha_2$ and $\alpha_2, \alpha_3$, contained in some blue edges $e_1$ and $e_2$, each $e_i$ having endpoints $v_{i,1}, v_{i,2}$, oriented from $v_{i,1}$ to $v_{i,2}$. If $P\neq v_{1,2}$ and $Q \neq v_{2,2}$, if we move $P$ and $Q$ at the same rate towards $v_{1,2}$ and $v_{2,2}$, the length of $\alpha_2$ decreases. To prove this, we transport $\alpha_2$ iin $S^3$ and get 
$$P = (\cos \theta, \sin \theta), \quad Q = (\eta \cos \varphi, \eta^2 \sin \varphi).$$
We have
$$ f(\theta, \varphi) = d(P,Q) = \arccos( \Re \eta \cos \theta \cos \varphi + \Re \eta^2 \sin \theta \sin \varphi)$$
and therefore
$$ \frac {\partial f}{\partial \theta} + \frac{\partial f}{\partial \varphi} =
\frac{(\Re \eta - \Re \eta^2)\sin(\theta + \varphi)}{\sin d(P,Q)} > 0.$$

Thus by pushing $P$ and $Q$ towards $v_{1,2}$ and $v_{2,2}$ at the same rate each we decrease the length of each of $\alpha_1, \alpha_2,\alpha_3$. Analogously we can check that if $P= v_{1,2}$ and $Q \neq v_{2,2}$, by sliding $Q$ towards $v_{2,2}$ we decrease the lengths of both $\alpha_2, \alpha_3$, and the same with $P$ and $Q$ having opposite roles. We can then conclude as in the previous lemma.
\end{proof}

We can also say something useful in one case where the induced orientations on the common edges do not match.

\begin{lemma} \label{bc:lemma}
Let $\alpha = \alpha_1 * \alpha_2$ be a geodesic multichord, where the chords $\alpha_1, \alpha_2$ are of type $(b)$, $(c)$ respectively, and induce opposite orientations on the blue edge $e$ containing their common endpoint. Suppose that the two blue edges $e_1,e_2$ containing the other endpoints of $\alpha_1, \alpha_2$ are not parallel. Then $\alpha$ has length $\geq 3\pi/5$.
\end{lemma}
\begin{proof}
Let $P,Q$ and $Q,R$ be the endpoints of $\alpha_1$ and $\alpha_2$. We have $P\in e_1, Q \in e, R \in e_2$.
Up to isometries of $\tilde X$, we may suppose that $e$ contains the origin $O$, and $\alpha_1, \alpha_2$ are small isotopic to the vectors $(1,-1,1)$ and $(-2,0,0)$ exiting from $O$. The edges $e_1,e_2$ are both parallel to $(1,1,-1)$, see Figure \ref{chords:fig}.

\begin{figure}
\centering
\labellist
\pinlabel $P$ at 100 57
\pinlabel $Q$ at 85 100
\pinlabel $R$ at 45 95
\pinlabel $e_1$ at 120 57
\pinlabel $e$ at 75 117
\pinlabel $e_2$ at 55 120
\pinlabel $(0,1)$ at 145 68
\pinlabel $(0,\eta)$ at 133 105
\pinlabel $(0,\eta^2)$ at 100 135
\pinlabel $(0,\eta^3)$ at 40 135
\pinlabel $\alpha_1$ at 107 77
\pinlabel $\alpha_2$ at 65 102
\small\hair 2pt
\endlabellist
\includegraphics[width= 3.5 cm]{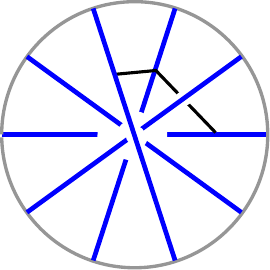}
\caption{The chords $\alpha_1, \alpha_2$ in $S^3$.}
\label{chords2:fig}
\end{figure}
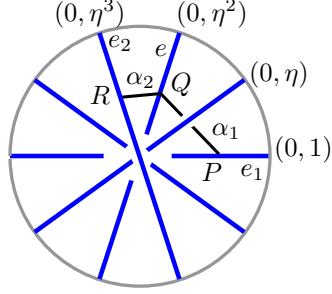

We project and lift $\alpha_1, \alpha_2$ to $S^3$ as in Figure \ref{chords2:fig}. We have
$$P = (\cos \theta, \sin \theta), \quad Q = (\eta \cos \varphi, \eta^2 \sin \varphi), \quad
R = (\eta^4 \cos \psi, \eta^3 \sin \psi)$$
for some $\theta, \varphi, \psi \in [0, \pi/2]$. Since $\alpha = \alpha_1 * \alpha_2$ is geodesic, the two chords $\alpha_1,\alpha_2$ make an angle $\beta \geq \pi$ at $Q$ below $e$. Therefore the geodesic segment $s$ connecting $P$ and $R$ passes above $e$. A computation shows that this holds precisely when
\begin{equation} \label{PQR:eqn}
\tan \psi \leq \left( \frac {\Im \eta^2}{\Im \eta} \right)^2 \tan \theta = \frac {3+\sqrt 5}2 \tan \theta.
\end{equation}
We now show that 
$$f(\theta, \varphi, \psi) = d(P,Q) + d(Q,R) \geq \frac {3\pi}5 $$
for every $(\theta, \varphi, \psi) \in [0,\pi/2]^3$ such that \eqref{PQR:eqn} holds. This will conclude the proof. As in the proof of the previous lemma we get $\partial f/\partial \psi \leq 0$ and therefore it suffices to prove the assertion when the equality holds: 
$$\tan \psi = \frac {3+\sqrt 5}2 \tan \theta.$$
We deduce that
$$ R = \frac{(\eta^4\cos \theta, \frac{3+\sqrt 5}2\eta^3 \sin \theta)}{\sqrt{\cos^2 \theta + \left(\frac{3+\sqrt 5}{2}\right)^2 \sin^2 \theta}}.$$
In this case the segment $s$ connecting $P$ and $R$ intersects $e$ at some point $Q'$. 
If we fix $P$ and $R$ and vary $Q$, the minimum of $f$ is clearly reached when $Q=Q'$. So we also may suppose that $P,Q,R$ are aligned and get
\begin{align*}
f(\theta, \varphi, \psi) =d(P,R) & = \arccos (P\cdot R) = \arccos \frac{-\Re \eta \cos^2 \theta - \Re \eta^2 \frac{3+\sqrt 5}2 \sin^2 \theta}
{\sqrt{\cos^2 \theta + \left(\frac{3+\sqrt 5}{2}\right)^2 \sin^2 \theta}} \\
& =
\arccos \frac{-(\sqrt 5 + 1)}
{4\sqrt{\cos^2 \theta + \left(\frac{3+\sqrt 5}{2}\right)^2 \sin^2 \theta}} = g(\theta).
\end{align*}
Finally we have $g(\theta) \geq 3\pi/5$ for every $\theta \in [0,\pi/2]$. Indeed $g(\pi/2) = \arccos((1-\sqrt 5)/4) = 3\pi/5$ and $g'(\theta) \leq 0$ for all $\theta \in [0,\pi/2]$. 
\end{proof}

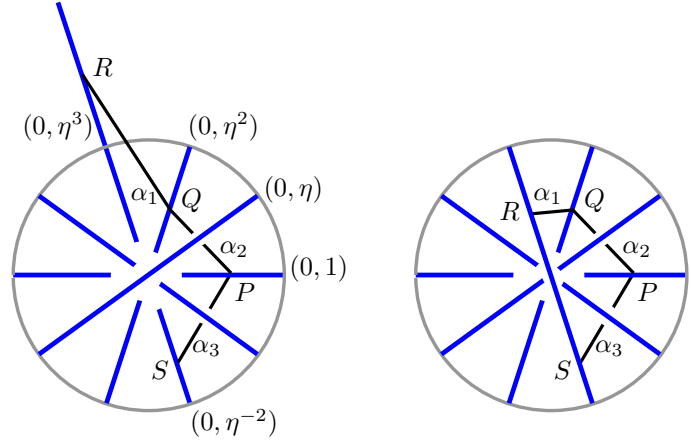
\begin{figure}
\centering
\labellist
\pinlabel $P$ at 110 57
\pinlabel $Q$ at 85 100
\pinlabel $R$ at 43 163
\pinlabel $S$ at 70 20
\pinlabel $(0,\eta^{-2})$ at 105 -5
\pinlabel $(0,1)$ at 145 68
\pinlabel $(0,\eta)$ at 133 105
\pinlabel $(0,\eta^2)$ at 100 135
\pinlabel $(0,\eta^3)$ at 22 135
\pinlabel $\alpha_1$ at 63 102
\pinlabel $\alpha_2$ at 105 77
\pinlabel $\alpha_3$ at 92 30
\pinlabel $P$ at 300 57
\pinlabel $Q$ at 275 100
\pinlabel $R$ at 235 95
\pinlabel $S$ at 260 20
\pinlabel $\alpha_1$ at 253 102
\pinlabel $\alpha_2$ at 295 77
\pinlabel $\alpha_3$ at 282 30
\small\hair 2pt
\endlabellist
\includegraphics[width= 9 cm]{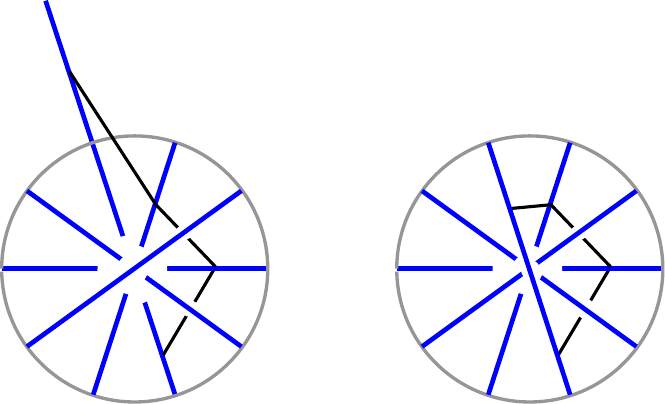}
\caption{The multichord $\alpha = \alpha_1 * \alpha_2 * \alpha_3$ in $S^3$.}
\label{chords3:fig}
\end{figure}

We now turn to study some geodesic multicurves that connect two points lying in the same blue line of $\tilde X$. Our guess is that any such multicurve should have length $\geq \pi$, and we prove this in some cases.

\begin{lemma} \label{bbc:lemma}
Let $\alpha = \alpha_1 * \alpha_2 * \alpha_3 $ be a geodesic multichord that connects two points that belong to the same blue line of $\tilde X$. Suppose that $\alpha_1,\alpha_2,\alpha_3$ are of type $(b),(b),(c)$ or $(c),(b),(b)$. Then $\alpha$ has length $\geq \pi$. 
\end{lemma}
\begin{proof}
We suppose that they are of type $(c), (b), (b)$. After projecting to $S^3/D_{10}$ and lifting to $S^3$ the geodesic multichord $\alpha$ is as in Figure \ref{chords3:fig}-(left). We have
\begin{gather*}
P = (\cos \theta, \sin \theta), \quad Q = (\eta \cos \varphi, \eta^2 \sin \varphi), \\
R = (\eta^4 \cos \psi, \eta^3 \sin \psi), \quad S = (\eta^{-1}\cos \alpha, \eta^{-2} \sin \alpha)
\end{gather*}
for some $\theta, \varphi, \psi, \alpha \in [0, \pi/2]$. The segment $RS$ in the figure contains the points $(\eta^{-1},0)$ and $(0, \eta)$ and it lifts to $\tilde X$. Also the evident homotopy between $RS$ and $\alpha$ lifts to $\tilde X$. Since $\alpha = \alpha_1 * \alpha_2 * \alpha_3$ is geodesic, we deduce that $\alpha_1,\alpha_2$ and $\alpha_2, \alpha_3$ both make an angle $\geq \pi$ at $Q$ and $P$ below the blue edges that contain them. This implies that the segments $PR$ and $QS$ lie above the blue edges containing $Q$ and $R$, and this holds precisely when the following inequalities are fulfilled:
\begin{equation} \label{PQRS:eqn}
\tan \psi \leq \frac {3+\sqrt 5}2 \tan \theta, \qquad \varphi \leq \alpha.
\end{equation}
The first inequality was already discovered in \eqref{PQR:eqn}. The picture can also be represented as in Figure \ref{chords3:fig}-(right), where one should beware that $S$ belongs to a blue line that lies below the blue line containing $R$. This picture is analogous to Figure \ref{chords2:fig}, with $S$ added, hence we get \eqref{PQR:eqn} also here. The right inequality $\varphi \leq \alpha$ is proved similarly. The length of $\alpha$ is
\begin{align*}
f(\theta, \varphi, \psi, \alpha) & = d(R,Q) + d(Q,P) + d(P,S) \\
& =
\arccos(\Re \eta^3 \cos \varphi \cos \psi + \Re \eta \sin \varphi \sin \psi)  \\
& \quad + \arccos (\Re \eta \cos \theta \cos \varphi + \Re \eta^2 \sin \theta \sin \varphi) \\
& \quad + \arccos (\Re \eta \cos \alpha \cos \theta + \Re \eta^2 \sin \alpha \sin \theta).
\end{align*}
Since $\alpha$ is a geodesic, we have
$$
0 = \frac{\partial f}{\partial \theta} = 
- \frac{-\Re \eta \sin \theta \cos \varphi + \Re \eta^2 \cos \theta \sin \varphi}{\sin d(Q,P)}- \frac{-\Re \eta \sin \theta \cos \alpha + \Re \eta^2 \cos \theta \sin \alpha}{\sin d(P,S)}.
$$
This implies in particular that the two addenda have opposite signs, and since $\varphi \leq \alpha$, after dividing the addenda by $\cos\theta \cos \varphi$ and $\cos \alpha \cos \theta$ we must have
$$-\Re \eta \tan \theta + \Re \eta^2 \tan \varphi \leq 0, \qquad
-\Re \eta \tan \theta + \Re \eta^2 \tan \alpha \geq 0.$$
In particular we get
$$\tan \alpha \geq \frac{\Re \eta}{\Re \eta^2} \tan \theta
= \frac {3 + \sqrt 5}2 \tan \theta
$$
which together with \eqref{PQRS:eqn} gives 
$$\alpha \geq \psi.$$
This implies in particular that the blue segment $RS$ has length $\geq \pi$. If we change the chord $\alpha_1$ by sliding $R$ so that $\psi$ rises to $\psi' = \alpha$, the new chord $\alpha_1'$ is longer than $\alpha_1$ (this simple fact was already noticed in the previous lemmas). However, the new multichord $\alpha' = \alpha_1'*\alpha_2*\alpha_3$ connects two andipodal points in $S^3$ and is therefore long at least $\pi$. Hence the original length of $\alpha$ was also at least $\pi$.
\end{proof}

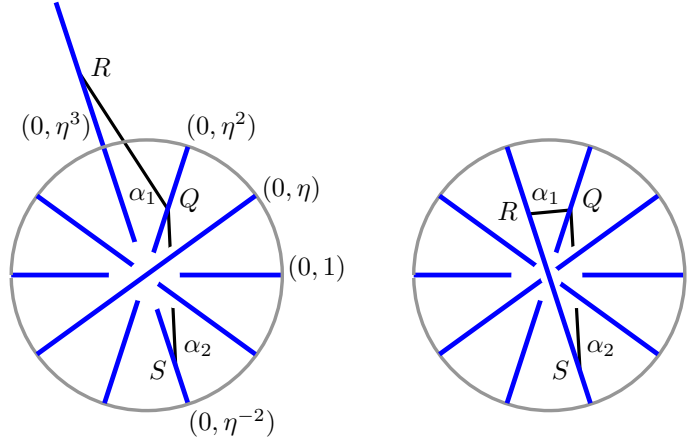
\begin{figure}
\centering
\labellist
\pinlabel $Q$ at 85 100
\pinlabel $R$ at 43 163
\pinlabel $S$ at 70 20
\pinlabel $(0,\eta^{-2})$ at 105 -5
\pinlabel $(0,1)$ at 145 68
\pinlabel $(0,\eta)$ at 133 105
\pinlabel $(0,\eta^2)$ at 100 135
\pinlabel $(0,\eta^3)$ at 22 135
\pinlabel $\alpha_1$ at 63 102
\pinlabel $\alpha_2$ at 89 30
\pinlabel $Q$ at 275 100
\pinlabel $R$ at 235 95
\pinlabel $S$ at 260 20
\pinlabel $\alpha_1$ at 253 102
\pinlabel $\alpha_2$ at 279 30
\small\hair 2pt
\endlabellist
\includegraphics[width= 9 cm]{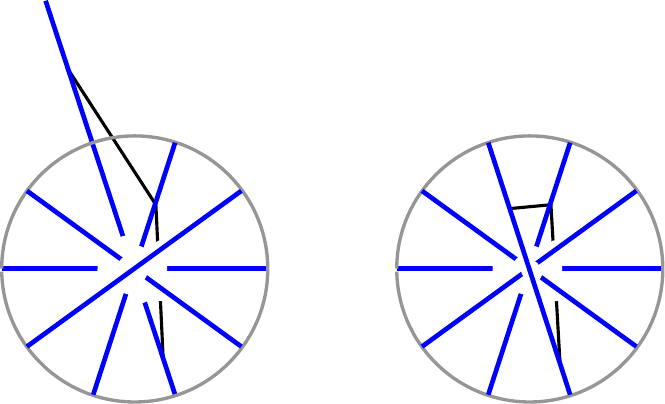}
\caption{The multichord $\alpha = \alpha_1 * \alpha_2$ in $S^3$.}
\label{chords4:fig}
\end{figure}

\begin{lemma} \label{hc:lemma}
Let $\alpha = \alpha_1 * \alpha_2$ be a geodesic multichord that connects two points that belong to the same blue line of $\tilde X$. Suppose $\alpha_1,\alpha_2$ are of one of the following types:
$$(h),(c), \quad (c),(a),  \quad (g), (b),  \quad (b), (d), \quad (h), (d), \quad (g), (a).$$ 
Then $\alpha$ has length $\geq \pi$. 
\end{lemma}
\begin{proof}
We consider the case $\alpha_1,\alpha_2$ is of type $(c), (a)$, and the case $(h), (c)$ will be analogous. By transporting to $S^3$ we get the configuration shown in Figure \ref{chords4:fig}, which is similar to the one of Figure \ref{chords3:fig}. 

There are two cases to consider. If $RS$ has length $\geq \pi$, we conclude as in the proof of the previous lemma: we can slide $R$ until it reaches $R'$ that is antipodal to $S$, this move decreases the length of $\alpha_1$, and the resulting $\alpha' = \alpha_1' * \alpha_2$ has length $\geq \pi$ since it connects two antipodal points. Hence $\alpha$ has also length $\geq \pi$. 
If $RS$ has length $< \pi$, the angle at $Q$ of the triangle $QRS$ is $< \pi$, a contradiction since $\alpha$ is geodesic. 

The case $(b)$, $(d)$ is in Figure \ref{gb:fig}-(left), and $(g), (b)$ will be analogous. We have
$$P = (\eta^{-2}\cos \theta, \eta^6 \sin \theta), \quad Q = (\cos \varphi, -\sin \varphi), \quad
R = (\eta^3\cos \psi, \eta \sin \psi).$$
Note that $P$ and $R$ lie in opposite blue edges: the one containing $P$ lies below the one shown in the figure. The multichord $\alpha$ is isotopic (with fixed endpoints) to the blue geodesic $PR$ that contains $(0,\eta^6)$. As above, there are two cases to consider: if $PR$ has length $ \geq \pi$, we can slide $R$ by increasing $\psi$ until $R=-P$; this move decreases the length of $\alpha_2$ and at the end we find a curve with length $\geq \pi$, so we are done. If $PR$ has length $< \pi$, it makes an angle $<\pi$ at $Q$, a contradiction.

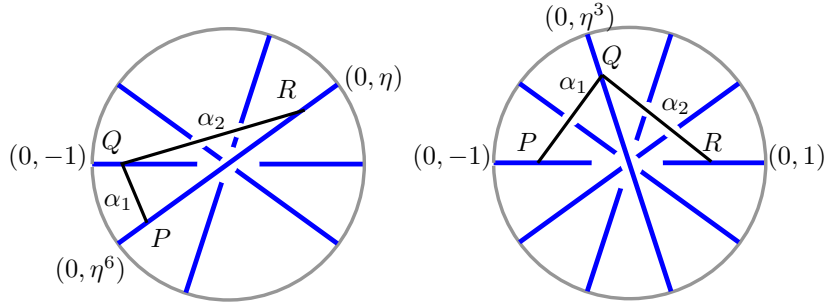
\begin{figure}
\centering
\labellist
\pinlabel $P$ at 33 30
\pinlabel $R$ at 92 100
\pinlabel $Q$ at 10 75
\pinlabel $(0,-1)$ at -20 68
\pinlabel $(0,\eta)$ at 133 105
\pinlabel $(0,\eta^6)$ at 0 15
\pinlabel $\alpha_1$ at 13 47
\pinlabel $\alpha_2$ at 57 85
\pinlabel $Q$ at 245 115
\pinlabel $R$ at 292 75
\pinlabel $P$ at 205 75
\pinlabel $(0,-1)$ at 170 68
\pinlabel $(0,1)$ at 332 68
\pinlabel $(0,\eta^3)$ at 230 134
\pinlabel $\alpha_1$ at 227 102
\pinlabel $\alpha_2$ at 275 92
\small\hair 2pt
\endlabellist
\includegraphics[width= 9 cm]{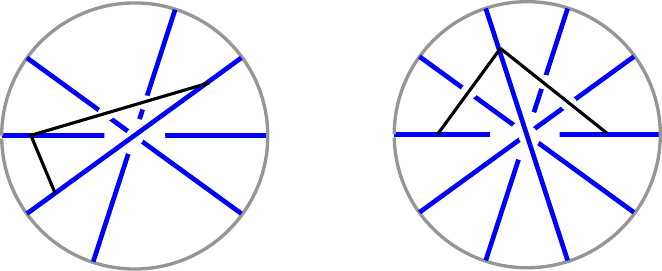}
\caption{The multichord $\alpha = \alpha_1 * \alpha_2$ in $S^3$.}
\label{gb:fig}
\end{figure}

The case $(h), (d)$ is in Figure \ref{gb:fig}-(right), and $(g), (a)$ is analogous. We have
$$P = (\cos \theta, - \sin \theta), \quad Q = (\eta^4\cos \varphi, \eta^3\sin \varphi), \quad
R = (\cos \psi, \sin \psi).$$
The multicurve $\alpha$ is homotopic (with fixed endpoints) to the geodesic arc $PR$ that contains $(-1,0)$. If we slide $\alpha_1$ and $\alpha_2$ by sending $P$ and $R$ respectively to $(0,-1)$ and $(0,1)$ we decrease the lengths of both chords, and we end up with a multichord that connects the two antipodal points $(0, \pm 1)$, that must have length $\geq \pi$. 
\end{proof}

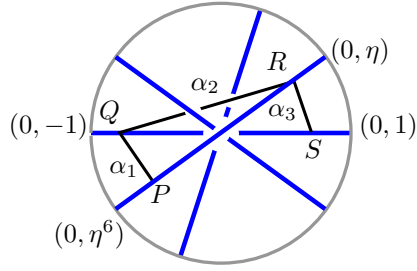
\begin{figure}
\centering
\labellist
\pinlabel $P$ at 35 35
\pinlabel $R$ at 92 100
\pinlabel $Q$ at 10 75
\pinlabel $S$ at 110 57
\pinlabel $(0,1)$ at 147 68
\pinlabel $(0,-1)$ at -20 68
\pinlabel $(0,\eta)$ at 133 105
\pinlabel $(0,\eta^6)$ at 0 15
\pinlabel $\alpha_1$ at 17 47
\pinlabel $\alpha_2$ at 57 88
\pinlabel $\alpha_3$ at 95 75
\small\hair 2pt
\endlabellist
\includegraphics[width= 3.5 cm]{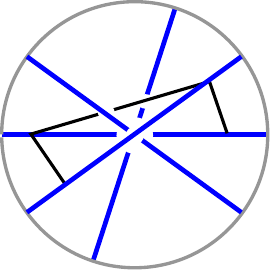}
\caption{The multichord $\alpha = \alpha_1 * \alpha_2 * \alpha_3$ in $S^3$.}
\label{bfb:fig}
\end{figure}

We end this section by providing further estimates in a few cases.

\begin{lemma} \label{bfb:lemma}
Let $\alpha = \alpha_1*\alpha_2*\alpha_3$ be a geodesic multichord where $\alpha_1, \alpha_2, \alpha_3$ are of type $(b), (f), (b)$ and two consecutive chords induce opposite orientations on the edge containing their common endpoints. The multichord $\alpha$ has length $\geq 6\pi/5$.
\end{lemma}
\begin{proof}
We transport $\alpha$ to $S^3$ as in Figure \ref{bfb:fig}. We have
\begin{gather*}
P = (\eta^{-2} \cos \theta, \eta^6 \sin \theta), \quad Q = (\cos \varphi, - \sin \varphi), \\
R = (\eta^3 \cos \psi, \eta \sin \psi), \quad S = (-\cos \alpha, \sin \alpha).
\end{gather*}
If $\varphi \geq \alpha$, we can slide $Q$ by decreasing $\varphi$ until we get $\varphi = \alpha$. This operation decreases the length of $\alpha_2$ and since we end up with two antipodal points $Q'=-S$ we get that $\alpha_2*\alpha_3$ has length $\geq \pi$. Since $\alpha_1$ has length $\geq \pi/5$, we are done.
So we suppose $\varphi < \alpha$. Since $\alpha_2$ is of type $(f)$ the inequalities \eqref{sym:eqn} holds, so in particular
\begin{equation} \label{>>:eqn}
0>\Re\eta^{2}\cos \varphi \sin \psi - \Re \eta \sin \varphi \cos \psi > \Re \eta^2 \cos \alpha \sin \psi  - \Re \eta \sin \alpha \cos \psi.
\end{equation}
We have
\begin{align*}
d(Q,R) & = \arccos(-\Re \eta^2 \cos \varphi \cos \psi - \Re \eta \sin \varphi \sin \psi), \\
d(R,S) & = \arccos (\Re \eta^2 \cos \psi \cos \alpha + \Re \eta \sin \psi \sin \alpha) 
\end{align*}
We write $f(\theta, \varphi, \psi, \alpha) = d(P,Q) + d(Q,R) + d(R,S)$. Since $\alpha$ is a geodesic multichord we must have
$$
0 = \frac{\partial f}{\partial \psi} = -\frac{\Re\eta^{2}\cos \varphi \sin \psi - \Re \eta \sin \varphi \cos \psi}{\sin d(Q,R)}
- \frac{-\Re \eta^2 \sin \psi \cos \alpha + \Re \eta \cos \psi \sin \alpha}{\sin d(R,S)}.
$$
If $\sin d(Q,R) \leq \sin d(R,S)$, given that $d(Q,R)\geq \pi/2$, we deduce that $d(Q,R)+d(R,S) \geq \pi$ and hence we conclude. So we suppose $\sin d(Q,R) > \sin d(R,S)$. Combining this with \eqref{>>:eqn} we deduce that $\partial f/\partial \psi < 0$, a contradiction.
\end{proof}

We will need a formula to calculate the distance between a point and a line (that is, a closed geodesic) in $S^3$.

\begin{lemma} \label{dell:lemma}
Let $Q = (z,w) \in S^3$ and $\ell = \{(\cos \theta z_0, \sin \theta w_0)\}$, $z_0,w_0 \in S^1$. Then
$$d(Q, \ell) = \arccos \sqrt{\Re^2 (zz_0^{-1}) + \Re^2(ww_0^{-1})}$$
\end{lemma}
\begin{proof}
After multiplying by $(z_0^{-1}, w_0^{-1})$ we may suppose $(z_0,w_0)=(1,1)$. Then $\cos (d(Q,\ell))$ is the projection in $\matR^4$ of $Q$ onto the vector plane containing $\ell$.
\end{proof}

\begin{figure}
\centering
\labellist
\pinlabel $P$ at 60 30
\pinlabel $R$ at 92 100
\pinlabel $Q$ at 10 75
\pinlabel $(0,-1)$ at -20 68
\pinlabel $(0,\eta)$ at 133 105
\pinlabel $(0,\eta^7)$ at 35 -5
\pinlabel $\alpha_1$ at 20 47
\pinlabel $\alpha_2$ at 57 88
\pinlabel $R$ at 307 100
\pinlabel $Q$ at 225 75
\pinlabel $P$ at 285 110
\pinlabel $(0,-1)$ at 190 68
\pinlabel $(0,\eta)$ at 348 105
\pinlabel $(0,\eta^7)$ at 245 -5
\pinlabel $\alpha_1$ at 263 97
\pinlabel $\alpha_2$ at 275 88
\small\hair 2pt
\endlabellist
\includegraphics[width= 9 cm]{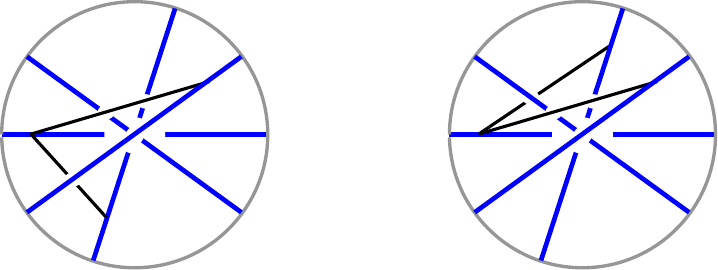}
\caption{The multichord $\alpha = \alpha_1 * \alpha_2$ in $S^3$.}
\label{bd:fig}
\end{figure}
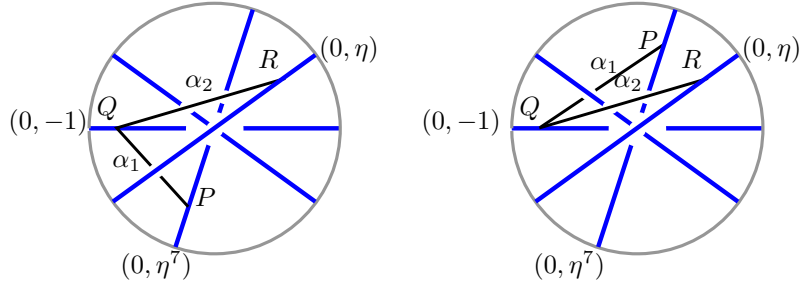

\begin{lemma} \label{bd:lemma}
Let $\alpha = \alpha_1 * \alpha_2$ be a geodesic multichord, where the chords $\alpha_1, \alpha_2$ are of one of the following types: 
$$(b), (d), \quad (g), (b), \quad (c), (d), \quad (g), (c).$$ 
Suppose that $\alpha_1, \alpha_2$ induce opposite orientations on the blue edge $e$ containing their common endpoint. Then $\alpha$ has length $\geq 4\pi/5$.
\end{lemma}
\begin{proof}
We first consider the case $(b), (d)$, and $(g), (b)$ will be analogous. The multichord is transported in $S^3$ as in Figure \ref{bd:fig}-(left). We have
$$P =(\eta \cos \theta, \eta^7 \sin \theta), \quad Q= (\cos \varphi, - \sin \varphi), \quad
R= (\eta^3 \cos \psi, \eta \sin \psi).$$
We prove that any multichord as in Figure \ref{bd:fig} as length $\geq 4\pi/5$.
Since by sliding $R$ towards the center we decrease the length of $\alpha_2$, we may suppose that $R = (\eta^3,0)$. We may also suppose that $P$ is the nearest point to $Q$ in the blue line $\ell$ containing $P$. Therefore the multichord has length 
\begin{align*}
f(\varphi) & = d(Q, \ell) + d(Q,R) \\
 & =  \arccos \sqrt{\Re^2(\cos \varphi\eta^{-1}) + \Re^2(-\sin \varphi \eta^3)}+
\arccos(\Re \eta^3 \cos \varphi) \\
& = \arccos \sqrt{\frac{3-\sqrt 5}8 + \frac{\sqrt 5}4 \cos^2 \varphi} +
\arccos \left(-\frac{\sqrt 5 - 1}4 \cos \varphi \right)
\end{align*}
We have $f(0) = 4\pi/5$ and $f(\pi/2) = 9\pi/10$. We consider the function
$$g(t) = \arccos \sqrt{\frac{3-\sqrt 5}8 + \frac{\sqrt 5}4 t^2} +
\arccos \left(-\frac{\sqrt 5 - 1}4 t \right)$$
with domain $t \in [0,1]$. We have $g(0) = 9\pi/10$ and $g(1) = 4\pi/5$. One checks that $g'(0) > 0$ and $g'(1) < 0$, and $g'(t)=0$ for at most two points $t \in (0,1)$, hence we must have $g(t) \geq g(1)$ for all $i\in [0,1]$, and the proof is complete.

We now consider the case $(c), (d)$, and $(g), (c)$ will be analogous. The multichord is transported in $S^3$ as in Figure \ref{bd:fig}-(right). We have
$$P =(\eta \cos \theta, \eta^2 \sin \theta), \quad Q= (\cos \varphi, - \sin \varphi), \quad
R= (\eta^3 \cos \psi, \eta \sin \psi).$$

The argument here is simpler: we slide $P$ and $R$ until $P=(\eta,0)$ and $R=(\eta^3,0)$, and this decreases the lengths of both chords. Now a simple computation shows that the shortest possible length is attained when $Q=(1,0)$, and we get $4\pi/5$.
\end{proof}

\subsection{Proof of Theorem \ref{CAT1:teo}}
By \cite[Theorem 5.4]{BH} the space $\tilde X$ is $\CAT(1)$ if and only if the triangulation $\Delta$ satisfies the link condition and $\tilde X$ contains no closed geodesic of length $<2\pi$. The link condition was proved in Proposition \ref{link:prop}, so we are left to show that there are no closed geodesic $\gamma \subset \tilde X$ of length $< 2 \pi$.

\subsubsection{Argument by contradiction} \label{contradiction:subsubsection}
Suppose one such $\gamma$ exist, and let $\gamma$ be the shortest closed geodesic in $\tilde X$. Following Bowditch \cite{Bow}, we know that $\gamma$ cannot be homotoped to a shorter closed curve passing through curves of length $< 2\pi$. (The arguments in \cite{Bow} apply only to compact spaces, but $\tilde X$ is periodic and we can apply them to the quotient of $\tilde X$ by a sufficiently large lattice.)

We know from Proposition \ref{gm:prop} that $\gamma$ is a closed geodesic multichord
$\gamma = \alpha_1 * \cdots * \alpha_k$. We suppose that $k$ is maximal along all shortest closed geodesics in $\tilde X$. By Lemma \ref{abcdefgh:lemma}, each $\alpha_i$ has either length $< \pi$, and is hence of some type $(a), \ldots, (h)$, or it has length $\pi$ but can be homotoped (without increasing lengths) to a concatenation of shorter geodesics. Since $k$ is maximal, the second case does not occur (modifying $\alpha_i$ with the homotopy would increase $k$; the resulting closed multichord is still geodesic, because if it were not it could be homotoped to a shorter closed curve, contradicting minimality of $\gamma$).

Therefore each $\alpha_1, \ldots, \alpha_k$ is a chord of length $< \pi$. Let $P_i$ be the common endpoint of $\alpha_i$ and $\alpha_{i+1}$ (indices will always be considered cyclically). We choose a blue edge $e_i$ containing each $P_i$ (the choice is of course forced if $P_i$ lies in the interior of $e_i$), thus determining a type for all $\alpha_i$. We decide to choose the edges $e_i$ in any way that maximizes the number of chords that are of type $(b)$. 

Since the types are all determined, each $\alpha_i$ enters $e_i$ in some well-defined sector $s_i$, and $\alpha_{i+1}$ leaves the same $e_i$ in some sector $s_i'$. By Lemma \ref{sectors:lemma} the two sectors $s_i, s_i'$ are neither coincident, nor adjacent, except in the peculiar case where $\alpha_i, \alpha_{i+1}$ are both red or green segments, that we now exclude. In this case the types of $\alpha_i, \alpha_{i+1}$ are $(b)$, $(c)$ (or conversely). We then decide to change the edge $e_{i+1}$, so that the type of $\alpha_{i+1}$ transforms from $(c)$ to $(b)$ and the new $s_i'$ is not adjacent anymore to $s_i$. This will change also the type of $\alpha_{i+2}$. If we get that $s_{i+1}$ and $s_{i+1}'$ are also adjacent, we iterate this argument: either we end up after finitely many steps, or we come back to the initial $i$, and in that very peculiar case $\gamma$ is itself made of red or green segments only. By analysing the red and green graphs we see that it does not contain any loop with less than 10 segments, so $k\geq 10$, and hence $\gamma$ would have length $k\pi/5 \geq 2\pi$ in that case, a contradiction.

Let $m_i$ be the midpoint of $e_i$. We may suppose after an isometry of $\tilde X$ that $m_1=0$. Consider the vector $w_i = m_i - m_{i-1}$ based at $m_{i-1}$. We have $w_1 + \cdots + w_k=0$.
Each $w_i$ is transformed by some isometry of $\tilde X$ into the vector in \eqref{vectors:eqn} based at $0$ that corresponds to the type of $\alpha_i$. 

We should think of the sequence $w_1,\ldots, w_k$ as a discretized description of the geodesic multichord $\gamma = \alpha_1* \cdots * \alpha_k$. We now prove our theorem by showing that there are no possible discretized versions of $\gamma$. So we now define a discretized version of our problem.

\subsubsection{Strings}
Let a \emph{string} be a sequence $w_1, \ldots, w_k \in \matZ^3$ of vectors such that the following holds. We set $m_1=0$ and $m_{i} = w_1+ \cdots + w_{i-1}$. Each $w_i$, considered as a vector based at $m_i$, is obtained from one vector in \eqref{vectors:eqn} based at $0$ via some isometry of $\tilde X$. Recall that the group $G$ of isometries of $\tilde X$ is the group of Euclidean isometries of the configuration of blue lines, so these are affine isometries.

Each $m_i$ is the midpoint of some blue edge $e_i$ in $\tilde X$, and $w_i$ is a Euclidean chord from $m_i$ to $m_{i+1}$, in the sense that it is a Euclidean segment which will not be geodesic in general with respect to the metric of $\tilde X$.
A string $w_1,\ldots, w_k$ determines a Euclidean multichord connecting $m_1$ to $m_{k+1}$.

Each Euclidean chord $w_i$ from $m_i$ to $m_{i+1}$ is obtained from one in \eqref{vectors:eqn} and as such it has a type in $(a), \ldots, (h)$, it exits from $e_i$ through an initial sector $s_i'$ and enters $e_{i+1}$ in a final sector $s_{i+1}$. It also induces an orientation on $e_i$ and $e_{i+1}$ as prescribed by its type following Figure \ref{chords_orient_e:fig}. We set the \emph{length} $l(w_i)$ of $w_i$ to be the minimum length of a chord having the same type of $w_i$, as determined in Corollary \ref{abcdefgh:cor}. Hence $l(w_i) \in \{\pi/5, 2\pi/5, \pi/2, 3\pi/5\}$.

The string $w_1, \ldots, w_k$ is \emph{admissible} if the following requirements are fulfilled (indices are considered cyclically modulo $k$):
\begin{itemize}
\item $w_1 + \cdots + w_k = 0$;
\item $l(w_1) + \cdots + l(w_k) < 2 \pi$;
\item The sectors $s_i, s_i'$ are neither coincident nor adjacent, for all $i$;
\item If neither $w_i$ nor $w_{i+1}$ are of type $(b)$, they induce opposite orientations on the edge $e_{i+1}$;
\item If $w_i$ is of type $(b)$, and neither $w_{i-1}$ nor $w_{i+1}$ are of type $(b)$, then they induce opposite orientations on either $e_i$ or $e_{i+1}$.
\end{itemize}

An admissible string $w_1, \ldots, w_k$ has a \emph{bonus length} $B$ that is calculated as follows. We start with $B=0$, and then iteratively:
\begin{enumerate}
\item If $(w_{i-1}, w_i, w_{i+1})$ are of type $(b,b,c)$ or $(c,b,b)$ and $m_{i-1}, m_{i+2}$ belong to the same blue line in $\tilde X$, we add $2\pi/5$ to $B$;
\item If $(w_{i-1}, w_i, w_{i+1})$ are of type $(b,f,b)$ and $e_i, e_{i+1}$ gets opposite orientations from the adjacent vectors, we add $\pi/5$ to $B$;
\item If $w_i, w_{i+1}$ are of type $(h, c), (c,a), (g,b),  (b,d), (h,d)$, or $(g,a),$ and the points $m_{i-1}, m_{i+1}$ belong to the same blue line in $\tilde X$, we add to $B$ respectively the value $2\pi/5, 2\pi/5, 3\pi/10, 3\pi/10, \pi/10, \pi/10$;
\item If $w_i, w_{i+1}$ are of type $(b,c)$ or $(c,b)$, they induce opposite directions on $e_{i+1}$, and the blue lines containing $m_i$ and $m_{i+2}$ are not parallel, we add $\pi/10$ to $B$;
\item If $w_i, w_{i+1}$ are of type $(b, d),  (g, b),  (c, d)$, or $(g, c)$ and induce opposite directions on $e_{i+1}$, we add $\pi/10$ to $B$.
\end{enumerate}
It is important here that each $w_i$ should contribute to at most one bonus.
Therefore the algorithm for calculating the bonus $B$ goes as follows: for each point (1-5) we consider $i=1, \ldots, k$, and if we find a triple $w_{i-1}, w_i, w_{i+1}$ or a pair $w_i, w_{i+1}$ that fulfills the requirement, we add the stated amount to $B$, and then we remove the vertices $w_{i-1}, w_i, w_{i+1}$ or $w_i, w_{i+1}$ from the list so that they will not be used anymore.

The \emph{length} of an admissible string $w_1,\ldots, w_k$ is $l(w_1) + \cdots + l(w_k) + B$, the sum of the lengths of its chords plus the bonus. Using a code available in \cite{code} we can prove the following.

\begin{lemma} \label{code:lemma}
There is no admissible string of length $<2 \pi$. 
\end{lemma}

In Section \ref{contradiction:subsubsection} we have proved that if $\tilde X$ contains a closed geodesic of length $< 2\pi$, then it contains a closed geodesic multicurve $\gamma = \alpha_1 * \cdots \alpha_k$ where each $\alpha_i$ is a geodesic chord of length $< \pi$ of some type $(a), \ldots, (h)$, and the sectors $s_i, s_i'$ at the edge $e_i$ are neither coincident nor adjacent. We have defined its discretized version by taking the midpoint $m_i$ of the edge $e_i$ and defining $w_i = m_i-m_{i-1}$ based at $m_i$. By construction the sequence $w_1,\ldots, w_k$ is a string.

Lemmas \ref{nob:lemma} and \ref{xby:lemma} imply that the string $w_1,\ldots, w_k$ is admissible. To show this, note that the second case in the conclusion of both lemmas is excluded, by the maximality of the chords $\alpha_i$ of type $(b)$: if $\alpha_i, \alpha_{i+1}$ are both red or green edges, their types are either $(b,b)$, $(b,c)$, $(c,b)$, or $(c,c)$, but the latter case is excluded because by changing the edge $e_{i+1}$ we would transform it into $(b,b)$; analogously, if $\alpha_{i-1},\alpha_i,\alpha_{i+1}$ are all red or green edges, their types cannot be $(c,b,c)$, otherwise by changing $e_i$ and $e_{i+1}$ they would transform into $(b,c,b)$.

Lemmas \ref{bc:lemma}, \ref{bbc:lemma}, \ref{hc:lemma}, \ref{bfb:lemma}, and \ref{bd:lemma} imply that the length of $\alpha$ is at least the length of the string $w_1,\ldots, w_k$. By Lemma \ref{code:lemma} no admissible string may have length $< 2\pi$, and hence we get a contradiction. This completes the proof of Theorem \ref{CAT1:teo}.

\end{document}